\documentclass[12pt]{amsart}%
\usepackage{amsmath}%
\usepackage{amssymb}%
\usepackage{amscd}%
\usepackage{color}%
\usepackage{enumitem}%
\usepackage{pifont}%
\usepackage{ifsym}%

\usepackage{graphicx}%
\usepackage{mathrsfs}%
\usepackage[all]{xy}%
\usepackage{stmaryrd}
\usepackage{newtxtext,newtxmath}%
\usepackage[all, warning]{onlyamsmath}
\usepackage{amsrefs}
\usepackage{bbm}
\usepackage[capitalise]{cleveref}
\usepackage{stmaryrd}
\usepackage{trimclip}
\makeatletter
\DeclareRobustCommand{\arr}{%
 \mathrel{\mathpalette\short@to\relax}%
}
\newcommand{\short@to}[2]{%
  \mkern2mu
  \clipbox{{.3\width} 0 0 0}{$\m@th#1\vphantom{+}{\shortrightarrow}$}%
  }
\makeatother

\def\ol#1{\overline{#1}}
\def\wh#1{\widehat{#1}}
\def\wt#1{\widetilde{#1}}
\theoremstyle{plain}
    \newtheorem{thmint}{Theorem}
    \newtheorem{theorem}{Theorem}[section]
    \newtheorem{proposition}[theorem]{Proposition}
    \newtheorem{lemma}[theorem]{Lemma}
    \newtheorem{corollary}[theorem]{Corollary}

    \newtheorem*{question*}{Question}
\theoremstyle{definition}
    \newtheorem{definition}[theorem]{Definition}
    
    \newtheorem{example}[theorem]{Example}
    \newtheorem{remark}[theorem]{Remark}
      
\def\Alphabet{A,B,C,D,E,F,G,H,I,J,K,L,M,N,O,P,Q,R,S,T,U,V,W,X,Y,Z}
\def\grabet{a,b,c,d,e,f,g,h,i,j,k,l,m,n,o,p,q,r,s,t,u,v,w,x,y,z}
\def\endpiece{xxx}
\def\makeAlphabet[#1]{\expandafter\makeA#1,xxx,}
\def\makealphabet[#1]{\expandafter\makea#1,xxx,}
\def\makeA#1,{\def\temp{#1}\ifx\temp\endpiece\else%
\mkbb{#1}\mkfrak{#1}\mkbf{#1}\mkcal{#1}\mkscr{#1}\mkbs{#1}\expandafter\makeA\fi}%
\def\makea#1,{\def\temp{#1}\ifx\temp\endpiece\else\mkfrak{#1}\mkbf{#1}\mkbs{#1}\expandafter\makea\fi}%
\def\mkbb#1{\expandafter\def\csname bb#1\endcsname{\mathbb{#1}}}
\def\mkfrak#1{\expandafter\def\csname fr#1\endcsname{\mathfrak{#1}}}
\def\mkbf#1{\expandafter\def\csname b#1\endcsname{\mathbf{#1}}}
\def\mkcal#1{\expandafter\def\csname c#1\endcsname{\mathcal{#1}}}
\def\mkscr#1{\expandafter\def\csname s#1\endcsname{\mathscr{#1}}}
\def\mkbs#1{\expandafter\def\csname bs#1\endcsname{{\boldsymbol{#1}}}}
\def\makeop[#1]{\xmakeop#1,xxx,}
\def\mkop#1{\expandafter\def\csname #1\endcsname{{\mathrm{#1}}}} %
\def\xmakeop#1,{\def\temp{#1}\ifx\temp\endpiece\else\mkop{#1}\expandafter\xmakeop\fi}%
\def\makeup[#1]{\xmakeup#1,xxx,}
\def\mkup#1{\expandafter\def\csname #1\endcsname{{\mathrm{#1}\,}}} %
\def\xmakeup#1,{\def\temp{#1}\ifx\temp\endpiece\else\mkup{#1}\expandafter\xmakeup\fi}%
\makeAlphabet[\Alphabet]
\makealphabet[\grabet]
\makeop[Map,Hom,alt,id,loc,col,unif,len,Supp,Consv,pr,Aut,Ext,ab,tors,Coind,univ,pro,av,sp,rad,ex,Span]
\makeup[Im,Ker,Coker]

\def\diam#1{\operatorname{diam}(#1)}
\def\La{\Lambda}
\def\nabe{\nabla_{\!e}}
\def\FC{\sF\kern-0.5mm\sC}

\def\R{\bbR}
\def\Z{\bbZ}
\def\E{\bbE}
\def\pair#1{\langle#1\rangle}

\def\abs#1{|#1|}
\def\dabs#1{\Vert#1\Vert}
\def\Dabs#1{\bigl|\kern-0.3mm\bigl|#1\bigr|\kern-0.3mm\bigr|}

\def\vp{\varphi}

\def\N{\bbN}
\def\bsxi{\boldsymbol{\xi}}
\def\C{\Consv^\phi}
\def\G{G}
\def\e{\eta}
\def\Fr{\sN}

\begin{document}

\setcounter{tocdepth}{1}
\newpage
\title[Varadhan's Decomposition]{Varadhan's Decomposition of Shift-Invariant Closed $L^2$-forms for Large Scale Interacting Systems on Euclidean Lattices}
\author[Bannai]{Kenichi Bannai}\email{bannai@math.keio.ac.jp}
\author[Sasada]{Makiko Sasada}\email{sasada@ms.u-tokyo.ac.jp}
\thanks{This work was supported in part by JST CREST Grant Number JPMJCR1913, KAKENHI 18H05233, 24K21515 and the UTokyo Global Activity Support Program for Young Researchers.}
\address[Bannai]{Department of Mathematics, Faculty of Science and Technology, Keio University, 3-14-1 Hiyoshi, Kouhoku-ku, Yokohama 223-8522, Japan.}
\address[Sasada]{Department of Mathematics, University of Tokyo, 3-8-1 Komaba, Meguro-ku, Tokyo 606-8502, Japan.}
\address[Bannai, Sasada]{Mathematical Science Team, RIKEN Center for Advanced Intelligence Project (AIP),1-4-1 Nihonbashi, Chuo-ku, Tokyo 103-0027, Japan.}

\date{\today}

\begin{abstract}
	We rigorously formulate and prove for a relatively general 
	class of interactions 
	\emph{Varadhan's Decomposition of shift-invariant closed $L^2$-forms}
	for a large scale interacting system on the Euclidean lattice
	with finite range. 
	Such decomposition
	of closed forms has played an essential role in proving the diffusive scaling
	limit of nongradient systems. 
	A general expression in terms of conserved quantities was sought from 
	observations for specific models, but a precise formulation or rigorous proof up until now had been elusive. 
	Our result is based on a general decomposition theorem
	of shift-invariant closed \emph{uniform} forms studied in our previous article \cite{BKS20}.
	In the present article, we show that the same universal structure also appears for $L^2$-forms.
	The essential assumptions are: (i) the set of states on each vertex is a finite set, (ii) the measure on the
	configuration space 
	is the product measure, and (iii) there is a certain uniform spectral gap estimate for the mean field version of the interaction. 
	As a special case, our result gives Varadhan's decomposition for the case of the multi-species exclusion process, 
which is a new result that could not be proved by existing methods.
Our result also gives complete proofs of Varadhan's decompositions
for finite range interactions, whose detailed proofs have been missing in the literature --
presumably because there exists an obstruction to the standard method of proof.

\end{abstract}

\maketitle

\tableofcontents

%
%
%
\section{Introduction}\label{sec: intro}
%
%
%

In the last three decades, the theory of the hydrodynamic behavior of interacting particle systems has been well developed, and a rigorous understanding of the various models has been achieved. 
However, we still do not have a sufficiently conceptual understanding of the diffusive scaling limit for nongradient models. In particular, the origin of specific expressions in the 
variational formula for diffusion coefficients in terms of conserved quantities was a mystery. 
There are not many results for nongradient models with more than one conserved quantity, and it was not certain how common such formula holds for models with multiple conserved quantities. 

In our previous article \cite{BKS20}, we discovered that there is indeed a universal conceptual
structure for a certain general class of models which leads to a general
\emph{decomposition of shift-invariant closed uniform forms} in terms of a group action on the 
configuration space and conserved quantities. 
If this decomposition 
extends to $L^2$-forms, then this should also appear in the variational formula for the diffusion coefficient, and it can be said that the extension to $L^2$-form allows us to capture this universal structure
 of nongradient models. 

The purpose of this article is to formulate and prove for relatively general interactions
having a uniformly bounded spectral gap
\emph{Varadhan's Decomposition of shift-invariant closed $L^2$-forms}
for a large scale interacting system on an Euclidean lattice with finite range, 
when the set of states on each vertex is a finite set.
As a special case, our result gives Varadhan's decomposition for the case of the multi-species exclusion process, 
which is a new result that could not be proved by existing methods. We emphasize that the number of conserved quantities of the multi-species exclusion process can be arbitrary large. For such case, the exact algebraic computation used in the existing methods to obtain the explicit expression of the decomposition is impractical.
Our result also gives a complete proof of Varadhan's decomposition
for finite range (in particular not necessarily nearest neighbour) interactions, whose detailed proof has also been missing in the literature --
presumably because there exists an obstruction to the standard method of proof. The obstruction is discussed in detail in Remark \ref{rem:finite-range} precisely. 

In the following, we summarize our setting and main result without giving complete definitions. The precise definitions are given in later sections.

We let $S$ be a finite nonempty set,
the typical example being $S=\{0,\ldots,\kappa\}$ for some integer $\kappa>0$ with base point $*=0$,
and we let 
\[
	S^{\bbZ^d}\coloneqq\prod_{x\in \bbZ^d}S
\] 
be the configuration space on $\Z^d$.
We define an \emph{interaction} to be a map $\phi\colon S\times S\rightarrow S\times S$
satisfying a certain symmetry condition which expresses the interaction between states
on adjacent vertices. 
Models such as the \emph{multi-species exclusion process} and the \emph{generalized exclusion process}
may be described in a unified manner using the interaction.
Let $E\subset \Z^d\times\Z^d$ be a set of directed edges, such that the pair $(\Z^d,E)$ defines a 
\emph{locale}, in other words, a connected
locally finite simple symmetric directed graph.
For 
\[
	\E^d\coloneqq\{ x,y\in\Z^d\mid \sum_{j=1}^d\abs{x_j-y_j}=1\},
\]
where $x_j$ and $y_j$ are the $j$-th components of $x=(x_j),y=(y_j)\in\Z^d$,
the graph $(\Z^d,\E^d)$ is simply the usual nearest neighbor Euclidean lattice.
The set $E$ stipulates the pairs of vertices of $\Z^d$ where an interaction may occur.
For any configuration $\e=(\eta_x)\in S^{\Z^d}$ and directed edge $e=(x,y)\in E$,
we denote by $\e^e$ the configuration obtained by interacting
the states on the vertices of $e$.  In other words, $\e^e$ and $\e$
coincides outside the $x$ and $y$ components, and if we denote by $\eta^e_x$ and $\eta^e_y$
the $x$ and $y$ components of $\e^e$, then we have $(\eta^e_x,\eta^e_y)=\phi(\eta_x,\eta_y)$,
where $\eta_x$ and $\eta_y$ are the $x$ and $y$ components of $\e$.
We call such $\e^e$ \emph{a transition} of $\e$.

A \emph{local function} $f\colon S^{\Z^d}\rightarrow\R$ is any function on $S^{\Z^d}$
whose value depends on the states of the configuration at only a finite number of vertices,
and we denote by $C_\loc(S^{\bbZ^d})$ the $\R$-linear space of local functions on $S^{\bbZ^d}$.
For any local function $f$ and directed edge $e\in E$, we define the differential $\nabe f$ to be the 
local function given for any $\e\in S^{\Z^d}$ by
\[
	\nabe f(\e)\coloneqq f(\e^e)-f(\e).
\]
In \cite{BKS20}, we constructed the space of \emph{uniform functions} $C_\unif(S^{\bbZ^d})$,
a class of functions which gives a rigorous framework to consider 
certain infinite sums of normalized
local functions which appear in the formulation of the decomposition of closed forms.
This space allows for the rigorous consideration of infinite sums such as
\begin{equation}\label{eq: 1}
	\Gamma_f\coloneqq\sum_{x\in\Z^d}\tau_x(f),
\end{equation}
which in \cite{KL99}*{p.144}, one of the most pedagogical textbooks for the hydrodynamic limit, is stated ``does not make sense''.
Here, $f$ is any certainly normalized local function and
$\tau_x$ denotes the translation by $x=(x_1,\ldots,x_d)\in\Z^d$.
The differential $\nabe$ extends $\R$-linearly to a differential $\nabe\colon C_\unif(S^{\bbZ^d})\rightarrow C_\loc(S^{\bbZ^d})$
giving the differential $\partial\Gamma_f=(\nabe\Gamma_f)_{e\in E}$.

One key ingredient in the formulation of the decomposition of closed forms
are certain forms defined using the conserved quantities.
We define a \emph{conserved quantity} to be any certainly normalized function 
$\xi\colon S\rightarrow\R$ satisfying 
\[
	\xi(s'_1)+\xi(s'_2)=\xi(s_1)+\xi(s_2)
\]
for any $(s_1,s_2)\in S\times S$, where $(s_1',s_2')=\phi(s_1,s_2)$ (see \cref{def: CQ}).
Our key insight is that we may characterize conserved quantities simply as 
invariants preserved by interaction on adjacent vertices.
If we denote by $\Consv^\phi(S)$ the $\R$-linear space of conserved quantities,
then $\Consv^\phi(S)$ is a finite dimensional space since we have assumed that $S$ is finite,
and the number of linearly independent conserved quantities of the interaction
may be deduced from the dimension $c_\phi\coloneqq\Consv^\phi(S)$ of this space.

For each $x\in\Z^d$, any conserved quantity defines a local function $\xi_x(\e)\coloneqq\xi(\eta_x)$ for any
$\e=(\eta_x)\in S^{\Z^d}$. The infinite sum
\begin{equation}\label{eq: 2}
	\xi_{\Z^d}\coloneqq\sum_{x\in\Z^d}\xi_x
\end{equation}
defines a uniform function on $S^{\Z^d}$ which satisfies $\nabe\xi_{\Z^d}=0$
for any $e\in E$.
We say that two configurations $\e=(\eta_x)$ and $\e'=(\eta'_x)$ have the \emph{same conserved quantities}, 
if their components
differ only at a finite set of vertices $\La\subset\Z^d$, and for any conserved quantity $\xi$,
we have $\sum_{x\in\La}\xi_x(\eta)=\sum_{x\in\La}\xi_x(\eta')$.
A \emph{path} in a configuration space is a sequence of transitions, and we 
say that an interaction is \emph{irreducibly quantified}, if there exists a path, in the configuration
space between any two distinct configurations with the same conserved quantities.

A \emph{uniform form} is a system $(\omega_e)_{e\in E}$ of local functions $\omega_e\colon S^{\Z^d}\rightarrow \R$ for $e\in E$
satisfying certain conditions (see \cref{def: Uniform Form}). 
For a path in a configuration space, we 
define the integration $\int_{\vec\gamma}\omega$ of a form $\omega=\{\omega_e\}$ with respect to the path $\vec\gamma$
to be the sum of values $\omega_e(\e)$ over the transitions $(\e,\e^e)$ which appear in the path.
We say that a path is \emph{closed} if the configuration at the beginning and the ending of the path coincide,
and we say that a form $\omega$ is \emph{closed}, if the integral of $\omega$ vanishes for
any closed path in the configuration space.  

Next, we assume that the set of directed edges $E$ is preserved by translation
$\tau_x$ for any $x\in\Z^d$.  This is equivalent to $(\Z^d,E)$ having an action of the group $G=\Z^d$ 
by translation.
We call such $(\Z^d,E)$ a Euclidean lattice with \emph{finite range}, recalling that $(\Z^d, E)$ is assumed to be a locally finite graph.
We say that a form is \emph{shift-invariant}, 
if it is invariant with respect to the translation.
Our notion of shift-invariant closed $L^2$-forms for the case of the 
generalized exclusion process coincides with the notion of germs of closed 
forms given in \cite{KL99}*{\S Appendix 4}.
The main result of our article concerns the \emph{decomposition of shift-invariant closed $L^2$ forms},
which is a generalization of \cite{KL99}*{Theorem 4.14}.  

We first review the result of our previous article
\cite{BKS20}.  We fix a basis $\xi^{(1)},\ldots,\xi^{(c_\phi)}$ of $\Consv^\phi(S)$.
The main result of \cite{BKS20}, specialized to the case of $(\Z^d,E)$ is given as follows.
\begin{thmint}[\cite{BKS20}*{Theorem 5}]\label{thm: old}
	Let $(\Z^d, E)$ be a Euclidean lattice with finite range,
	which by definition has an action of the group $G=\Z^d$
	given by translation.
	Let $\phi$ be an irreducibly quantified interaction,
	and assume in addition that $\phi$ is simple if $d=1$.
	Then for any shift-invariant closed uniform form $\omega=(\omega_e)_{e\in E}\in\prod_{e\in E}C_{\loc}(S^{\Z^d})$,
	there exist unique constants $a_{ij}\in\R$ for $i=1,\ldots,c_\phi$ and $j=1,\ldots,d$,
	and a local function $f\in C_\loc(S^{\Z^d})$
	such that for any $e\in E$, we have
	\[
		\omega_e=\nabe\Bigl(\sum_{x\in\Z^d}\tau_x(f)+\sum_{i=1}^{c_\phi}
		\sum_{j=1}^d a_{ij} 
		\sum_{x\in\Z^d}x_j\xi^{(i)}_{x}\Bigr)
	\]
	in $C_{\loc}(S^{\Z^d})$, where $\xi^{(i)}_{x}$ for any $x=(x_1,\ldots,x_d)\in\Z^d$ is the local function defined 
	as $\xi^{(i)}_{x}(\e)\coloneqq\xi^{(i)}(\eta_x)$ for $\e=(\eta_x)\in S^{\Z^d}$.
\end{thmint}

Here, an interaction $\phi$ is \emph{simple} if and only if $c_\phi=1$, and
the monoid generated by $\xi(S)$ in $\R$ for any non-zero conserved quantity $\xi$ is isomorphism to $\N$ or $\Z$. 
In this article, we will use \cref{thm: old} to prove the decomposition for shift-invariant closed \emph{$L^2$-forms}. 

Now, we take a probability measure $\nu$ fully supported on $S$, i.e. $\nu(s)>0$ for any $s \in S$,  and 
consider the product measure $\mu\coloneqq\nu^{\otimes\Z^d}$ on $S^{\Z^d}$. Also, we fix a shift-invariant \emph{weight} 
$r=(r_e)_{e\in E}\in\prod_{e\in E}C_\loc(S^{\bbZ^d})$, which determines a weighted $L^2$-space, 
such that $r_e(\eta)>0$ for any $e\in E$ and $\eta \in S^{\bbZ^d}$. 
When applied to prove the diffusive scaling limit of a nongradient reversible process, $\mu$ is chosen to be a reversible measure for the process and the weight encodes the information of the frequency of the transition on each directed edge, hence
such data may correspond to the stochastic process given by the generator
\[
	Lf(\e)=\sum_{e\in  E}r_e(\e)\nabe f(\e)=\sum_{e\in  E}r_e(\e)(f(\e^e)-f(\e)).
\] 
However, for the main theorem of this article, we do not need the reversibility nor any relation between the weight $r=(r_e)$ and the probability measure $\mu$ which we consider in the main result.  
Thus we take the probability measure $\mu$ independently from the weight. Indeed, we do not need to consider any stochastic process in this article as the results only concern some forms in a weighted $L^2$-space.

We let $L^2(\mu)_{r_e}$ be the weighted $L^2$-space $L^2(r_e d\mu)$ of measurable functions on $S^{\Z^d}$
with a norm given by $\dabs{f}_{r_e}^2\coloneqq E_\mu[r_e f^2]$. Since $r_e$ is bounded by positive constants from above and below,
the space of local functions $C_\loc(S^{\Z^d})$ is a dense subspace of $L^2(\mu)_{r_e}$.
We rigorously extend the notion of closed forms to $L^2$-forms, which is precisely given in Section \ref{sec:co-local}.
Our main result is as follows.

\begin{thmint}[=\cref{thm: main}]\label{thm: intro}
	Let the assumptions be as in Theorem \ref{thm: old}.
	Furthermore, suppose $\nu$ has a uniformly bounded 
	spectral gap for the interaction $(S,\phi)$ in the sense of Definition \ref{def: SG}.
	Furthermore, we consider a weight $r=(r_e)_{e\in E}$ invariant with respect to the action of the group $G=\Z^d$.
	Then for any shift-invariant closed $L^2$-form $\omega=(\omega_e)_{e\in  E}\in\prod_{e\in  E}L^2(\mu)_{r_e}$,
	there exist unique constants $a_{ij}\in\R$ for $i=1,\ldots,c_\phi$ and $j=1,\ldots,d$,
	and a sequence of local functions $f_n\in C_\loc(S^{\Z^d})$,
	such that for any $e\in E$, we have
	\begin{equation}\label{eq: MT}
		\omega_e=\lim_{n\rightarrow\infty}
		\nabe\Bigl(\sum_{x\in\Z^d}\tau_x(f_n)+\sum_{i=1}^{c_\phi}
		\sum_{j=1}^d a_{ij} 
		\sum_{x\in\Z^d}x_j\xi^{(i)}_{x}\Bigr)
	\end{equation}
	in $L^2(\mu)_{r_e}$.
\end{thmint}

\cref{thm: intro} is valid for the following example
of the multi-species exclusion process,
which could not be proven by existing methods.

\begin{example}\label{example: 1}
	For an integer $\kappa>0$, let $S=\{0,1,\ldots,\kappa\}$ with $*=0$, the map
	given by
	\[
		\phi\colon S\times S\rightarrow S\times S,  \quad (s_1,s_2)\mapsto(s_2,s_1)
	\]
	is an interaction which is \emph{irreducibly quantified} (see \cite{BKS20}*{Proposition 2.27} for a proof). 
	This interaction gives the 
	\emph{multi-species exclusion process}.
	We have $c_\phi=\dim_\R\C(S)=\kappa$ in this case, 
	and a basis of the space of conserved quantities $\C(S)$ is given by
	$\xi^{(i)}\colon S\rightarrow\R$
	for $i=1,\ldots,\kappa$, such that
	$
		\xi^{(i)}(s)\coloneqq 1
	$
	if $i=s$ and $\xi^{(i)}(s)\coloneqq 0$ if $i\neq s$. 
	Let $\nu$ be any probability measure fully supported on $S$.
	It is well known that $\nu$ has a uniformly bounded 
	spectral gap for the interaction $(S,\phi)$ in the sense of Definition \ref{def: SG} (cf, \cite{DS81}, Section 4.1.2. of \cite{CLR10}, Theorem 3.4. of \cite{Ca08}), 
	hence \cref{thm: intro} holds in this case for any Euclidean lattice with finite range $(\Z^d, E)$ with $d > 1$ and any shift-invariant weight $r=(r_e)_{e \in E}$. 

	This result is applicable to study the diffusive scaling limits for the following process. Fix a Euclidean lattice with finite range $(\Z^d, E)$, $d > 1$ and a 
	symmetric matrix $(\lambda_{ij})\in M_{\kappa+1}(\R)$ such that $\lambda_{ij}>0$ for $i,j=0,1,\ldots,\kappa$. Consider a Markov process with generator 
	\[
		Lf(\e)=\sum_{e\in  E}\lambda_{\e_{o(e)},\e_{t(e)}}(f(\e^e)-f(\e)).
	\]
	Then, the process is reversible for product measures $\mu=\nu^{\otimes\Z^d}$ for any $\nu$ and choosing $r_e(\eta)=\lambda_{\e_{o(e)},\e_{t(e)}}$, \cref{thm: intro} provides the key decomposition for the non-gradient method for this process. 
	
	For the case where $\kappa=2$ and $\lambda_{ij}=\lambda\mathbf{1}_{ij =0}$ is called the two-color simple exclusion process and studied by Quastel  \cite{Qua92}. For this case, the weight is degenerate and so \cref{thm: intro}  can not be applied directly. On the other hand, since $\lambda_{10}=\lambda_{01}=\lambda_{20}=\lambda_{02}$, the color-blind process is the symmetric simple exclusion process and this property enables analysis specific to this two-color simple exclusion process.
		\end{example}

\cref{thm: intro} is also valid for the following example
of the generalized exclusion process, which is also called the partial exclusion process.

\begin{example}\label{example: 2}
	For an integer $\kappa>0$, let $S=\{0,1,\ldots,\kappa\}$ with $*=0$, the map
	given by
	\[	
		\phi(s_1,s_2)=
		\begin{cases}
				(s_1-1,s_2+1)  & s_1>0, s_2<\kappa\\
				(s_1,s_2)&\text{otherwise}
		\end{cases}
	\]
	is an interaction which is \emph{irreducibly quantified} (see \cite{BKS20}*{Proposition 2.27} for a proof). 
	This interaction gives the 
	\emph{generalized exclusion process}.
	We have $c_\phi=\dim_\R\C(S)=1$ in this case, 
	and a basis of the space of conserved quantities $\C(S)$ is given by
	$\xi \colon S\rightarrow\R$ such that
	$
		\xi(s)\coloneqq s.
	$
	Let $\nu$ be any probability measure fully supported on $S$.
	It is shown in Theorem 3.1. of \cite{Ca04} that $\nu$ has a uniformly bounded 
	spectral gap for the interaction $(S,\phi)$ in the sense of Definition \ref{def: SG},
	hence \cref{thm: intro} holds in this case for any Euclidean lattice with finite range $(\Z^d, E)$ with $d > 1$ and any shift-invariant weight $r=(r_e)_{e \in E}$. 

	This result is applicable to study the diffusive scaling limits for the following process. Fix a Euclidean lattice with finite range $(\Z^d, E)$, $d > 1$ and two functions $a, b: \{0,1,2, \dots, \kappa\} \to \R_{\ge 0}$ such that $a(0)=b(0)=0$ and $a(s)>0,b(s)>0$ for $s \neq 0$ . Consider a Markov process with generator 
	\[
		Lf(\e)=\sum_{e\in  E} a(\eta_{o(e)})b(\kappa -\eta_{t(e)}) (f(\e^e)-f(\e)).
	\]
	Then, the process is reversible for a one-parameter family of product measures $\mu_{\rho}=\nu_{\rho}^{\otimes\Z^d}$ determined by functions $a,b$ and indexed by the density of particles $\rho$ (cf. Section 2.2. of  \cite{SS18}, but note that the dynamics is symmetric in our case, so we do not need any further assumption on $a, b$). Choosing $r_e(\eta)=a(\eta_{o(e)})b(\kappa -\eta_{t(e)})+\mathbf{1}_{ \{ \eta_{o(e)}(\kappa-\eta_{t(e)})=0 \} }$, \cref{thm: intro} provides the key decomposition for the non-gradient method for this process. The indicator function $\mathbf{1}_{ \{ \eta_{o(e)}(\kappa-\eta_{t(e)})=0 \}}$ is added to make sure that $r_e(\eta)>0$ for any $\eta$ and $e \in E$, and indeed if we replace $a(\eta_{o(e)})b(\kappa -\eta_{t(e)})$ by $r_e(\eta)$ in the generator, the generator itself does not change. For the special case $a(s)=b(s)=\mathbf{1}_{\{ s \ge 1\}}$ on the standard Euclidean lattice $(\Z^d,\E^d)$ is studied in Section 7 of \cite{KL99}. 
	\end{example}


To study the diffusive scaling limit, such as the hydrodynamic limit and the equilibrium fluctuation, for nongradient models, the decomposition of \cref{thm: intro}
is necessary in order to apply existing methods to general models regardless of whether we use the entropy method or the relative entropy method.
(In contrast, for the specific case of the exclusion process, a method to avoid the use of the
decomposition has recently been announced by Funaki-Gu-Wang \cite{FGW24}). 
In fact, over the past thirty years, this decomposition has been established for a variety of models. All models discussed in this paragraph satisfy $c_{\phi}=1$ unless explicitly stated otherwise.
The first results in the analysis of nongradient models were obtained for the Ginzburg-Landau model in \cite{Var93} by Varadhan, where the local state space $S=\R$, and for the two-color simple exclusion process \cite{Qua92} by Quastel, where $S=\{0,1,2\}$ with two conserved quantities ($c_{\phi}=2$). 
Since then, the decomposition as well as the diffusive scaling limits are established for several nongradient models with a finite local state space $S$ and product reversible measures, including the generalized exclusion process \cite{KLO94, KL99} with a special jump rate, the exclusion process with speed change \cite{FUY96}, the lattice gas with energy ($c_{\phi}=2$) \cite{N03}, and a two-species exclusion process with annihilation and creation mechanism \cite{Sas10} for which our results are all applicable. Among nongradient models with non-product reversible measures, the only rigorous result for the decomposition and the hydrodynamic limit is given
in \cite{VY97}, where the exclusion process with certain mixing condition was studied. For models with a continuous local state space $S$, besides the Ginzburg-Landau model mentioned above, as far as we know, the decomposition theorem is only established for some energy conserving stochastic model \cite{Her14}, anharmonic oscillators with stochastic noise \cite{OS13}, a disordered harmonic chain with noise (where the decomposition is established for a class of \lq\lq good" forms) \cite{EM19} and an active exclusion process ($c_{\phi}=\infty$ in some sense) \cite{E21}. There is also a preprint \cite{LOS15} in this direction. The generalization of nongradient method to nonreversible models are also known, and the most essential part of the proof, namely the decomposition of shift-invariant
closed forms, is typically reduced to the reversible case (cf.\ \cite{Xu93, K98, FGQ04, Sas11, OS13}). As we do not assume the reversibility nor any relation between the measure $\mu$ and the weight $r=(r_e)$, our result is also applicable for nonreversible models. 

There are many important nongradient models even with product reversible measures for which the diffusive scaling limits are not established, including the multi-species simple exclusion process (cf.\ \cite{I18})
-- even in cases where the required spectral gap estimate is obtained.  
In addition,
the generalization of the nongradient method from a nearest neighbor interaction version to a finite range interaction version on $\Z^d$ are often claimed to be ``straightforward,'' but in fact,
there is no clearly written proof in the literature. 
Here, by the finite range interaction, we mean that the jumps of particles can be finite range, but do not mean the reversible measures have a finite range interaction. We have discovered that
when generalizing to the finite range interaction version,
there is in general an obstruction in the proof of the locality of the boundary term, as we remark in Section 6 of this article (see Remark \ref{rem:finite-range}). Hence in this article, we avoid a direct proof and reduce
our proof of the main theorem to the case of the nearest neighbor interaction on the Euclidean lattice $\Z^d$, or in our terminology, the model on the locale $(\Z^d, \E^d)$ (see \S\ref{subsec: PMT}).  One of the major advantages of creating a general formalism is the development of systematic methods to solve difficult problems by reducing to an easier case.
This allows, as in our case, proof of results even in cases where parallel method of proof fails to follow through. 

With a similar spirit, we also reduce the problem to the case $r_e \equiv 1$ in the beginning of the proof. This is possible since the set itself and its topology of $L^2(\mu)_{r_e}$ is the same for any $r=(r_e)$. This is not the case if $r=(r_e)$ is degenerate or unbounded. 

The extension of the result on the Euclidean lattice $(\Z^d, E)$ with finite range to a general crystal lattice $(X,E)$, such as the hexagonal lattice, is a fundamental open problem, and we are currently preparing some articles for this extension. With this in mind, we write
the assertions and proofs in this article generally as possible with respect to the locale $(X,E)$. In the long term, we also hope that our main result may be extended to the case where $S$ is countable infinite, continuous, and $\nabla$ is a differential operator, an integral operator, and so on under suitable conditions.

As mentioned above, there have been many results on specific models, but up until now, there was no general framework for the nongradient method nor the decomposition of shift-invariant closed forms. The biggest factor that prevented full generalization was that the dimension and concrete expression of ``closed but not exact forms'' were obtained by calculating simultaneous linear equations in each model ``on a case-by-case basis,'' and there was no unified understanding of them. Such concrete calculations are impractical when the dimension of the space of conserved quantities is large, such as $c_{\phi}=100$ or $10,000$. The main reason we have succeeded in a complete generalization is that we have completely solved this problem in a previous article \cite{BKS20}, as a problem independent of probability measures. To connect the result for uniform forms to $L^2$-forms, in this article, we introduce several projective limits, which allow for us to make sense of infinite sums such as $\sum_{x\in\Z^d}\tau_x f$ and  $\sum_{x\in\Z^d}x_j\xi^{(i)}_{x}$ not only in a formal way but as a uniform function, which is a well-defined object. Also, to state the result in full generality, we define the notion of ``closed forms'' independent form specific models, just as the set of forms which are locally exact under restriction via conditional expectation. 
With such difficulties aside,
the proofs in this article are roughly in line with the standard methods introduced by Varadhan and developed in subsequent research mentioned above. The crucial difference is the treatment of the boundary term that remains at the end. In the existing proofs, it is necessary to find the limit of the boundary term specifically for each model, but in this article, we just show that the limit is a shift-invariant uniform
closed form, and the rest can be attributed to the result of the previous article \cite{BKS20}, thus 
completely avoiding model-specific discussions. We also note that the proof of the Boundary Estimate (Proposition \ref{prop: BE}), which is one of the key estimates of Varadhan’s non-gradient method, in the general setting is not 
a simple extension of existing methods, since we cannot use any specific property of the model.

The rest of the article is organized as follows. In Section \ref{sec:co-local}, we introduce our general setting and basic notions, such as uniform functions and forms as well as their norms. In this section, we consider a general locale $(X,E)$, and the existence of a group action is not assumed. In Section 3, we assume that a locale $(X,E)$ has an action of a group $G$, which is free and introduce some notions, such as exact/closed shift-invariant $L^2$-forms. Then, we introduce the notion of the uniform spectral gap estimate, which is an essential assumption of our main theorem, and the main theorem as well. We note that the uniform spectral gap estimate condition depends only on $(S,\phi)$ and $\nu$, but not on the locale $(X,E)$ and the rate $r=(r_e)$. In Section 4, we prove a few key estimates. All the results in this section hold for any locale $(X,E)$. Section 5 is devoted to the construction of converging sequence which directly gives the decomposition of \cref{thm: intro} for the locale $(X,E)=(\Z^d, \E^d)$. Finally, in Section 6 we give a proof of the main theorem. We first prove the result for the locale $(X,E)=(\Z^d, \E^d)$, and then show that the result for more general finite range Euclidean lattice $(\Z^d, E)$ can be reduced to the case of the standard Euclidean lattice $(\Z^d, \E^d)$. 
%
%
%
\section{Co-local Functions and Forms}\label{sec:co-local}
%
%
%

In this section, we introduce the data $((X,E),(S,\phi),r)$ which will be used to construct our model.
We will then introduce the notion of co-local functions and forms on the configuration space of our model.

%
\subsection{The Large Scale Interacting System}\label{subsec: IS}
%

In this subsection, we will introduce the data $((X,E),(S,\phi),r)$
which we use to model the large scale interacting system.

We define a \emph{directed graph} to be the pair $(X,E)$
consisting of a set $X$ which we call the \emph{set of vertices} and $E\subset X\times X$,
which we call the \emph{set of directed edges}.  We let $o,t\colon E\rightarrow X$ be the
projections to the first and second components of $X\times X$,
which we call the \emph{origin} and the \emph{target} of a directed edge.
We say that a graph $(X,E)$ is \emph{symmetric},
if the opposite $\bar e\coloneqq(t(e),o(e))\in E$ for any $e=(o(e),t(e))\in E$, \emph{locally finite}
if $E_x\coloneqq\{e\in E\mid o(e)=x\}$ is finite for any $x\in X$,
and \emph{simple} if $o(e)\neq t(e)$ for any
$e\in E$.

We define a \emph{path} on $(X,E)$ to be a finite sequence $\vec p\coloneqq(e^1,\ldots,e^N)$
of edges satisfying $t(e^i)=o(e^{i+1})$ for any $0<i<N$.  
We say that $\vec p$ is a path from $o(\vec p)\coloneqq o(e^1)$ to $t(\vec p)\coloneqq t(e^N)$ of 
length $\len(\vec p)\coloneqq N$.  
For any $x,x'\in X$, we let $P(x,x')$ the set of paths from $x$ to $x'$.
We define the \emph{distance} from $x$ to $x'$ by
\[
	d_{X}(x,x')\coloneqq\inf_{\vec p\in P(x,x')}\len(\vec p)
\]
if $P(x,x')\neq\emptyset$, and $d_X(x,x')\coloneqq\infty$ otherwise.
For any $Y\subset X$, we say that $Y$ is \emph{connected}, if for any $y,y'\in Y$,
there exists a path from $y$ to $y'$ consisting of edges whose origin and target are all elements of $Y$.

\begin{definition}
	We define a \emph{locale} to be a connected locally finite simple symmetric directed graph.
	The locale represents the space underlying our large scale interacting system.
\end{definition}

The most important example of a locale is the Euclidean lattice $(\Z^d,\E^d)$ given in \S\ref{sec: intro}.
Let $(X,E)$ be a graph.  For any $Y\subset X$, if we let $E_Y\coloneqq E\cap(Y\times Y)$,
then the pair $(Y,E_Y)$ defines a graph.  A \emph{subgraph of $(X,E)$} is any graph of the form
$(Y,E_Y)$
for some $Y\subset X$.  In particular, if $(X,E)$ is a locale and the set $Y\subset X$ is connected,
then $(Y,E_Y)$ is a locale.  A \emph{sublocale of $(X,E)$} is any locale of the form $(Y,E_Y)$ for some connected
$Y\subset X$. A \emph{finite locale} is a locale $(X,E)$ whose vertex set $X$ is finite.

As in \S\ref{sec: intro}, we define the \emph{set of states} $S$ to be a finite nonempty set,
and  we define the \emph{configuration space} of $S$ on $(X,E)$ to be the set
\[
	S^X\coloneqq\prod_{x\in X}S.
\]
An \emph{interaction} is defined to be a map $\phi\colon {S\times S}\rightarrow {S\times S}$,
satisfying $\bar\phi\circ\phi(s_1,s_2)=(s_1,s_2)$ for any $(s_1,s_2)\in S\times S$
such that $\phi(s_1,s_2)\neq(s_1,s_2)$, where we let
\[
	\bar\phi\coloneqq \iota\circ\phi\circ\iota
\]
for the bijection $\iota\colon S\times S\rightarrow S\times S$ obtained by
exchanging the components of $S\times S$.  The pair $(S,\phi)$ describes
the set of all possible states on a single vertex and the possible interactions
between adjacent vertices.

\begin{example}\label{example: interaction}
	Let $S=\{0,1,2,\ldots,\kappa\}$ for some integer $\kappa>0$.
	\begin{enumerate}\renewcommand{\labelenumi}{(\alph{enumi})}
		\item The map $\phi=\iota\colon S\times S\rightarrow S\times S$
		above
	obtained by exchanging the components
	of $S\times S$ defines an interaction giving the \emph{multi-species exclusion process}.
		\item
	 The map given by
	\[	
		\phi(s_1,s_2)=
		\begin{cases}
				(s_1-1,s_2+1)  & s_1>0, s_2<\kappa\\
				(s_1,s_2)&\text{otherwise}
		\end{cases}
	\]
	is also an interaction, giving the \emph{generalized exclusion process}.
	\end{enumerate}
	See \cite{BKS20}*{Example 2.18} for other examples of interactions.
\end{example}

We next introduce a certain invariant called the \emph{conserved quantity}
associated to the pair $(S,\phi)$.

\begin{definition}\label{def: CQ}
	We fix an element $*\in S$ which we call the \emph{base state}.
	We define a \emph{conserved quantity}, to be any map $\xi\colon S\rightarrow\R$
	satisfying $\xi(*)=0$ and
	\[
		\xi(s'_1)+\xi(s'_2)=\xi(s_1)+\xi(s_2)
	\]
	for any $(s_1,s_2)\in S\times S$, where $(s'_1,s'_2)\coloneqq\phi(s_1,s_2)$.
\end{definition}

We denote by $\C_*(S)$ the set of all conserved quantities, which has a natural structure
of an $\R$-linear space.  Given another base state $*'\in S$, we have a canonical $\R$-linear
isomorphism $\C_*(S)\cong\C_{*'}(S)$ mapping any $\xi\in\C_*(S)$ to $\xi'\coloneqq\xi-\xi(*')$.
Hence $\C_*(S)$ is independent up to canonical isomorphism of the choice of the base state.
For this reason, we will simply denote $\C_*(S)$ by $\C(S)$.  The dimension
\[
	c_\phi\coloneqq\dim_\R\C(S)
\]
is independent of the choice of the base state, and describes the number of independent
conserved quantities of the system.

The data $((X,E),(S,\phi))$ above gives the geometric data underlying our interacting system.
We next give some construction associated to this data.
For any locale $(X,E)$ and the pair $(S,\phi)$, we call any $\e\in S^X=\prod_{x\in X}S$ 
a \emph{configuration}.
For any element $\e=(\eta_x)\in  S^X$ and $e\in E$, we denote by $\e^e$
the element $\e^e=(\eta^e_x)\in S^X$ such that
\[
	\eta^e_x\coloneqq\begin{cases} 
		\eta_x    &  x \neq o(e),t(e)\\
		\eta'_{x} & x=o(e),t(e),  \\
		\end{cases}
\]
where $(\eta'_{o(e)},\eta'_{t(e)})=\phi(\eta_{o(e)}, \eta_{t(e)})$.
We define a subset $\Phi_E$ of $S^X\times S^X$ by
\[
	\Phi_E\coloneqq\{(\e,\e^e)\mid \e\in S^X, e\in E\}\subset S^X\times S^X.
\]
Then the pair $(S^X,\Phi_E)$ is a graph, which we call the \emph{configuration space with transition structure}.

\begin{lemma}\label{lem: symmetric}
	The graph $(S^X,\Phi_E)$ is a symmetric directed graph.
\end{lemma}

\begin{proof}
	This is proved in \cite{BKS20}*{Lemma 2.5}, and can be seen from the fact that
	for any $\e\in S^X$ and $e\in E$, if $\e^e\neq\e$, then from the condition
	on $\phi$, we have $(\e^e)^{\bar e}=\e$,
	hence $(\e^e,\e)=(\e^e,(\e^e)^{\bar e})\in\Phi_E$.
\end{proof}

The configuration space $S^X$ describes all of the possible configuration of states
of our system, and the transition structure describes all of the possible transitions which
may occur at an instance in time.  A condition that we often consider for the pair 
$(S,\phi)$ is the following.

\begin{definition}\label{def: FQ}
	We say that an interaction $(S,\phi)$ is \emph{irreducibly quantified}, if it 
	satisfies the following property:  For any finite locale $(X,E)$ and any 
	configurations $\e,\e'\in S^X$,  if $\sum_{x\in X}\xi(\eta_x)=\sum_{x\in X}\xi(\eta'_x)$
	for any conserved quantity $\xi\in\C(S)$, then there exists a path $\vec\gamma$ from
	$\e$ to $\e'$ in $S^X$.
\end{definition}

\emph{Irreducibly quantified} is equivalent to the condition that the associated stochastic process on any finite locale $(X,E)$ with fixed conserved quantities is \emph{irreducible}.  This condition plays an important role in the proof of the main theorem.

\begin{example} The interactions in Example \ref{example: interaction} are both irreducibly quantified.
	We let $*=0$ be the base state of $S=\{0,\ldots,\kappa\}$.
	\begin{enumerate}\renewcommand{\labelenumi}{(\alph{enumi})}
		\item For the multi-species exclusion process, we have $c_\phi=\kappa$, and a basis of $\C(S)$
		is given by $\xi^{(1)},\ldots,\xi^{(\kappa)}\colon S\rightarrow\R$ such that 
		$\xi^{(i)}(s)=\delta_{si}$ for $i=1,\ldots,\kappa$,
		where $\delta_{si}\coloneqq1$ if $s=i$ and $\delta_{si}\coloneqq0$ if $s\neq i$.
		\item For the generalized exclusion process, we have $c_\phi=1$, and a basis of $\C(S)$
		is given by $\xi\colon S\rightarrow\R$ such that $\xi(s)=s$.
	\end{enumerate}
	See \cite{BKS20}*{Example 2.18} for other examples of irreducibly quantified interactions.
\end{example}

Next, we introduce certain metric like data called the \emph{weight}. 
For any $\La\subset X$ let $C(S^\La)\coloneqq\Map(S^\La,\R)$
be the $\R$-linear space of $\R$-valued maps on $S^\La$.
The natural projection $S^X\rightarrow S^\La$ induces an injection $C(S^\La)\hookrightarrow C(S^X)$,
which allows us to view any $f\in C(S^\La)$ as a function on $S^X$ whose value on $\e=(\eta_x)$
depends only on the components for $x\in\La$.   A \emph{local function} on $S^X$
is any function in $C(S^\La)$ for a finite $\La\subset X$.

\begin{definition}\label{def: weight}
	We define the \emph{weight} of the system $((X,E),(S,\phi))$ to be a family of local
	functions $r=(r_e)_{e\in E}$ satisfying $r_e(\e)>0$ for any $e \in E$ and $\e\in S^X$.
	\end{definition}

 



In terms of application to the diffusive scaling limit for nongradient models, the weight should be computed by the frequency of each transition from $\eta$ to $\eta^e$, which is so-called the jump rate, and it is natural to consider an invariant or stationary probability measure $\mu$ on $S^X$ for the process given by this jump rate. However, since we can prove the main result 
without the stationarity nor reversibility condition, and instead require $\mu$ to be a product measure, we choose a probability measure
$\nu$ on $S$, satisfying $\nu(\{s\})>0$ for any $s\in S$ independently from the weight $r$, and consider the 
product measure $\mu\coloneqq\nu^{\otimes X}$. 

\begin{remark}
For any given interaction $(S,\phi)$, locale $(X,E)$ and probability measure $\nu$ fully supported on $S$, we can construct a large class of processes which are reversible with respect to $\mu=\nu^{\otimes X}$. 
\end{remark}

The following constants will be used in the calculation of the bound of norms later.

\begin{definition}\label{def: MB}
	For an interaction $(S,\phi)$ and a probability measure $\nu$ fully supported on $S$, we define 
		\begin{equation}\label{eq: MB}
	 	C_{\phi,\nu}:=\max_{(s_1,s_2) \in S^2} \max \left\{ \frac{\nu(s'_1)\nu(s'_2)}{\nu(s_1)\nu(s_2)}, \frac{\nu(s_1)\nu(s_2)}{\nu(s'_1)\nu(s'_2)} \right\}
		\end{equation}
		where $(s'_1,s'_2)=\phi(s_1,s_2)$. Since $S$ is finite and $\nu$ is supported on $S$, 
		we have $1 \leq C_{\phi,\nu} <\infty$. 
\end{definition}

%
\subsection{Uniform Functions}\label{subsec: functions}
%

We will study the property of the configuration space with transition structure using
a class of functions which we call the \emph{uniform co-local functions}.
This is a variant of the uniform function originally defined in \cite{BKS20}*{\S3}.
In this subsection, we will review the definition of co-local functions and forms,
as well as the definition of uniform functions.  Detailed proofs of the results
of this section are given in \cite{BS21}*{\S2 and \S3}.  

Let $(X,E)$ be a locale, and let $\sI$ be the set of finite subsets of $X$.
For any $\La\in\sI$, as in \S\ref{subsec: IS}, we let $C(S^\La)=\Map(S^\La,\R)$ be the 
$\R$-linear space of $\R$-valued functions on $S^\La$.  We may regard 
any $f\in C(S^\La)$ as a function on $S^X$ whose value depends only
on the components of $\e=(\eta_x)$ for $x\in\La$.
For any $\La,\La'\in\sI$ such that $\La\subset\La'$, the natural projection
$S^{\La'}\rightarrow S^\La$ induces an inclusion $C(S^\La)\hookrightarrow C(S^{\La'})$,
which forms a direct system with respect to inclusions in $\sI$.

\begin{definition}
	We define the  \emph{space of local functions} by 
	\[
		C_\loc(S^X)\coloneqq\varinjlim_{\La\in\sI}C(S^\La)=\bigcup_{\La\in\sI}C(S^\La),
	\]
	where the direct limit is taken with respect to the inclusions $C(S^\La)\hookrightarrow C(S^{\La'})$.
\end{definition}

We let $\cF$ be the set of subsets of $S$, which is a $\sigma$-algebra.
For any $\La\subset X$, we let $\cF_\La\coloneqq\cF^{\otimes\La}$ be the product 
$\sigma$-algebra on $S^\La$ obtained from $\cF$. We note that $\cF_{\La}$ is the sub-$\sigma$-algebra of $\cF_{X}$ generated by the projection $S^X\rightarrow S^\La$. 

Next, let $\mu$ 
be a probability measure on $(S^X,\cF_X)$.
For any $\La\subset X$, we define $C(S^\La_\mu)$ to be the $\R$-linear space of $\R$-valued
\emph{measurable} functions on $(S^\La, \cF_{\La})$.  If $\La$ is finite, then we have $C(S^\La_\mu)=C(S^\La)$,
since we have assumed that $S$ is finite.
The natural projection $S^X\rightarrow S^\La$ induces an inclusion $C(S^\La)\hookrightarrow C(S^X_\mu)$.

For any $\La,\La'\in \sI$ such that $\La\subset\La'$ and any
\emph{integrable} function $f\in C(S^{\La'})$, we let
$
	\pi^\La f\coloneqq E_\mu[f|\cF_\La]
$
be the \emph{conditional expectation} associated to 
the projection $\pr_\La\colon S^{\La'}\rightarrow S^\La$. Note that for any  $\La \in \sI$ and $f \in C(S^{\La})$, $f$ is integrable since $S$ is finite.
If $\La,\La'\in\sI$, then $\pi^\La$ induces the homomorphism
\[
	\pi^\La\colon C(S^{\La'})\rightarrow C(S^\La).
\]
Note that for any $\La\subset\La'\subset\La''$ and integrable function $f\in C(S^{\La''})$, 
the tower property of the conditional expectation gives
$\pi^\La f=\pi^\La(\pi^{\La'}f)$, which implies that $\pi^\La$ satisfies the compatibility
condition for the projective system.
Hence as in \cite{BS21}*{Definition 1.4}, we may
define the space of co-local functions $ C_\col(S^X_\mu)$
as follows.

\begin{definition}
	We define the space of \textit{co-local functions} by
	\[
		C_\col(S^X_\mu)\coloneqq\varprojlim_{\La\in\sI}C(S^\La),
	\] 
	where the projective limit is taken with respect to the projections $\pi^\La$.
	We call any element in $C_\col(S^X_\mu)$ a co-local function.
\end{definition}

By definition, a co-local function $f=(f^\La)\in  C_\col(S^X_\mu)$ is a system of functions $f^\La\in C(S^\La)$
for $\La\in\sI$ such that $\pi^\La f^{\La'}=f^\La$ for any $\La\subset\La'$.  In other words,
co-local functions are \emph{Martingales} indexed by $\La\in\sI$ for the filtration $(\cF_{\La})_{\La \in \sI}$.
Any integrable function $f\in C(S^X_\mu)$ defines the co-local function $(\pi^\La f)\in  C_\col(S^X_\mu)$.
This gives an injection 
\[
	C_\loc(S^X)\hookrightarrow  C_\col(S^X_\mu).
\]
When $\mu$ is a product measure $\mu=\nu^{\otimes X}$ where $\nu$ is a probability measure fully supported on $S$, the co-local functions may be expanded uniquely as an infinite sum of certain local functions.

\begin{proposition}[\cite{BS21}*{Proposition 2.3}]\label{prop: expansion}
	Suppose $\mu$ is a product measure $\mu=\nu^{\otimes X}$ on $S^X$, 
	where $\nu$ is a probability measure
	fully supported on $S$.
	For any $\La\in\sI$, let
	\[
		C_{\La}(S^X_\mu)\coloneqq\{f\in C(S^\La)\mid\pi^{\La'} f\equiv0\text{ if }\La\not\subset\La' \ \forall \La' \in \sI \}.
	\]
	Then for any co-local function 
	$(f^\La)\in  C_\col(S^X_\mu)$, there exists a unique family of functions $f_{\La''}\in C_{\La''}(S^X_\mu)$
	such that
	\begin{equation}\label{eq: satisfy}
		f^\La=\sum_{\La''\subset\La} f_{\La''}
	\end{equation}
	for any $\La\in\sI$.
\end{proposition}

\begin{proof}
	See the proof of \cite{BS21}*{Proposition 2.3} for details. 
	We let $f_\emptyset\coloneqq f^\emptyset=E_\mu[f]$ for the case $\La=\emptyset$,
	and we define inductively with respect to the number of elements of $\La$ as
	\[
		f_{\La}\coloneqq f^\La-\sum_{\La''\subsetneq\La}f_{\La''}
	\]
	for any $\La\in\sI$.  To prove that $f_\La\in C_\La(S^X_\mu)$ and satisfies the desired property,
	we use the fact proved in \cite{BS21}*{Lemma 2.2}
	that for $\La\in\sI$ and $\La',\La''\subset\La$, we have
	\[
		\pi^{\La'}(\pi^{\La''}f)=\pi^{\La'\cap\La''}f
	\]
	for any $f\in C(S^\La)$, which is true since $\mu$ is a product measure.
\end{proof}

\begin{remark}
Though we do not use the fact in this article, Proposition \ref{prop: expansion} holds for $\mu=\Pi_{x \in X} \nu_x$ where $\nu_x$ is a fully supported probability measure on $S$ for any $x \in X$. Actually, the proof above works directly. This will be useful in future works for the cases with inhomogeneous invariant measures.
\end{remark}

In the rest of the article, we always assume that $\mu$ is a product measure $\mu=\nu^{\otimes X}$ for some probability measure $\nu$ fully supported on $S$. 

For any co-local function $f=(f^\La)$, we will often write
$
	f=\sum_{\La''\in\sI}f_{\La''}
$
to imply that \eqref{eq: satisfy} holds for any $\La\in\sI$.
We let $C^0_\col(S^X_\mu)\coloneqq \{f\in  C_\col(S^X_\mu)\mid f_\emptyset=f^\emptyset=0\}$,
which we call the space of co-local functions of \emph{mean zero}.

We define the \emph{diameter} of $\La\subset X$ by 
$\diam{\La}\coloneqq\sup_{x,x'\in\La}d_X(x,x')$.
As in \cite{BS21}*{Definition 2.4}, we define the space of uniform co-local functions as follows.

\begin{definition}\label{def: UL}
	We say that a co-local function $f$ is \textit{uniform co-local},
	if there exists $R\geq0$ such that the expansion of Proposition \ref{prop: expansion}
	is given by
	\begin{equation}\label{eq: expansion}
		f=\sum_{\La\in\sI_R}f_\La,
	\end{equation}
	where we denote by $\sI_R$ the set consisting of $\La\in\sI$ such that $\diam{\La}\leq R$.
	We call the infimum of such $R$ the \emph{diameter} of $f$.
	We denote by $C_\unif(S^X_\mu)$ the $\R$-linear space of uniform co-local functions,
	and by $C^0_\unif(S^X_\mu)$ the subspace of \emph{normalized}
	uniform co-local functions satisfying $f_\emptyset=f^\emptyset=0$.
\end{definition}

We will next give an algebraic construction of the space of uniform co-local functions.
Namely, we will review the construction of the space of uniform functions
$C^0_\unif(S^X_*)$ denoted
$C^0_\unif(S^X)$ in \cite{BKS20}*{Definition 3.5}.   We will show in Proposition
\ref{prop: UL} that there exists a canonical isomorphism which we call the
\emph{renormalization} between $C^0_\unif(S^X_*)$ and $C^0_\unif(S^X_\mu)$.

We fix a point $*\in S$ which we call the \emph{base point}, and we denote by $\star\in S^X$
the configuration whose components are all at base state.
For any configuration  $\e\in S^X$, we define the support of $\e$ by $\Supp(\e)\coloneqq\{x\in X\mid \eta_x\neq*\}$,
and we let
\[
	S^X_*\coloneqq\{\e\in S^X\mid \abs{\Supp(\e)}<\infty \}.
\]
We let
$C(S^X_*)\coloneqq\Map(S^X_*,\R)$ to be the $\R$-linear space of maps from $S^X_*$ to $\R$.

For any $\La\in\sI$, the projection $S^X_*\rightarrow S^\La$ induces an inclusion $C(S^\La)\hookrightarrow C(S^X_*)$, hence we may view $C_\loc(S^X)$ as an $\R$-linear subspace of $C(S^X_*)$.
For any $\La,\La'\in\sI$ such that $\La\subset\La'$, we have an inclusion
$\iota_\La\colon S^\La\hookrightarrow S^{\La'}$ given by extending any configuration 
$\e\in S^\La$ by $\eta_x\coloneqq*$ for $x\in\La'\setminus\La$.
This induces a projection $\iota^\La\colon C(S^{\La'})\rightarrow C(S^\La)$. In particular, for any $\La, \La' \in \sI$, the projection $\iota^{\La'} : C(S^{\La})\rightarrow C(S^{\La'}) $ is well-defined as $\iota^{\La'}=  \iota^{\La'} \circ \iota_{\La' \cup \La}$. 

\begin{definition}
	For any $\La\in\sI$, we define the space of local functions with \emph{exact support $\La$} by
	\[
		C_{\La}(S^X_*)\coloneqq\{f\in C(S^\La)\mid\iota^{\La'} f\equiv0\text{ if }\La\not\subset\La'  \ \forall \La' \in \sI\}.
	\]
\end{definition}

We have an expansion of functions in $C(S^X_*)$ via local functions with exact support as follows.

\begin{proposition}\label{prop: expansion2}
	Let $(f_\La)_{\La\in\sI}$ be a family of functions $f_\La\in C_\La(S^X_*)$.  Then
	\[
		f\coloneqq\sum_{\La\in\sI}f_\La
	\]
	defines a function in $C(S^X_*)$.  Conversely, 
	for any function $f\in C(S^X_*)$, there exists a unique family of functions $f^*_\La\in C_{\La}(S^X_*)$
	such that
	\begin{equation}\label{eq: satisfy2}
		 f=\sum_{\La\in\sI} f^*_{\La}.
	\end{equation}
\end{proposition}

\begin{proof}
	The first statement is \cite{BKS20}*{Lemma 3.2}, and follows from the fact that $f(\e)$
	for any $\e\in S^X_*$ is a finite sum since $f_\La(\e)\neq0$
	only if $\La\subset\Supp(\e)$.
	The second statement is \cite{BKS20}*{Proposition 3.3}, noting that 
	\[
		C_{\La}(S^X_*)=\{ f\in C(S^\La)\mid f(\e)=0 \text{ if $\exists x\in\La$ such that $\eta_x=*$}\}.
	\]
	For any $f\in C(S^X_*)$, we let 
	$f^*_\emptyset\coloneqq \iota^\emptyset f=f(\star)$ for the case $\La=\emptyset$.
	We then define $f^*_\La$ by induction on the order of $\La$ as
	\[
		f^*_{\La}\coloneqq \iota^\La f-\sum_{\La''\subsetneq\La}f^*_{\La''}
	\]
	for any $\La\in\sI$.  This construction gives the desired property.
\end{proof}

By \cite{BS21}*{Lemma 1.3}, any function $C(S^X_*)$ may be expressed as a projective limit
of local functions.  Hence we will refer to a function in $C(S^X_*)$ as a \emph{pro-local function}.
As in \cite{BKS20}*{Definition 3.5}, we define the space of uniform functions as follows.

\begin{definition}\label{def: UL2}
	We say that a function $f\in C(S^X_*)$ is \textit{uniform pro-local},
	if there exists $R\geq0$ such that the expansion of Proposition \ref{prop: expansion2}
	is given by
	\[
		f=\sum_{\La\in\sI_R}f^*_\La.
	\]
	We call the infimum of such $R$ the \emph{diameter} of $f$.
	We denote by $C_\unif(S^X_*)$ the set of all uniform pro-local functions in $C(S^X_*)$,
	and by $C^0_\unif(S^X_*)$ the subspace of \emph{normalized} 
	uniform pro-local functions satisfying $f^*_\emptyset=\iota^\emptyset f=f(\star)=0$.
\end{definition}

Let $\xi\in\C(S)$ be a conserved quantity.   Then for any $x\in X$, 
the function $\xi_x\colon S\rightarrow\R$ given by $\xi_x(\e)\coloneqq\xi(\eta_x)$ for any $\e\in S^X$
defines a function in $C_{\{x\}}(S^X_*)$.  By Proposition \ref{prop: expansion2}, the infinite sum
\[
	\xi_X\coloneqq\sum_{x\in X}\xi_x
\]
defines a uniform pro-local function of diameter $0$ on $S^X$.

We define the \emph{renormalization} as follows.

\begin{proposition}\label{prop: UL}
	We have a canonical isomorphism which we call the renormalization
	\begin{equation}\label{eq: isom}
		\sR\colon C^0_\unif(S^X_*)\cong C^0_\unif(S^X_\mu).
	\end{equation}
\end{proposition}

\begin{proof}
	Let $f\in C^0_\unif(S^X_*)\subset C^0(S^X_*)$.  Then by Proposition \ref{prop: expansion2}
	and the definition of uniform pro-local functions given in Definition \ref{def: UL2}, there exists $R\geq0$ such that
	\[
		f=\sum_{\La''\in\sI_R} f^*_{\La''}.
	\]
	By Proposition \ref{prop: expansion}, for each $\La'' \in \sI_R$, there is the unique expansion
	of $f^*_{\La''}$ considered as a co-local function as
	\[
		f^*_{\La''}=\sum_{\La'\subset\La''}f^*_{\La'',\La' },
	\]
       where $f^*_{\La'',\La'} \in C_{\La'}(S^X_{\mu})$. 
	Now, for each $\La'\in\sI$ such that $\La'\neq\emptyset$, we construct $\sR(f)_{\La'}$ by
	\[
		\sR(f)_{\La'}\coloneqq\sum_{\substack{\La''\in \sI_R\\\La'\subset\La''}}f^*_{\La'',\La'}.
	\]
	Note that since the number of $\La''\in\sI_R$ such that $\La'\subset\La''$  is finite,
	$\sR(f)_{\La'}$ gives a function in $C_{\La'}(S^X_{\mu})$.
	We define $\sR(f)\in  C_\col(S^X_\mu)$ by
	\[
		\sR(f)=\sum_{\La'\in\sI^0}\sR(f)_{\La'},
	\]
	where $\sI^0\coloneqq\sI\setminus\{\emptyset\}$.
	Note that $\sR(f)_{\La'}=0$ if $\diam{\La'}>R$, hence $\sR(f)$ is a co-local function in $C^0_\unif(S^X_\mu)$
	as desired. 
	We may construct the inverse isomorphism of $\sR$ in a similar manner, proving that \eqref{eq: isom}
	is an isomorphism.
\end{proof}

Given two base states $*,*'\in S$, one may also construct a canonical 
isomorphism $C^0_\unif(S^X_*)\cong C^0_\unif(S^X_{*'})$ by the same manner as the last proposition.
In what follows, we will denote $C^0_\unif(S^X_*)$ simply as $C^0_\unif(S^X)$, and will regard
any element $f\in C^0_\unif(S^X)$ also as elements in $C^0_\unif(S^X_{*'})$ or $C^0_\unif(S^X_\mu)$
through the renormalization isomorphisms. 

\begin{definition}
A \textit{uniform function} is any element of $C^0_\unif(S^X)$ or $C^0_\unif(S^X_\mu)$, where they are identified by the canonical 
isomorphism.
\end{definition}

%
\subsection{Uniform Forms}\label{subsec: forms}
%

In this subsection, we consider the space of forms on the configuration space
with transition structure. 
For any connected $\La\in\sI$, the 
graph $(\La, E_\La)$ is a finite locale.  We let $(S^\La,\Phi_{E_{\La}})$ be the associated
configuration space with transition structure, where
\[
	\Phi_{E_{\La}}\coloneqq\{(\e,\e^e)\mid\e\in S^\La,e\in E_\La\}.
\]
For any $\vp=(\e,\e^e)\in\Phi_{E_{\La}}$, let $\ol\vp\coloneqq(\e^e,\e)$, which by Lemma \ref{lem: symmetric}
is an element in $\Phi_{E_{\La}}$.

\begin{definition}
	We define \textit{the space of forms} on $S^\La$ by
	\[
		C^1(S^\La)\coloneqq\Map^\alt(\Phi_{E_{\La}},\R),
	\]
	where $\Map^\alt(\Phi_{E_{\La}},\R)\coloneqq\{\omega
	\in\Map(\Phi_{E_{\La}},\R)\mid\omega(\ol\vp)=-\omega(\vp)\}$.
\end{definition}

We have a natural inclusion
\begin{equation}\label{eq: inclusion}
	C^1(S^\La)\hookrightarrow\prod_{e\in E_\La}C(S^\La)
\end{equation}
given by $\omega\mapsto (\omega_e)_{e\in E_\La}$, where $\omega_e\in C(S^\La)$
is the map given by $\omega_e(\e)\coloneqq \omega((\e,\e^e))$
for any $e\in E_\La$ and $\e\in S^\La$.
The image of any $\omega\in C^1(S^\La)$ with respect to the
inclusion \eqref{eq: inclusion}
consists of $(\omega_e)$ satisfying $\omega_e(\e)=0$ if $\e^e=\e$,
$\omega_e(\e)=-\omega_{\bar e}(\e^e)$ if $\e^e\neq\e$, and $\omega_e(\e)=\omega_{e'}(\e)$
if $\e^e=\e^{e'}$.

\begin{lemma}\label{lem: preserve}
	For any $\La,\La'\in\sI$ such that $\La\subset\La'$, the projection 
	$\pi^\La\colon C(S^{\La'})\rightarrow C(S^\La)$ induces via \eqref{eq: inclusion} 
	an homomorphism
	\[
		\pi^\La\colon C^1(S^{\La'})\rightarrow C^1(S^\La)
	\]
	given by $(\pi^\La\omega)_e=\pi^\La \omega_e$ for $\omega \in C^1(S^{\La'})$ and $e \in E_{\Lambda}$.
\end{lemma}

\begin{proof}
	This follows from \cite{BS21}*{Lemma 3.4}, \cite{BS21}*{Lemma 3.5},
	and the fact that we have assumed that $\mu$ is the product measure.
\end{proof}

Using Lemma \ref{lem: preserve}, we define the space of co-local forms as follows.

\begin{definition}
	We define \textit{the space of co-local forms} on $S^\La$ by
	\[
		C^1_\col(S^X_\mu)\coloneqq\varprojlim_{\La\in\sI}C^1(S^\La),
	\]
	where the projective limit is taken with respect to the projections $\pi^\La$.
 \end{definition}

The inclusion \eqref{eq: inclusion} gives the inclusion 
\begin{equation}\label{eq: co-local inclusion}
	C^1_\col(S^X_\mu)\hookrightarrow\prod_{e\in E} C_\col(S^X_\mu).
\end{equation}
For any $\La\in\sI$, we define the differential
\begin{equation}\label{eq: differential}
	\partial_\La\colon C(S^\La)\rightarrow C^1(S^\La)
\end{equation}
to be the $\R$-linear homomorphism given by $\partial_\La f(\vp)\coloneqq f(t(\vp))-f(o(\vp))$.
The image of $\partial_\La f(\vp)$ with respect to the inclusion \eqref{eq: inclusion}
coincides with $(\nabe f)_{e\in E_\La}$, where $\nabe f(\e)\coloneqq f(\e^e)-f(\e)$
for any $\e\in S^\La$.   Since the measure $\mu$ is a product measure, we have the following.

\begin{proposition}\label{prop: compatible}
	The projection $\pi^\La$ is compatible with the differentials $\partial_\La$ on $C(S^\La)$, namely $\partial_\La \pi^\La f= \pi^{\La}\partial_{\La'}f$ for any $\La \subset \La'$ and $f \in C(S^{\La'})$.
	Hence the differential on each $C(S^\La)$ induces the differential
	\[
		\partial\colon  C_\col(S^X_\mu)\rightarrow C^1_\col(S^X_\mu)
	\]
	such that $(\partial f)^{\Lambda}=\partial_{\Lambda}f^{\Lambda}$ for $f=(f^{\Lambda}) \in C_\col(S^X_\mu)$.
\end{proposition}

\begin{proof}
	This follows from \cite{BS21}*{Proposition 3.7}, noting that a product measure is 
	\emph{ordinary}, as shown in \cite{BS21}*{Lemma 3.4}.
\end{proof}

Since the projection $\pi^\La$ on forms is induced from the embedding \eqref{eq: inclusion},
the projection is also compatible with $\nabe$ for any $e\in E_\La$.  This implies that $\nabe$ induces
a differential
\begin{equation}\label{eq: nabe}
	\nabe\colon  C_\col(S^X_\mu)\rightarrow  C_\col(S^X_\mu)
\end{equation}
such that we have $\partial f=(\nabe f)_{e\in E}\in C^1_\col(S^X_\mu)\subset\prod_{e\in E} C_\col(S^X_\mu)$ 
for any $f\in  C_\col(S^X_\mu)$.

We next define the notion of closed forms.
Let $\La\in\sI$, and let $(S^\La,\Phi_{E_{\La}})$ be a configuration space with transition structure
for a pair $(S,\phi)$.   Recall that a path in $(S^\La,\Phi_{E_{\La}})$ is a sequence of transitions
$\vec\gamma=(\vp^1,\ldots,\vp^N)$ in $\Phi_{E_{\La}}$ such that $t(\vp^i)=o(\vp^{i+1})$
for any integer $0<i<N$.     For any form $\omega\in C^1(S^\La)$, we define the integration
of $\omega$ with respect to the path $\vec\gamma$ by
\[
	\int_{\vec\gamma}\omega\coloneqq\sum_{i=1}^N\omega(\vp^i).
\]
We say that a path $\vec\gamma$ is \emph{closed}, if 
$t(\vec\gamma)=o(\vec\gamma)$, where $o(\vec\gamma)=o(\vp^1)$ and $t(\vec\gamma)=t(\vp^N)$.
As in \cite{BKS20}*{Definition 2.14}, we define the closed forms on $(S^\La,\Phi_{E_{\La}})$ as follows.

\begin{definition}
	We say that a form $\omega\in C^1(S^\La)$ is \textit{closed}, if 
	\[
		\int_{\vec\gamma}\omega=0
	\]
	for any closed path $\vec\gamma$ in $(S^\La,\Phi_{E_{\La}})$.
\end{definition}

We denote by $Z^1(S^\La)$ the $\R$-linear space of closed forms in $C(S^\La)$.
The notion of closed forms is compatible with the projections $\pi^\La$.

\begin{lemma}
	The projections $\pi^\La$ induce $\R$-linear homomorphisms
	\[
		\pi^\La\colon Z^1(S^{\La'})\rightarrow Z^1(S^\La)
	\]
	for any $\La,\La '$ in $\sI$ such that $\La\subset\La'$.
\end{lemma}

\begin{proof}
	This is \cite{BS21}*{Lemma 3.10}, since a product measure is \emph{ordinary}
	by \cite{BS21}*{Lemma 3.4}.
\end{proof}

\begin{definition}
	We define the space of \emph{closed co-local forms} $Z^1_\col(S^X_\mu)$ by
	\[
		Z^1_\col(S^X_\mu)\coloneqq\varprojlim_{\La\in\sI}Z^1(S^\La),
	\]
	where the limit is the projective limit with respect to the projections $\pi^\La$.
\end{definition}

By \cite{BKS20}*{Lemma 2.14}, we see that for any $\La\in\sI$ and $f\in C(S^\La)$, we have
$\partial_\La f\in Z^1(S^\La)$, where $\partial_\La$ is the differential defined in
\eqref{eq: differential}.  Moreover, by \cite{BKS20}*{Lemma 2.15}, 
the differential $\partial_\La\colon C(S^\La)\rightarrow Z^1(S^\La)$ is surjective
and induces the isomorphism
\begin{equation}\label{eq: SES}
	C(S^\La)/\Ker\partial_\La\cong Z^1(S^\La).
\end{equation}
The compatibility of the differential $\partial_\La$ with the projection $\pi^\La$ gives the following.

\begin{proposition}\label{prop: SES}
	The differential $\partial\colon	C^0_\col(S^X_\mu)\rightarrow C^1_\col(S^X_\mu)$ 
	of Proposition \ref{prop: compatible}
	induces the isomorphism
	\[
		C^0_\col(S^X_\mu)/H^0_\col(S^X_\mu)\cong Z^1_\col(S^X_\mu),
	\]
	where $H^0_\col(S^X_\mu)\coloneqq\varprojlim_{\La\in\sI}\Ker\partial_\La$.
\end{proposition}

\begin{proof}
	This is \cite{BS21}*{Proposition 3.13},
	noting that a product measure is \emph{ordinary} by \cite{BS21}*{Lemma 3.4}.
	The proof follows from the fact that the projective system $\{\Ker\partial_\La\}$ satisfies
	the Mittag-Leffler condition, which follows immediately from the fact that
	 $\Ker\partial_\La$ is finite dimensional for any $\La\in\sI$, since
	we have assumed that $S$ is finite.
\end{proof}

Next, we consider the differential of uniform functions.
Again, fix a base state $*\in S$.
We let $C^1(S^X_*)\coloneqq\Map^\alt(\Phi^*_E,\R)$, where $\Phi^*_E\coloneqq\Phi_E\cap(S^X_*\cap S^X_*)$.
Similarly to the case of \eqref{eq: inclusion}, there exists an embedding
\begin{equation}\label{eq: inclusion2}
	C^1(S^X_*)\hookrightarrow\prod_{e\in E}C(S^X_*)
\end{equation}
mapping $\omega\in C^1(S^X_*)$ to $(\omega_e)_{e\in E}$, where $\omega_e(\e)\coloneqq\omega((\e,\e^e))$
for any $e\in E$ and $\e\in C(S^X_*)$.  
We define the differential
\[
	\partial\colon C^0(S^X_*)\rightarrow C^1(S^X_*)
\]	
by $\partial f(\vp)\coloneqq f(t(\vp))-f(o(\vp))$.  Again, we have $(\partial f)_e=\nabe f$
through the inclusion \eqref{eq: inclusion2},
where $\nabe f(\e)=f(\e^e)-f(\e)$ for any $\e\in S^X_*$.

We next define the notion of uniform forms.
For any $x\in X$ and $R \geq 0$, 
we let $B(x,R)\coloneqq\{x'\in X\mid d_X(x,x')\leq R\}$ be the ball with center $x$ and radius $R \geq 0$.
Following \cite{BKS20}*{Definition 3.10}, for any $R>0$, we let
\begin{align*}
	C^1_R(S^X)&\coloneqq C^1_\col(S^X_\mu)\cap\Bigl(\prod_{e\in E}C\bigl(S^{B(o(e),R)}\bigr)\Bigr).
\end{align*}

\begin{definition}\label{def: Uniform Form}
	We define the space of \emph{uniform forms} by
	\begin{align*}
		C^1_\unif(S^X)&\coloneqq\bigcup_{R>0} C^1_R(S^X) \subset C^1_\col(S^X_\mu).
	\end{align*}
\end{definition}
Since the space of local functions $C(S^{B(o(e),R)})$ is also a subset of $C(S^X_*)$,
we may also view $C^1_\unif(S^X)$ as a subspace of $C^1(S^X_*)$.
This space is defined independently of the choice of the probability measure $\mu$.

\begin{lemma}\label{lem: UL}
	If $f\in C^0_\unif(S^X_*)$, then we have $\partial f\in C^1_\unif(S^X)$.
\end{lemma}

\begin{proof}
	Since $\nabe f^*_\La\neq 0$ only if $\{o(e),t(e)\}\cap\La\neq\emptyset$, 
	\[
		\nabe f=\sum_{\substack{\La\in\sI\\ \{o(e),t(e)\}\cap\La\neq\emptyset}}\nabe f^*_\La.
	\]
	Then, since $f$ is uniform, the right hand side is
	a finite sum and $\nabe f$ is a local function. Moreover, if the diameter of $f$ is $R$, then $\partial f\in C^1_{R+1}(S^X)$, hence
	$\partial f$ is a uniform form.
\end{proof}

\begin{lemma}\label{lem: WD}
	The isomorphism $\sR$ of \eqref{eq: isom} is compatible with the differential.
	In other words, we have a commutative diagram
	\[
		\xymatrix{
			C^0_\unif(S^X_*)\ar[r]^\partial\ar[d]_{\sR} & C^1_\unif(S^X)\ar@{=}[d]\\
			C^0_\unif(S^X_\mu)\ar[r]^\partial& C^1_\unif(S^X).
		}
	\]
\end{lemma}

\begin{proof}
	Since $\partial$ is linear, and since any $f\in C^0_\unif(S^X_*)$ by Proposition \ref{prop: expansion2}
	may be expressed as a sum of $f^*_\La$
	for $\La\in\sI$, 
	it is sufficient to prove our assertion for $f^*_\La\in C(S^\La)$.
	By definition of $\sR$, we have
	\[
		\sR(f^*_\La)=\sum_{\La''\subset\La,\La''\neq\emptyset}(f^*_\La)_{\La''}
		=f^*_\La-E_\mu[f^*_\La].
	\]
	Our assertion now follows from the fact that the differential $\partial$ of a constant is \emph{zero}.
\end{proof}

Lemma \ref{lem: WD} assures that for any uniform function $f\in C^0_\unif(S^X)$ or $C^0_\unif(S^X_\mu)$,
the differential $\partial f$ is well-defined in $C^1_\unif(S^X)$ and does not depend on the representative of $f$.

\begin{definition}
	We define the space of \emph{closed uniform form} by
	\[
		Z^1_\unif(S^X)\coloneqq Z^1_\col(S^X_\mu)\cap C^1_\unif(S^X).
	\]
\end{definition}

%
\subsection{Norms on Functions and Forms}
%
In this subsection, we define norms on the spaces of co-local functions and forms.
Recall that $\mu=\nu^{\otimes X}$ is the product measure on $S^X$.
For any measurable function $f\in C(S^X_\mu)$, we let
\begin{equation}\label{eq: norm}
	\dabs{f}_\mu^2\coloneqq E_\mu[f^2]=\int_{S^X}f^2d\mu.
\end{equation}

\begin{definition}
	The space of $L^2$-functions $L^2(\mu)$ is given as the quotient
	\[
		L^2(\mu)\coloneqq\{f\in C(S^X_\mu)\mid\dabs{f}_\mu<\infty\}/
		\{f\in C(S^X_{\mu})\mid\dabs{f}_\mu=0\}
	\]
	of the space of measurable functions on $S^X$ bounded for the norm $\dabs{\cdot}_\mu$ divided 
	by the subspace of measurable functions with norm \emph{zero}.
\end{definition}

The space $L^2(\mu)$ is a Hilbert space for the inner product 
\begin{equation}\label{eq: inner}
	\pair{f,g}_\mu\coloneqq E_\mu[fg].
\end{equation}
In particular, $L^2(\mu)$ is complete for the topology given by the norm $\dabs{\cdot}_\mu$.
Moreover, any function $f\in L^2(\mu)$ is integrable for the measure $\mu$.
The conditional expectation $\pi^\La$ satisfies the following properties.
\begin{lemma}\label{lem: 1}
	For any $\La\in\sI$,
	the homomorphism $\pi^\La$ is an orthogonal projection of $L^2(\mu)$ to $C(S^\La)$
	with respect to the inner product \eqref{eq: inner}.
\end{lemma}

\begin{proof}
	The assertions follow directly from the definition of the conditional expectation.
	Detailed proofs are given in \cite{BS21}*{Lemma 2.1}.
\end{proof}

For any $f\in L^2(\mu)$, the system $(f^\La)_{\La\in\sI}$ for $f^\La\coloneqq\pi^\La f$ defines
an element in $ C_\col(S^X_\mu)$.  This shows that we have a homomorphism $L^2(\mu)\rightarrow  C_\col(S^X_\mu)$.
The martingale convergence theorem for $L^2$-bounded martingales gives the following.

\begin{theorem}
	Let
	\[
		C_{L^2}(S^X_\mu)
		\coloneqq\{ (f^\La)_{\La\in\sI}\in  C_\col(S^X_\mu)\mid\sup_{\La\in\sI}\dabs{f^\La}_\mu<\infty\}.
	\]
	Then $C_{L^2}(S^X_\mu)$ coincides with the image of $L^2(\mu)$ in $ C_\col(S^X_\mu)$.
\end{theorem}

\begin{proof}
	See \cite{BS21}*{Theorem 5.2} and  \cite{BS21}*{Corollary 5.2} for a detailed proof.
\end{proof}

In what follows, we will identify $L^2(\mu)$ with $C_{L^2}(S^X_\mu)$.
We next introduce norms on the space of forms $C^1_\col(S^X_\mu)$.
We let $r=(r_e)_{e\in E}$ be a weight as in Definition \ref{def: weight}.
For any $e\in E$, we define a weighted $L^2$-norm for $f$ in $C(S^X_{\mu})$ as follows.
\begin{equation}\label{eq: norm2}
	\dabs{f}_{r_e}^2\coloneqq E_\mu[r_e f^2]=\int_{S^X}r_e f^2d\mu.
\end{equation}
If the weight $r=(r_e)$ is the \emph{trivial weight} satisfying 
$r_e\equiv1$ for any $e\in E$, then the norm $\dabs{\cdot}_{r_e}$
coincides with the norm $\dabs{\cdot}_\mu$ of \eqref{eq: norm}.

\begin{definition}
	We define the space of $L^2$-functions $L^2(\mu)_{r_e}$ to be the quotient
	\[
		L^2(\mu)_{r_e}\coloneqq\{f\in C(S^X_\mu)\mid\dabs{f}_{r_e}<\infty\}/\{f\in C(S^X_{\mu})\mid\dabs{f}_{r_e}=0\}.
	\]
\end{definition}

The space $L^2(\mu)_{r_e}$ is a Hilbert space for the inner product 
\begin{equation}\label{eq: inner2}
	\pair{f,g}_{r_e}\coloneqq E_\mu[r_efg].
\end{equation}
Since $r_e$ is a positive local function, it is bounded from above and below by some positive constants.  Hence the norms $\dabs{\cdot}_{r_e}$ 
and $\dabs{\cdot}_\mu$ are equivalent,
and we have $L^2(\mu)_{r_e}=L^2(\mu)$.
Again using the inclusion $C^1_\col(S^X_\mu)\hookrightarrow\prod_{e\in E} C_\col(S^X_\mu)$
of \eqref{eq: co-local inclusion}, 
we may define the space of $L^2$-forms as follows.
\begin{definition}
	We define the \emph{space of $L^2$-forms} $C^1_{L^2}(S^X_\mu)$ by
	\begin{equation}\label{eq: LF1}
		C^1_{L^2}(S^X_\mu)\coloneqq C^1_\col(S^X_\mu)\cap\Bigl(\prod_{e\in E}L^2(\mu)_{r_e}\Bigr),
	\end{equation}
	which is independent from the choice of the weight $r=(r_e)$.
\end{definition}

The following constants will be used in the calculation of the bound of norms.

\begin{definition}\label{def: TB}
	Let $r=(r_e)$ be a weight for the system $((X,E),(S,\phi))$.  
	  \begin{enumerate}\renewcommand{\labelenumi}{(\alph{enumi})}
		\item For any $e\in E$, we define the \emph{weight bound}
 		to be the infimum $M_{r_e}$ of constants $M\geq 1$ satisfying
		\begin{equation}\label{eq: RB}
	 	M^{-1} \leq r_e(\eta) \leq M
		\end{equation}
		for any $\e\in S^X$.
		Such $M_{r_e}$ exists since $r_e$ is a positive
		local function. 
	
		\item For any $e\in E$, 
		we define the \emph{transition bound}
 		to be the infimum $A_{r_e}$ of constants $A\geq 1$ satisfying
		\begin{equation}\label{eq: TB}
	 		A^{-1}\leq\frac{(\pi^\La r_e)(\e)\mu(\e)}{(\pi^\La r_{\bar e})(\e^e)\mu(\e^e)}\leq A
		\end{equation}
		for any $\La\in \sI$ such that $e\in E_\La$ and $\e\in S^\La$.
		Such $A_{r_e}$ exists since $r_e$ and $r_{\bar e}$ are
		local functions. 
	\end{enumerate}
\end{definition}
\begin{remark}\label{rem: TB}
	\begin{enumerate}\renewcommand{\labelenumi}{(\alph{enumi})}
		\item If the weight $r=(r_e)$ is the trivial weight 
		given by $r_e\equiv 1$ for any $e\in E$, then $A_{r_e}=C_{\phi,\nu}$ where $C_{\phi,\nu}$ is the constant defined in Definition \ref{def: MB}.
		\item For any weight $r=(r_e)$ and $e \in E$, 
		we see from the definition that $A_{r_e} \le M_{r_e}M_{r_{\bar e}} C_{\phi,\nu}$.
	\end{enumerate}
\end{remark}

Let $\La\in\sI$ and $(\La,E_\La)$ be the corresponding finite locale.
For any $f\in C(S^\La)$ and $e\in E_\La$, recall that $\nabe f(\e)=f(\e^e)-f(\e)$ for any $\e\in S^\La$.
We may prove the following.

\begin{lemma}\label{lem: MPL0}
	Let $\La\in\sI$ and $e\in E_\La$.
	Then for any $f\in C(S^\La)$, we have
	\[
		A_{r_{\bar e}}^{-1}\dabs{\nabla_{\!\bar{e}} f}_{r_{\bar e}}^2\leq \dabs{\nabe f}^2_{r_e}\leq 
		A_{r_e}\dabs{\nabla_{\!\bar{e}} f}^2_{r_{\bar e}}.
	\]
\end{lemma}

\begin{proof}
	We have
	\begin{align*}
		\dabs{\nabe f}^2_{r_e}&=\sum_{\e\in S^\La} (f(\e^e)-f(\e))^2(\pi^\La r_e)(\e)\mu(\e)=\sum_{\substack{\e\in S^\La\\\e^e\neq \e}} (f(\e^e)-f(\e))^2 (\pi^\La r_e)(\e)\mu(\e)
		\\&=\sum_{\substack{\e\in S^\La\\\e^e\neq (\e^e)^{\bar e}, \e^e \neq \e}} (f(\e^e)-f((\e^e)^{\bar e}))^2 (\pi^\La r_e) (\e)\mu(\e)\\
		& = \sum_{\substack{\e'\in S^\La\\\e'\neq (\e')^{\bar e}}} (f(\e')-f((\e')^{\bar e}))^2 (\pi^\La r_e) ((\e')^{\bar e})\mu((\e')^{\bar e})\\
		&\leq A_{r_e}\sum_{\e'\in S^\La} (f(\e')-f((\e')^{\bar e}))^2(\pi^\La r_{\bar e})(\e') \mu(\e')
		=A_{r_e}\dabs{\nabla_{\bar e} f}^2_{r_{\bar e}}.
	\end{align*}
	Here, we use the fact that the map $\e \to \e^e$ is the bijection from $\{\e \in S^{\La} ; \e^e \neq \e\}$ to $\{ \e' \in S^{\La} ; \e' \neq (\e')^{\bar{e}} \}$. This gives the second inequality of the statement.  The first inequality follows by replacing $e$ by $\bar e$.
\end{proof}

We may prove the continuity of $\nabe$ as follows.

\begin{lemma}\label{lem: bound}
	For any $\La\in\sI$, $f\in C(S^\La)$ and $e\in E_\La$,  we have
	\[
		\dabs{\nabe f}^2_{r_e}\leq 4M_{r_e}C_{\phi,\nu}\dabs{f}^2_\mu.
	\]
	In particular, $\nabe$ gives a continuous map from $L^2(\mu)$ to $L^2(\mu)_{r_e}$.
\end{lemma}

\begin{proof}
	By definition, we have
	\begin{align*}
		\dabs{\nabe f}^2_{r_e}&=\sum_{\substack{\e\in S^\La\\\e^e\neq\e}} 
		\bigl(f(\e^e)-f(\e)\bigr)^2\pi^\La r_e(\e)\mu(\e)
		\leq M_{r_e}\sum_{\substack{\e\in S^\La\\\e^e\neq\e}} 
		\bigl(f(\e^e)-f(\e)\bigr)^2\mu(\e) \\
		&\leq 2M_{r_e} \sum_{\substack{\e\in S^\La\\\e^e\neq\e}}\bigl(f(\e^e)^2\mu(\e)+f(\e)^2\mu(\e)\bigr)\\
	     &\leq 2M_{r_e}C_{\phi,\nu} \sum_{\substack{\e\in S^\La\\\e^e\neq\e}}\bigl(f(\e^e)^2\mu(\e^e)+f(\e)^2\mu(\e)\bigr)\\
		&=  2M_{r_e}C_{\phi,\nu}\Bigl(\sum_{\substack{\e\in S^\La\\\e^{\bar e}\neq\e}} f(\e)^2\mu(\e)+\sum_{\substack{\e\in S^\La\\\e^e\neq\e}}f(\e)^2\mu(\e)\Bigr)
		\leq 4M_{r_e}C_{\phi,\nu}\dabs{f}^2_\mu,
	\end{align*}
	where the first inequality follows from the definition of the weight bound $M_{r_e}$
	given in Definition \ref{def: TB}, the second inequality follows from Schwartz inequality, and the third inequality follows from the definition of $C_{\phi,\nu}$ given in Definition \ref{def: MB}.
\end{proof}

Next, we introduce the notion of the \emph{spectral gap},
which plays an important role in the assumption for our main theorem.
Let  $(\La,E_\La)$ be a finite locale, with interaction $(S,\phi)$
and a product measure $\mu=\nu^{\otimes\La}$.
For any $\La\in\sI$, we denote by $\dabs{\cdot}_{\La}$ the norm on
\[
	A^0(S^\La)\coloneqq C(S^\La)/\Ker\partial_\La
\]
induced from the norm $\dabs{\cdot}_\mu$ on $C(S^\La)$ as $\dabs{f}_{\La}=\inf_{g \in \Ker\partial_\La} \dabs{f-g}_\mu$.
In general, $\dabs{h}_\La\leq\dabs{h}_\mu$, and we have $\dabs{h}_\La=\dabs{h}_\mu$
if and only if $h\in(\Ker\partial_\La)^\perp$, where $(\Ker\partial_\La)^\perp$
is the orthogonal complement of $\Ker\partial_\La$ for the inner product \eqref{eq: inner}.

To define the spectral gap, we first give a simple lemma.

\begin{lemma}\label{lem: fsg}
	For any finite locale $(\La, E_{\Lambda})$, there exists a constant $C >0$ which 
	depends on the locale $(\La, E_{\La})$ and the weight $r=(r_e)$ such that for any $f\in A^0(S^{\Lambda})$,
	\[
		\dabs{f}^2_\La \le C \sum_{e \in E_{\La}}\dabs{\nabe f}^2_{r_e}.
	\]
\end{lemma}
\begin{proof}
Since $(\Ker\partial_\La)^\perp$ is a finite dimensional space and the quadratic form $f \mapsto \sum_{e \in E_{\La}}\dabs{\nabe f}^2_{r_e}$ does not degenerate on $(\Ker\partial_\La)^\perp$, we have
\[
 \inf_{f \in (\Ker\partial_\La)^\perp, \dabs{f}^2_\mu=1} \sum_{e \in E_{\La}}\dabs{\nabe f}^2_{r_e} >0.
\]
This implies the assertion.
\end{proof}

To state our assumption for the main theorem, we only need to introduce the notion of the spectral gap associated to the trivial weight $r_e \equiv 1$ as follows. 

\begin{definition}\label{def: fsg}
	For an interaction $(S,\phi)$, a finite locale $(\Lambda, E_{\Lambda})$ and a probability measure $\nu$ on $S$ supported on $S$, we define the spectral gap $C_{SG, (\Lambda, E_{\Lambda})}$ (which also depend on $(S,\phi)$ and $\nu$) to be the maximum of $C>0$ satisfying 
	\[
		\dabs{f}^2_\La \le C^{-1} \sum_{e \in E_{\La}}\dabs{\nabe f}^2_{\mu}
	\]
	for any $f\in A^0(S^{\La})$ where $\mu=\nu^{\otimes \La}$. 
\end{definition}

Finally, we introduce a norm on the space of $L^2$-forms $C^1_{L^2}(S^X_\mu)$.

\begin{definition}\label{def: r}
	We define the norm $\dabs{\cdot}_{r,\sp}$ to be the norm on $C^1_{L^2}(S^X_\mu)$ defined by
	\[
		\dabs{\omega}_{r,\sp}\coloneqq\sup_{e\in E}\dabs{\omega_e}_{r_e},
	\]
	and we let
	$
		C^1_{L^2}(S^X_\mu)_{r,\sp}
	$
	be the subspace of $C^1_{L^2}(S^X_\mu)$ given as
	\[
		C^1_{L^2}(S^X_\mu)_{r,\sp}\coloneqq\{\omega\in  C^1_{L^2}(S^X_\mu)\mid \dabs{\omega}_{r,\sp}<\infty\}.
	\]
	Furthermore, we let $Z^1_{L^2}(S^X_\mu)_{r,\sp}\coloneqq C^1_{L^2}(S^X_\mu)_{r,\sp} \cap Z^1_\col(S^X_\mu)$.
\end{definition}

We remark that although $C^1_\unif(S^X)\subset C^1_{L^2}(S^X_\mu)$, we have in general
\[
	C^1_\unif(S^X)\not\subset C^1_{L^2}(S^X_\mu)_{r,\sp}.
\]
Certain homogeneity condition such as shift-invariance of the weight as well as the form is 
necessary for an uniform form to be an element of  $C^1_{L^2}(S^X_\mu)_{r,\sp}$.
We will simply denote $\dabs{\cdot}_{r,\sp}$ by $\dabs{\cdot}_\sp$, if the weight $r$
is the trivial weight such that $r_e\equiv 1$ for any $e\in E$.

%
%
%
\section{The Main Theorem}
%
%
%

In this section, we introduce the norms on shift-invariant forms, and state our main theorem.

%
%
\subsection{Closed and Exact Shift-Invariant $L^2$-Forms}\label{subsec: CE}
%
%

In this subsection, in order to formulate our main theorem, we introduce the space 
of closed and exact  shift-invariant $L^2$-forms on $S^X$.
We will also introduce the notion of a uniformly bounded spectral gap.

In this subsection, we assume that the locale $(X,E)$ has an action of a group $G$.
An action of $G$ on $(X,E)$ gives for any $\tau\in G$ a bijection $\tau\colon X\rightarrow X$ 
inducing a bijection $\tau\colon E\rightarrow E$ on the set of directed edges,
such that the identity element of $G$ induces the identity map of $X$, 
and the operation of elements of $G$ maps to the composition of bijections.
In the rest of the paper, we always assume that the action of $G$ is free, and that the set $X/G$ of orbits 
of $X$ with respect to the action of $G$ is a finite set.
In this case,  the set $E/G$ of orbits 
of $E$ with respect to the action of $G$ is also a finite set.
The most typical example of such locale is given by the \emph{Euclidean lattice} $(\Z^d,\E^d)$,
with action of $G=\Z^d$ given by translation.  A locale which is a \emph{topological crystal}
in the sense of \cite{Sun13}*{\S 6.3}, also referred to as a \emph{crystal lattice}, satisfies this condition for some finitely generated free abelian group $G$. 

The action of $G$ on $(X,E)$ induces an action of $G$ on
$S^X$ given by mapping $\e=(\eta_x)_{x\in X}\in\prod_{x\in X}S$ to 
$\tau(\e)\coloneqq(\eta_{\tau^{-1}(x)})_{x\in X}$ for any $\tau\in G$.
For any subset $\La\subset X$, the element $\tau\in G$ induces 
an $\R$-linear isomorphism
\begin{equation}\label{eq: function}
	\tau\colon C\bigl(S^\La\bigr)\xrightarrow\cong C\bigl(S^{\tau(\La)}\bigr)
\end{equation}
given by
$
	\tau(f)(\e)=f(\tau^{-1}(\e))
$
for any $f\in C\bigl(S^\La\bigr)$ and $\e\in S^{\tau(\La)}$.  
Furthermore, the action of $G$ defines a map of graphs 
$\tau\colon (S^\La,\Phi_\La)\rightarrow (S^{\tau(\La)},\Phi_{\tau(\La)})$,
which induces an $\R$-linear isomorphism
\[
	\tau\colon C^1(S^\La)\xrightarrow\cong C^1(S^{\tau(\La)}),
\] 
given by $\tau(\omega)(\vp)\coloneqq\omega(\tau^{-1}(\vp))$ for any $\omega\in C^1(S^\La)$ and
$\vp\in\Phi_\La$.
Consider the image of $\omega$ in $\prod_{e\in E_{\La}}C(S^\La)$ through the embedding
\eqref{eq: inclusion}.
For any $\tau\in G$, we have
\[
	\tau(\omega)_e(\e)=\tau(\omega)(\vp)=\omega(\tau^{-1}(\vp))=\omega_{\tau^{-1}(e)}(\tau^{-1}(\e))
	=\tau(\omega_{\tau^{-1}(e)})(\e),
\]
where $\vp=(\e,\e^e)\in\Phi_\La$.
This shows that we have $\tau(\omega)_e=\tau(\omega_{\tau^{-1}(e)})$ for any $e\in E_{\La}$.

Since $\mu=\nu^{\otimes X}$ on $S^X$,
we have
$\mu(A)=\mu(\tau(A))$ for any $A\in\cF_X$ and $\tau\in G$.
In other words, $\mu$ is invariant with respect to the action of $G$.
The action of $G$ is compatible with the projection $\pi^\La$,
hence we have an action of $G$ on $C_\col(S^X_\mu)$ and $C^1_\col(S^X_\mu)$.
Since $\mu$ is invariant with respect to the action of $G$, we have
\[
	\dabs{f}_\mu=\dabs{\tau(f)}_\mu
\]
for any $f\in C_{L^2}(S^X_\mu)$ and $\tau\in G$, hence $G$ also acts on $L^2(\mu)$.

Next, we consider weights which are invariant with respect to the action of $G$.

\begin{definition}
	We say that a weight $r=(r_e)_{e\in E}$ is \emph{invariant with respect to the action of $G$},
	if
	\[
		r_e=\tau(r)_e=\tau(r_{\tau^{-1}(e)}) 
	\]
	for any $\tau\in G$ and $e\in E$.
\end{definition}

Let $\dabs{\cdot}_{r,\sp}$ be the norm on $C^1_{L^2}(S^X_\mu)$ defined in Definition \ref{def: r},
given by
\[
	\dabs{\omega}_{r,\sp}=\sup_{e\in E}\dabs{\omega_e}_{r_e}.
\]
If the weight $r=(r_e)$ is invariant with respect to the action of $G$, then for any $\tau\in G$, we have
\begin{align*}
	\dabs{\tau(\omega)_e}^2_{r_e}&=\dabs{\tau(\omega_{\tau^{-1}(e)})}^2_{r_e}
	=E_\mu[r_e\tau(\omega_{\tau^{-1}(e)})^2]
	=E_\mu[\tau(r_{\tau^{-1}(e)}\omega_{\tau^{-1}(e)}^2)]\\
	&=\dabs{\omega_{\tau^{-1}(e)}}^2_{r_{\tau^{-1}(e)}},
\end{align*}
where the last equality follows from the fact that $\mu$ is invariant with respect to the action of $G$.
This shows that $\dabs{\omega}_{r,\sp}<\infty$ if and only if $\dabs{\tau(\omega)}_{r,\sp}<\infty$,
hence $G$ also acts on the $\R$-linear space $C^1_{L^2}(S^X_\mu)_{r,\sp}\subset C^1_{L^2}(S^X_\mu)$
of Definition \ref{def: r}.

For any $\R$-linear space $V$ with an action of $G$, we let 
$V^G\coloneqq\{ v\in V\mid \tau(v)=v\,\forall\tau\in G\}$ be the shift-invariant subspace of $V$.

\begin{lemma}\label{lem: same}
If a weight $r=(r_e)$ is invariant with respect to the action of $G$, then
	the inclusion $C^1_{L^2}(S^X_\mu)_{r,\sp}\subset C^1_{L^2}(S^X_\mu)$ induces the identity
	\[
		C^1_{L^2}(S^X_\mu)_{r,\sp}^G= C^1_{L^2}(S^X_\mu)^G
	\]
	on the shift-invariant parts.  Moreover, the norms induced by $\dabs{\cdot}_{r,\sp}$ on 
	$C^1_{L^2}(S^X_\mu)^G$ are equivalent for any weight $r=(r_e)$ which is invariant with respect
	to the action of $G$.
\end{lemma}

\begin{proof}
	Let $\omega=(\omega_e)_{e\in E}\in C^1_{L^2}(S^X_\mu)^G$ be any shift-invariant $L^2$-form.
	For any $\tau\in G$, consider $\vp=(\e,\e^e)\in\Phi_X$.
	Since
	\[
		\tau(\omega)_e(\e)=\tau(\omega)(\vp)=\omega(\tau^{-1}(\vp))=\omega_{\tau^{-1}(e)}(\tau^{-1}(\e))
		=\tau(\omega_{\tau^{-1}(e)})(\e),
	\]
	the shift-invariance of $\omega$ shows that $\omega_e=\tau(\omega_{\tau^{-1}(e)})$ for any $e\in E$.
	Then we see that
	\[
		\dabs{\omega_{\tau^{-1}(e)}}^2_{r_{\tau^{-1}(e)}}=E_\mu[r_{\tau^{-1}(e)}(\omega_{\tau^{-1}(e)})^2]
		=E_\mu[\tau(r_{\tau^{-1}(e)}(\omega_{\tau^{-1}(e)})^2)]=\dabs{\omega_e}^2_{r_e},
	\]
	where the center equality follows from the fact that the probability measure
	$\mu$ is invariant with respect to the action of $G$.  This shows that
	\[
		\dabs{\omega}_{r,\sp}=\sup_{e\in E}\dabs{\omega}_{r_e}=\sup_{e\in E_0}\dabs{\omega}_{r_e},
	\]
	where $E_0$ is a set of representatives of the orbits of $E$ with respect to the action of $G$.
	Since $E_0$ is finite, we see that $\omega\in C^1_{L^2}(S^X)_{r,\sp}^G$ as desired.
	The equivalence of norms follows from the fact that the norms $\dabs{\cdot}_{r_e}$ 
	are equivalent to the norm $\dabs{\cdot}_{\mu}$ on $L^2(\mu)$ for any $e\in E_0$.
\end{proof}

If we let $E_0$ be a set of representatives of the orbits of $E$ with respect to the action of $G$,
then the projection to the $E_0$-component induces an injection
\[
	C^1_{L^2}(S^X_\mu)^G\hookrightarrow\prod_{e\in E_0}L^2(\mu)_{r_e}.
\]
The inner product $\pair{\cdot,\cdot}_{r_e}$ on $L^2(\mu)_{r_e}$ of \eqref{eq: inner2}
induces via linearity an inner product $\pair{\cdot,\cdot}_{r}$ on the product,
which in turn induces an inner product on $C^1_{L^2}(S^X_\mu)^G$.
We let $\dabs{\cdot}_{r}$ be the norm on $C^1_{L^2}(S^X_\mu)^G$
induced from the inner product $\pair{\cdot,\cdot}_{r}$, given as
\[
	\dabs{\omega}^2_{r}\coloneqq\pair{\omega,\omega}_{r}=\sum_{e\in E_0}\pair{\omega_e,\omega_e}_{r_e}=\sum_{e\in E_0}\dabs{\omega_e}_{r_e}^2
\]
for any $\omega\in C^1_{L^2}(S^X_\mu)^G$.  
Since $\omega$ is invariant with respect to the action of $G$,
the inner products $\pair{\cdot,\cdot}_{r}$ and the induced norm $\dabs{\cdot}_{r}$
is independent of the choice of $E_0$.
The relation
\[
	\dabs{\omega}^2_{r,\sp} \leq \dabs{\omega}^2_{r} \leq \abs{E_0}\dabs{\omega}^2_{r,\sp}
\]
shows that the norms $\dabs{\cdot}_{r}$ and $\dabs{\cdot}_{r,\sp}$ are equivalent on 
$C^1_{L^2}(S^X_\mu)^G$, hence induces the same topology.
Note that since this topology is the topology of a finite sum of Hilbert spaces $L^2(\mu)_{r_e}$,
for a sequence $\{\omega_n\}_{n\in\N}$ in $C^1_{L^2}(S^X_\mu)^G$, we have
\[
	\omega=\lim_{n\rightarrow\infty}\omega_n
\]
if and only if $\displaystyle\omega_e=\lim_{n\rightarrow\infty}\omega_{n,e}$ in $L^2(\mu)_{r_e}$
for any $e\in E$.

Next, we consider the compatibility of the action of the group with the differential.

\begin{lemma}\label{lem: differential}
	The action of $G$ is compatible with the 
	differential 
	\[
		\partial_\La\colon C(S^\La)\rightarrow C^1(S^\La),
	\]
	namely $\partial_{\tau \La}( \tau f )= \tau ( \partial_{\La} f)$ for any $f \in C(S^\La)$ and $\tau \in G$.
\end{lemma}

\begin{proof}
	This is \cite{BS21}*{Lemma 4.1}, and follows from the definition of the action of $G$.
\end{proof}

Since the action of $G$ is compatible with the projection $\pi^\La$, the differential
\[
	\partial\colon C^0_\col(S^X_\mu)\rightarrow C^1_\col(S^X_\mu)
\]
is also compatible with the action of $G$.
By Proposition \ref{prop: SES},  the space of closed co-local forms
$Z^1_\col(S^X_\mu)$ coincides with the image of $\partial$, hence
we see that $Z^1_\col(S^X_\mu)$ also has an action of $G$.
This induces an action of $G$ on $Z^1_{L^2}(S^X_\mu)=Z^1_\col(S^X_\mu)\cap C^1_{L^2}(S^X_\mu)$.

\begin{definition}\label{def: CE}
	We denote the space of
	\emph{shift-invariant closed $L^2$-forms} by
	\[
		\sC_{\mu}\coloneqq Z^1_{L^2}(S^X_\mu)^G,
	\] 
	and the space of \emph{exact forms} by
	\[
		\sE_{\mu}\coloneqq\ol{\partial(C^0_\unif(S^X)^G)},
	\] 
	where the bar indicates the closure of 
	$\partial(C^0_\unif(S^X)^G)$ with respect to the topology of $C^1_{L^2}(S^X_\mu)^G$.
\end{definition}

Our main theorem concerns the decomposition of forms
in $\sC_{\mu}$ as a sum of an exact form in $\sE_{\mu}$ and certain forms obtained as 
differentials of uniform functions. For this, we first prove that $\sE_\mu\subset\sC_\mu$.

\begin{lemma}\label{lem: closed}
	The space of \emph{shift-invariant closed $L^2$-forms} $\sC_\mu$
	is closed in $\prod_{e\in E/G}L^2(\mu)$.
	In particular, $\sC_\mu$ is a Hilbert space equipped with the inner product induced from $\prod_{e\in E/G}L^2(\mu)$ and we have $\sE_\mu\subset\sC_\mu$.
\end{lemma}
\begin{proof}
	Suppose we have a sequence $\{\omega_n\}_{n\in\N}$ in $\sC_\mu=Z^1_{L^2}(S^X_\mu)^G$
	such that 
	\begin{equation}\label{eq: converge0}
		\lim_{n\rightarrow\infty}\omega_n=\omega
	\end{equation}
	for some element $\omega=(\omega_e)\in\prod_{e\in E}L^2(\mu)$.
	Since the measure $\mu$ is invariant with respect to the action of $G$,
	we see that $\omega$ is also invariant with respect to the action of $G$.
	In order to prove our assertion, by definition of closed co-local forms,
	it is sufficient to prove that $\pi^\La\omega\in Z^1(S^\La)$ for any $\La\in\sI$.
	The convergence \eqref{eq: converge0} is  equivalent to the condition that $\lim_{n\rightarrow\infty}\omega_{n,e}=\omega_e$
	in $L^2(\mu)$ for any $e\in E$, which implies that
	$\lim_{n\rightarrow\infty}\pi^\La\omega_{n,e}=\pi^\La\omega_e$ in $C(S^\La)$
	for any $e\in E$ and $\La\in\sI$.
	Since $\omega_n\in Z^1_{L^2}(S^X_\mu)^G$, for any $\La\in\sI$, we have
	$\pi^\La\omega_n\in Z^1(S^\La)$.
	We may take $F_n^\La\in(\Ker\partial_\La)^\perp$ such that $\partial_\La F_n^\La=\pi^\La\omega_n$.
	Lemma \ref{lem: fsg} and the fact that $\pi^\La\omega_n\rightarrow\pi^\La\omega$
	imply that $F_n^\La$ converges to some $F^\La_\infty$ in $(\Ker\partial_\La)^\perp$ .
	By the continuity of $\partial_\La$ which follows from Lemma \ref{lem: bound},
	we have $\partial_\La F^\La_\infty=\pi^\La\omega$.
	This shows that $\pi^\La\omega\in Z^1(S^\La)$ as desired. Now, to prove $\sE_\mu\subset\sC_\mu$, we only need to show that
	$ \partial(C^0_\unif(S^X)^G) \subset Z^1_{L^2}(S^X_\mu)^G$. Noting $ \partial C^0_\unif(S^X) \subset C^1_\unif(S^X) \cap Z^1_\col(S^X_\mu)$, the assertion follows.
\end{proof}

Next, we introduce the notion of the \emph{uniformly bounded spectral gap},
which is an important assumption for our main theorem.  We stress that we view the
condition of the spectral gap as a condition on the \emph{interaction}, and not a
condition of the underlying \emph{locale}.

We say that a locale $(\La,E_\La)$ is \emph{complete}, if it is complete
as a graph.  In other words, the set of edges satisfies $E_\La=(\La\times\La)\setminus\Delta_\La$,
where $\Delta_\La\coloneqq\{(x,x)\mid x\in \La\}$.
Since a locale by definition is locally finite, a complete locale is always a finite locale.

\begin{definition}\label{def: SG}
	We say that the data $((S,\phi),\nu)$ has a \emph{uniformly bounded spectral gap},
	if there exists a constant $C_{SG}>0$ such that for any complete locale $(\La,E_\La)$
	\[
	 C_{SG, (\La,E_{\La})} \geq C_{SG} |\La|,
	\]
	where $C_{SG, (\La,E_{\La})}$ is the spectral gap for $(\Lambda, E_{\Lambda})$ given in Definition \ref{def: fsg}.
\end{definition}

\begin{example}
	Below are examples of interactions that are known to have a \emph{uniformly bounded spectral gap}.
	\begin{enumerate}\renewcommand{\labelenumi}{(\alph{enumi})}
		\item Multi-species exclusion process introduced in Example \ref{example: interaction} with any probability measure $\nu$ supported on 
		$S$  has a uniformly bounded spectral gap, as shown in \cite{CLR10}.
		\item Generalized exclusion process introduced in Example \ref{example: interaction} with any probability measure $\nu$ supported on $S$ has a uniformly bounded spectral gap. The spectral gap estimate for a special choice of measure $\nu$ is given in \cite{KLO94, KL99}. For general $\nu$, it is shown in Section 3 of \cite{Ca04}. 
		\item More generally, Theorem 2.1 of \cite{Ca08} and Theorem 3 of \cite{Sas13} 
		 is applicable for our model. Hence, if the spectral gap for the corresponding simple averaging dynamics of $(S,\nu)$ (which is called a binary collision process in \cite{Ca08}) on the complete graph with $3$ sites is sufficiently large, 
		 then we can conclude that the interaction $((S,\phi), \nu)$ has a uniformly bounded spectral gap.
	\end{enumerate}
\end{example}

%
\subsection{The Statement of the Main Theorem}\label{subsec: MT}
%

In this subsection, we state our main theorem.
Consider a finite range Euclidean lattice $(\Z^d,E)$.
We fix an interaction $(S,\phi)$ 
which is irreducibly quantified in the sense of Definition \ref{def: FQ}.  We will also assume that $\mu$ is a product measure $\mu=\nu^{\otimes X}$ such that
$((S,\phi),\nu)$ has a uniformly bounded spectral gap in the sense
of Definition \ref{def: SG}.
We let $r=(r_e)_{e\in E}$ be a weight for the system $((\Z^d,E),(S,\phi))$
as introduced in Definition \ref{def: weight}.

By definition,
the finite range Euclidean lattice $(\Z^d,E)$ has a free action of the group $G=\Z^d$
acting via translation.  For any $x\in\Z^d$, we denote by $\tau_x\in G$ the translation by $x$,
given by $\tau_x(y)=x+y$ for any $y\in\Z^d$. 
Note that since $\mu$ is the product measure on $S^{\Z^d}$, it is 
invariant with respect to the action of $G$.
We assume that the rate is invariant with respect to 
the action of $G$, in other words that we have $r_e=\tau(r_{\tau^{-1}(e)})$ for any $e\in E$ and $\tau\in G$.

We let $\sC_\mu$ and $\sE_\mu$ be the space of closed and exact shift-invariant $L^2$-forms on $S^{\Z^d}$
given in Definition \ref{def: CE}.  In other words, we let $\sC_\mu\coloneqq Z^1_{L^2}(S^{\Z^d}_\mu)^G$
and 
\[
	\sE_\mu\coloneqq\overline{\partial(C^0_\unif(S^{\Z^d})^G)}\subset\sC_\mu,
\]
where the inclusion follows from Lemma \ref{lem: closed}.

\begin{lemma}\label{lem: tau}
	For any conserved quantity $\xi\in\C(S)$ and $j=1,\ldots, d$, let $\frA^j_\xi$ be the uniform function given by
	\begin{equation}\label{eq: A}
		\frA^j_\xi = \sum_{x\in \Z^d} x_j\xi_x, 	
	\end{equation}
	where $x=(x_1,\ldots,x_d)\in\Z^d$.  Then for any $y=(y_1,\ldots,y_d)\in\Z^d$,
	we have 
	\[
		(1-\tau_y)\frA^j_\xi=y_j\sum_{x\in\Z^d}\xi_x.
	\]
	In particular,	
	we have $\partial	\frA^j_\xi \in Z^1_\unif(S^{\Z^d})^G \subset \sC_{\mu}$.
\end{lemma}

\begin{proof}
	By definition, $\frA^j_\xi=\sum_{x\in \Z^d} x_j\xi_x$ is a uniform function with diameter $R=0$,
	hence
	\[
		\partial\frA^j_\xi\in Z^1_\unif(S^{\Z^d})=Z^1_\col(S^{\Z^d}_\mu)\cap C^1_\unif(S^{\Z^d})
	\] 
	by Lemma \ref{lem: UL}.
	For any $y=(y_1,\ldots,y_d)\in\Z^d$, we have
	\begin{align*}
		(1-\tau_y)\frA^j_\xi& =\sum_{x\in\Z^d} x_j\xi_x-\sum_{x\in\Z^d} x_j\xi_{x+y}
		=\sum_{x\in\Z^d} x_j\xi_x-\sum_{x\in\Z^d} (x_j-y_j)\xi_x
		=y_j\sum_{x\in\Z^d}\xi_x 
	\\&= y_j\xi_{\Z^d}
	\end{align*}
	for $j=1,\ldots, d$. Note that the infinite sums are all well-defined. Since $\partial\xi_{\Z^d}=0$, the compatibility of the group action with respect
	to the differential gives $(1-\tau_y)\partial \frA^j_\xi=0$.  Our assertion follows from the fact
	that any $\tau\in G$ is of the form $\tau=\tau_y$ for some $y\in\Z^d$.
\end{proof}

We let $c_\phi=\dim_\R\C(S)$, and we fix a basis $\xi^{(1)},\ldots,\xi^{(c_\phi)}$ of $\C(S)$.
Denote by 
\[
	\cV\coloneqq\Span_\R\{  \frA^j_{\xi^{(i)}}\mid i=1,\ldots,c_\phi, j=1,\ldots, d \}
\] 
the $\R$-linear subspace
of $C^0_\unif(S^{\Z^d})$ spanned by $\frA^j_{\xi^{(i)}}$ for $i=1,\ldots, c_\phi$ and $j=1,\ldots, d$.
Our main result, Varadhan's decomposition of shift-invariant closed $L^2$-forms, is given as follows.
\begin{theorem}\label{thm: main}
	Let $(\Z^d,E)$ be the lattice with action of $G=\Z^d$ by translation.
	We assume that $(S,\phi)$ is an interaction which is irreducibly quantified,
	and that $((S,\phi),\nu)$ has a uniformly bounded spectral gap in the sense of Definition \ref{def: SG}.  
	We assume in addition that $\phi$ is simple if $d=1$.
	Furthermore, we consider a weight $r=(r_e)_{e\in E}$ invariant with respect to the action of $G$.
	If we let $\sC_{\mu}$ and $\sE_{\mu}$ be the spaces of closed and exact shift-invariant $L^2$-forms,
	then we have a decomposition
	\[
		\sC_{\mu}\cong\sE_{\mu}\oplus\partial\cV
	\]
	where $\partial\cV$ is an $\R$-linear space of dimension $c_\phi d$ with basis $\partial\frA^j_{\xi^{(i)}}$
	for $i=1,\ldots,c_\phi$ and $j=1,\ldots,d$.
\end{theorem}

We remark that we do not need to assume that the weight is reversible for the measure $\mu$. Actually, we prove the result by reducing the proof to the trivial weight $r_{e} \equiv 1$. 
Theorem \ref{thm: main} for the case of the exclusion process
was proved in \cite{FUY96}*{Theorem 4.1}, for the generalized exclusion process
in \cite{KL99}*{Theorem 4.14} with a specific choice of $\nu$, for the two-species exclusion process with one-conserved quantity in \cite{Sas10}, and for the lattice gas with energy in \cite{N03}. Our result generalizes such results
to general irreducibly quantified interactions satisfying the spectral gap estimate.
In particular, our result holds for the case of the multi-species exclusion process
whose space of conserved quantities is multi-dimensional and its dimension can be arbitrary large.

We may recover \cref{thm: intro} from \cref{thm: main} as follows.
For any shift-invariant closed $L^2$-form $\omega=(\omega_e)_{e\in E}\in Z^1_{L^2}(S^{\Z^d}_\mu)^G=\sC_\mu$, \cref{thm: main} shows that we have a decomposition 
\[
	\omega=\omega'+\omega'', \qquad \omega'\in\sE_\mu, \quad\omega''\in\partial\cV.
\]
By definition of $\sE_\mu$, there exists a series of uniform functions $\{F_n\}_{n\in\N}$ in $C^0_\unif(S^{\Z^d})^G$
such that $\displaystyle\lim_{n\rightarrow\infty}\partial F_n=\omega'$ for the topology of $\sC_\mu$.
By \cite{BKS20}*{Lemma 5.15}, any $F_n\in C^0_\unif(S^X)^G$ is of the form $F_n=\sum_{\tau\in G}\tau(f_n)=\sum_{x\in\Z^d}\tau_x(f_n)$
for some local function $f_n$.  Furthermore, by definition of $\partial\cV$,
we see that
\[
	\omega''=\sum_{i=1}^{c_\phi}\sum_{j=1}^d a_{ij}\partial\frA^j_{\xi^{(i)}}
\]
for some $a_{ij}\in\R$ for $i=1,\ldots,c_\phi$ and $j=1,\ldots,d$.
This shows that we have
\begin{equation}\label{eq: total}
	\omega=\lim_{n\rightarrow\infty}\partial\Bigl(\sum_{x\in\Z^d}\tau_x(f_n)+\sum_{i=1}^{c_\phi}\sum_{j=1}^d a_{ij}\frA^j_{\xi^{(i)}}\Bigr).
\end{equation}
Equation \eqref{eq: MT} of \cref{thm: intro} now follows from the definition \eqref{eq: A} of $\frA^j_{\xi^{(i)}}$,
by taking the $e$-component of \eqref{eq: total} for any $e\in E$. The uniqueness of $\{a_{ij}\}$ follows from the fact that $\{ \partial\frA^j_{\xi^{(i)}} \}$ is a basis of 
$\partial\cV$.

Note by Lemma \ref{lem: same}, the topology hence the closure is independent 
of the choice of the various norms $\dabs{\cdot}_{r}$, $\dabs{\cdot}_{r,\sp}$
on $C^1_{L^2}(S^{\Z^d}_\mu)^G$.  Hence it is sufficient to prove the theorem for the case of trivial weight,
such that $r_e\equiv1$ for any $e\in E$.
The main strategy is to reduce the proof to the case of the Euclidean lattice $\Z^d=(\Z^d,\bbE^d)$.
The latter case will be proved roughly as follows.
Assume the conditions of Theorem \ref{thm: main}.
Given a closed shift-invariant $L^2$-form $\omega\in\sC_\mu=Z^1_{L^2}(S^{\Z^d}_\mu)^G$,
we follow the strategy originally due to Varadhan, and construct from $\omega$ certain sequence of shift-invariant
uniform functions $\{\Psi_n\}_{n\in\N}$ in $C^0_\unif(S^{\Z^d})^G$ 
such that we have
\[
	\lim_{n\rightarrow\infty}\partial\Psi_n = \omega + \omega^\dagger
\]
for some $\omega^\dagger\in\cC\coloneqq Z^1_{\unif}(S^{\Z^d})^G$.  
In the original strategy by Varadhan, the key point was to 
use the explicit description of the interaction $(S,\phi)$ and obtain explicit conditions which must be satisfied by $\omega^\dagger$ to show that $\omega^\dagger$ may be described in terms
of linear sums of the differential of $\frA^j_{\xi^{(i)}}$.   
Once we have this, if we let $\displaystyle \omega_\psi\coloneqq\lim_{n\rightarrow\infty}\partial\Psi_n$, then
this would show that $\omega$ splits as
\begin{equation}\label{eq: split}
	\omega =\omega_\psi + (-\omega^\dagger)  \in \sE_\mu\oplus\partial\cV.
\end{equation}
However, in our case of a general interaction, especially for such case of the multi-species
exclusion process when $c_\phi \gg 1$, such explicit calculation
are too much complicated to follow through via brute calculation. Also, these direct computations do not explain the reason why we have the differential of $\frA^j_{\xi^{(i)}}$ universally. 
Instead, we will use a general decomposition theorem
for $\cC$, which was proved in our previous article \cite{BKS20}.
We let $\cC= Z^1_\unif(S^{\Z^d})^G$ and $\cE\coloneqq\partial(C^0_\unif(S^{\Z^d})^G)$
be the spaces of closed and exact shift-invariant \emph{uniform forms}.  
The main result of \cite{BKS20} is the following.

\begin{theorem}[\cite{BKS20}*{Theorem 5}]\label{thm: previous}
	Let the assumptions be as in Theorem \ref{thm: main}.	
	If we let $\cC$ and $\cE$ be the spaces of closed and exact shift-invariant uniform forms,
	then we have a decomposition
	\[
		\cC\cong\cE\oplus\partial\cV,
	\]
	where $\partial\cV$ is an $\R$-linear space of dimension $c_\phi d$ with basis $\partial\frA^j_{\xi^{(i)}}$
	for $i=1,\ldots,c_\phi$ and $j=1,\ldots,d$.
\end{theorem}

By Theorem \ref{thm: previous}, we see that
the closed form $\omega^\dagger\in\cC$ of \eqref{eq: split} may
be decomposed as
\[
	\omega^\dagger=\omega' + \omega^\ddagger\in\cE\oplus\partial\cV.
\]
Since $\cE\subset\sE_\mu$, we have $\omega_\psi-\omega'\in\sE_\mu$.  This would give a decomposition
\[
	\omega= (\omega_\psi-\omega')+ (-\omega^\ddagger)\in \sE_\mu
	\oplus\partial\cV
\]
as desired.  

The rest of this article is devoted to formalizing the above argument as follows. In the next section, we prove some key lemmas for our proof, which hold for any locale $(X,E)$. In Section \ref{sec: boundary}, we construct a sequence of shift-invariant
uniform functions $\{\Psi_n\}_{n\in\N}$ appeared in the above strategy for the case $(X,E)=(\Z^d,\mathbb{E}^d)$. Finally, in Section \ref{sec: proof}, we give a proof of \cref{thm: main}. More precisely, in Subsection \ref{subsec : Euclidean}, we prove \cref{thm: main} for the Euclidean lattice $(X,E)=(\Z^d,\mathbb{E}^d)$, and in Subsection \ref{subsec: PMT}, by reducing the problem for a general finite range Euclidean lattice $(\Z^d, E)$ to that for the Euclidean lattice $(\Z^d,\mathbb{E}^d)$, we give a proof of \cref{thm: main} in general.

\section{Key lemmas}\label{sec: key-lemmas}

In this section, we prove some general results concerning norms on spaces of local functions and forms, and also a criterion of locality. 
More precisely, assuming that the weight $r=(r_e)_{e\in E}$ is trivial,
we prove the \emph{Moving Particle Lemma} (Lemma \ref{lem: MPL})
and the \emph{Boundary Estimate} (Proposition \ref{prop: BE}), 
which gives the relation between certain differentials and operators on local functions.

The result in this section holds for any locale $(X,E)$, so we consider any general locale. We fix a data $((X,E),(S,\phi))$ and a probability measure $\mu=\nu^{\otimes X}$ on $S^X$,
where $\nu$ is a probability measure fully supported on $S$.
We assume that $\phi$ is \emph{irreducibly quantified} in the sense of Definition \ref{def: FQ}.

%
%
\subsection{Moving Particle Lemma}
%
%

In this subsection, we prove the \emph{Moving Particle Lemma} (Lemma \ref{lem: MPL}), which 
bounds the change of potential 
induced by interactions between distant vertices 
in terms of the change of potentials induced by interactions on adjacent vertices of the locale.
Let $(X,E)$ be a locale. For any $\e\in S^X$ and $x,y\in X$ such that $x\neq y$, 
consider the configuration
\[
	\e^{x\arr y}\in S^X
\]
whose components outside $x,y\in X$
coincides with that of $\e$, and 
$
	(\eta^{x\arr y}_x, \eta^{x\arr y}_y)\coloneqq\phi(\eta_x,\eta_y).
$
If $x=y$, then we let $\e^{x\arr y}\coloneqq\e$.
For any local function $f \in C_\loc(S^X)$ and $x,y\in X$, we let $f_{x\arr y}$ be the 
local function given by
$f_{x\arr y}(\e)\coloneqq f(\e^{x\arr y})$ for any $\e\in X$.
Furthermore, we let
\[
	\nabla_{\!x\arr y} f \coloneqq f_{x\arr y}-f.
\]
Note that by definition, $\nabla_{\!x\arr y}f=0$ if $x=y$.
The Moving Particle Lemma is the following.

\begin{lemma}[Moving Particle Lemma]\label{lem: MPL}
	There exists a constant $C_{M\!P}>0$ depending only on $((S,\phi), \nu)$ such that for any
	finite locale $(\La,E_\La)$, vertices $x,y\in \La$, and function $f\in C(S^\La)$,
	we have
	\[
		\dabs{\nabla_{\!x\arr y}f}^2_\mu\leq C_{M\!P}\diam{\La}^2\dabs{\partial_\La f}^2_\sp,
	\]
	where
	\[
		\dabs{\omega}_\sp\coloneqq\sup_{e\in E_\La}\dabs{\omega_e}_\mu
	\]
	for any $\omega\in C^1(S^\La)$, which corresponds to the norm $\dabs{\cdot}_{r,\sp}$ when $r$ is 
	the trivial rate.
\end{lemma}

We will give the proof of Lemma \ref{lem: MPL} at the end of this section.
We first define the notion of an exchange.
Let $\e\in S^X$.
For any $x,y\in X$, we define the \emph{exchange}
\[
	\e^{x,y}\in S^X
\]
to be the configuration whose components coincides with that of $\e$ outside $x,y\in X$,
and $(\eta^{x,y}_x,\eta^{x,y}_y)\coloneqq (\eta_y,\eta_x)$.
For any local function $f\in C_\loc(S^X)$ and $x,y\in X$, we let $\nabla_{\!x,y}f$
be the local function given by
\[
	\nabla_{\!x,y} f(\e)\coloneqq f(\e^{x,y})-f(\e).
\]
Note that by definition, $\nabla_{\!x,y}f=0$ if $x=y$.

\begin{lemma}\label{lem: next}
	There exists a constant $\wt{C}_{\phi,\nu} >0$ depending only on $((S,\phi), \nu)$ such that for any finite locale $(\La,E_\La)$, $x,y \in \La$ and $f\in C(S^\La)$, we have
	\begin{equation}
		\dabs{\nabla_{\!x,y}f}^2_\mu\leq \wt{C}_{\phi,\nu}\dabs{\nabla_{x\arr y} f}^2_\mu.
	\end{equation}
\end{lemma}

\begin{proof}
	Let $x,y\in\La$.
	By the tower property of the conditional expectation, we have
	\begin{align*}
		\dabs{\nabla_{\!x,y}f}^2_\mu=E_\mu[(\nabla_{x,y}f)^2]&=E_\mu[E_\mu[(\nabla_{x,y}f)^2|\cF_{\La\setminus\{x,y\}}]],\\
		\dabs{\nabla_{\!x\arr y}f}^2_\mu=E_\mu[(\nabla_{x\arr y}f)^2]&=E_\mu[E_\mu[(\nabla_{x\arr y}f)^2|\cF_{\La\setminus\{x,y\}}]],
	\end{align*}
	hence it is sufficient to prove our result for the case when the locale is $(\{x,y\},E_{\{x,y\}})$
	with  $E_{\{x,y\}}=\{ (x,y), (y,x) \}\subset \{x,y\}\times\{x,y\}$. 
	 Consider two quadratic forms on $C(S^{\{x,y\}})$ given by
	\begin{align*}
		\cD_{x,y}(f) \coloneqq & E_\mu[(\nabla_{x,y}f)^2],  &\cD_{x \arr y} (f)\coloneqq & E_\mu[(\nabla_{x\arr y}f)^2].
	\end{align*}
	We first prove that $\cD_{x \arr y} (f)=0$ implies $\cD_{x, y} (f) =0$. 
	If this holds true, then our assertion follows since we may take 
	\[
	\wt{C}_{\phi,\nu}:= \sup_{f \in C(S^{\La}), \cD_{x \arr y} (f) \neq 0} \frac{\cD_{x, y} (f)}{\cD_{x \arr y} (f)}
	\]	
	which is finite since $C(S^{\La})$ is a finite dimensional space. Suppose $f\in C(S^{\{x,y\}})$ satisfied $\cD_{x \arr y} (f)=0$, namely that $E_\mu[(\nabla_{x \arr y} f)^2]=0$,
	which implies that $\nabla_{x \arr y}f=0$, hence $f(\eta^{x \arr y})=f(\eta)$ for any $\eta\in S^{\{x,y\}}$.  
	By Lemma \ref{lem: MPL0} for the case $r$ is trivial (see also Remark
	\ref{rem: TB} (b)),
	we have $\dabs{\nabla_{y \arr x} f}^2_\mu\leq C_{\phi, \nu}\dabs{\nabla_{x \arr y} f}^2_\mu=0$,
	which shows that $\nabla_{y \arr x}f=0$.
	Any $\eta=(\eta_x,\eta_y)\in S\times S$ has the same conserved quantities as $\eta^{x,y}=(\eta_y,\eta_x)$.
	Since $(S,\phi)$ is irreducibly quantified,
	there exists a path $\vec\gamma$ in $S^{\{x,y\}}=S\times S$ from $\eta$ to $\eta^{x,y}$.
	This implies that $\eta^{x,y}$ is obtained from $\eta$ by executing $e=(x,y)$ and $\bar e=(y,x)$
	a finite number of times, which shows that $f(\eta^{x,y})=f(\eta)$,
	hence $\nabla_{x,y} f=0$.  This proves that $\cD_{x,y}(f)=0$.
\end{proof}

We may now prove Lemma \ref{lem: MPL}.

\begin{proof}[Proof of Lemma \ref{lem: MPL}]
	For any $x,y\in \La$, since $\La$ is connected, there exists a path $\vec p=(e_1,\ldots,e_N)$ 
	from $x$ to $y$ such that $N=\len(\vec p)\leq\diam{\La}$.
	For any $\e\in S^\La$, let $\e^0\coloneqq\e$ and $\e^i\coloneqq(\e^{i-1})^{o(e_i),t(e_i)}$ for any $i=1,\ldots, N-1$.
	Furthermore, let $\e^N\coloneqq(\e^{N-1})^{e_N}$, and $\e^{i}=(\e^{i-1})^{o(e_{2N-i}),t(e_{2N-i})}$
	for $i=N+1,\ldots,2N-1$. 
	Note that the configurations $\e^i$ for $i\neq N$ is simply given by exchanging certain
	components of $\e^{i-1}$ where as $\e^N$ is given by as the transition of $\e^{N-1}$ induced by $e_N$.
	By construction, we have $\e^{x\arr y}=\e^{2N-1}$.
	This shows that we have
	\[
		\nabla_{\!x\arr y} f=\Bigl(\sum_{i=1}^{N-1}\nabla_{\!o(e_i),t(e_i)}f\Bigr)+\nabla_{\!e_N}f+
		\Bigl(\sum_{i=N+1}^{2N-1}\nabla_{\!o(e_i),t(e_i)}f\Bigr).
	\]
	By Schwartz's inequality, we have
	\begin{align*}
		\dabs{\nabla_{\!x\arr y}f}^2_\mu&\leq 3(N-1)\Bigl(\sum_{i=1}^{N-1}\dabs{\nabla_{\!o(e_i),t(e_i)}f}_\mu^2\Bigr)+
		3\dabs{\nabla_{\!e_N}f}_\mu^2+
		3(N-1)\Bigl(\sum_{i=N+1}^{2N-1}\dabs{\nabla_{\!o(e_i),t(e_i)}f}^2_\mu\Bigr)\\
		&\leq 3N \max\{1, \wt{C}_{\phi,\nu}\} \sum_{i=1}^{2N-1}\dabs{\nabla_{\!e_i}f}_\mu^2
		\leq 6N^2 \max\{1, \wt{C}_{\phi,\nu}\} \sup_{e\in E_\La} \dabs{\nabe f}^2_\mu\\
		&=6 N^2 \max\{1, \wt{C}_{\phi,\nu}\}\dabs{\partial_\La f}^2_\sp,
	\end{align*}
	where the second inequality follows by Lemma \ref{lem: next}.
	If we let $C_{M\!P}=6  \max\{1, \wt{C}_{\phi,\nu}\} >0$, then this gives the desired constant.
\end{proof}

%
%
\subsection{Boundary Estimates}
%
%

In this subsection, we prove Proposition \ref{prop: BE},
which gives the bound of differentials of local functions with respect
to edges which intersect the boundary of the locality.
We let $(X,E)$ be a locale. 

For any $\La\subset X$, we define the boundary $\partial\La$ of $\La$ by 
\[
	\partial\La\coloneqq\{e\in E\mid o(e)\in\La, t(e)\not\in\La\}.
\]

\begin{proposition}[Boundary Estimate]\label{prop: BE}
	There exists a constant $C_{BE}>0$ depending only on $(S,\phi)$ and the probability measure
	$\nu$ on $S$ satisfying the following property.
	Consider any $\La,\Sigma\in\sI$ such that $\La\subset\Sigma$.
	For any $h\in C(S^\Sigma)$ and $e\in\partial\La\cap E_\Sigma$, we have
	\[
		\dabs{\nabe(\pi^{\La}h)}^2_\mu\leq\frac{C_{BE}}{|\Sigma\setminus\La|}
		\Bigl(\dabs{h}^2_\mu+\sum_{y\in\Sigma\setminus\La}\dabs{\nabla_{o(e)\arr y}h}^2_\mu\Bigr).
	\]
	Similarly,
	\[
		\dabs{\nabla_{\bar e}(\pi^{\La}h)}^2_\mu\leq\frac{C_{BE}}{|\Sigma\setminus\La|}
		\Bigl(\dabs{h}^2_\mu+\sum_{x\in\Sigma\setminus\La}\dabs{\nabla_{x\arr t(\bar{e})}h}^2_\mu\Bigr).
	\]
\end{proposition}

We will give the proof of Proposition \ref{prop: BE} at the end of this subsection.
We let $\La''\coloneqq\Sigma\setminus\La$.
For any $\beta\in S$, we let $\Fr_\beta\colon S^{\La''}\rightarrow\bbN$
be the map
\[
	\Fr_\beta\coloneqq\sum_{y\in\La''} 1_{\{\eta_y=\beta\}},
\]
where $1_{\{\eta_y=\beta\}}$ for $y\in\La''$ is the indicator function of $\{ \e\in S^{\La''} \mid
\eta_y=\beta \} \subset S^{\La''} $.
By construction, $\Fr_\beta$ calculates the number of the occurrence of $\beta$ in the
components of $\e$.
Since $E_\mu[1_{\{\eta_y=\beta\}}]=\nu(\beta)$ for any $y\in\La''$, 
we have $E_\mu[\Fr_\beta]=\abs{\La''}\nu(\beta)$.
If we let
\begin{equation}\label{eq:F}
	\Fr^*_\beta\coloneqq\nu(\beta)^{-1}\abs{\La''}^{-1}\Fr_\beta,
\end{equation}
then we have $E_\mu[\Fr^*_\beta]=1$.  Moreover, since $\Fr^*_\beta$ is independent of $\cF_\La$,
we have
\[
	\pi^\La\Fr^*_\beta=E[\Fr^*_\beta | \cF_\La ]=E_\mu[\Fr^*_\beta]=1.
\]
By a direct computation, we have 
\begin{equation}\label{eq: calc}
		\pi^\La(\Fr^*_\beta-\pi^\La\Fr^*_\beta)^2 =\frac{1-\nu(\beta)}{\nu(\beta)\abs{\La''}}.
	\end{equation}

The following bound will be used in the proof of Proposition \ref{prop: BE}.

\begin{lemma}\label{lem: CE}
	For any bounded random variables $g,h$ 
	and $\sigma$-algebra $\cG$, we have
	\[
		(E[g|\cG]E[h|\cG]-E[gh|\cG])^2\leq
		E[(g-E[g|\cG])^2|\cG]	E[(h-E[h|\cG])^2|\cG].
	\]
\end{lemma}

\begin{proof}
	By the Schwartz's inequality
	 for conditional expectations, we have
	\begin{equation}\label{eq: CS}
		E[XY|\cG]^2\leq E[X^2|\cG]E[Y^2|\cG]
	\end{equation}
	for any bounded random variables $X$ and $Y$.
	If we let $X=g-E[g|\cG]$ and $Y=h-E[h|\cG]$, then 
	\begin{align*}
		E[XY|\cG]&=E[(g-E[g|\cG])(h-E[h|\cG])|\cG]\\
		&= E[(gh-E[g|\cG] h-g E[h|\cG]+E[g|\cG]E[h|\cG])|\cG]\\
		&=E[gh|\cG]-E[g|\cG]E[h|\cG],
	\end{align*}
	where we have used the fact that $E[g|\cG]$ and $E[h|\cG]$ are $\cG$-measurable.
	Our assertion now follows from \eqref{eq: CS}.
\end{proof}
	
\begin{lemma}\label{lem: prep}
	Let $\beta\in S$ and let $\Fr^*_\beta$ be the function given in \eqref{eq:F}.
	For any function $h\in C(S^\Sigma)$,
	we have
	\[
		\bigl((\pi^\La\Fr^*_{\beta})(\pi^\La h)(\e)-\pi^\La(\Fr^*_{\beta}\, h)(\e)\bigr)^2
		\leq\frac{1-\nu(\beta)}{\nu(\beta)\abs{\La''}}\pi^\La( h- \pi^\La h)^2(\e).
	\]
\end{lemma}

\begin{proof}
	By Lemma \ref{lem: CE}, we have
	\[
		\bigl((\pi^\La\Fr^*_{\beta})(\pi^\La h)(\e)-\pi^\La(\Fr^*_{\beta}\, h)(\e)\bigr)^2
		\leq\pi^\La( \Fr^*_{\beta} - \pi^\La \Fr^*_{\beta})^2\pi^\La( h- \pi^\La h)^2(\e).
	\]
	Combining this inequality with 
	\eqref{eq: calc},
	we see that 	
	\[
		\bigl((\pi^\La\Fr^*_{\beta})(\pi^\La h)(\e)-\pi^\La(\Fr^*_{\beta}\, h)(\e)\bigr)^2
		\leq\frac{1-\nu(\beta)}{\nu(\beta)\abs{\La''}}\pi^\La( h- \pi^\La h)^2(\e)
	\]
	as desired.
\end{proof}

\begin{lemma}\label{lem: diff}
	Let $h\in C(S^\Sigma)$.
	Then for any $e\in\partial\La\cap E_{\Sigma}$ and $\e\in S^\Sigma$ such that 
	$(\eta_{o(e)},\eta_{t(e)})=(\alpha,\beta)\in S\times S$ and $(\alpha',\beta'):=\phi(\alpha,\beta) \neq (\alpha,\beta)$, we have
	\[
		\pi^{\La}(\Fr^*_{\beta'}\, h)(\e^{e})
		=\frac{1}{\nu(\beta)\abs{\La''}}\sum_{y\in\La''}
		\pi^{\La}(1_{\{\eta_y=\beta\}}h_{o(e)\arr y})(\e).
	\]
\end{lemma}

\begin{proof}
		By definition of $\Fr^*_\beta$ and the conditional expectation $\pi^\La$, we have
	\begin{align*}
		\pi^{\La}(\Fr^*_{\beta'}\,h)(\e^{e})
		&=\frac{1}{\nu(\beta')\abs{\La''}}\sum_{y\in\La''}\pi^{\La}(1_{\{\eta_y=\beta'\}}h)(\e^{e})\\
		&=\frac{1}{\nu(\beta')\abs{\La''}}\sum_{y\in\La''}
		\frac{1}{\mu(\e^e|_{\La})}\sum_{\substack{\e'\in S^\Sigma\\\e'|_{\La}=\e^e|_\La,\,\e'_y=\beta'}}
		h(\e')\mu(\e').
	\end{align*}
	In the above sum, the condition $\e'|_{\La}=\e^e|_\La$ and $\eta_y=\beta'$
	ensures that $\e'=(\eta'_x)$ satisfies $\eta'_x=\eta_x$ for $x\in\La\setminus \{o(e)\}$, $\eta'_{o(e)}=\alpha'$,
	and $\eta'_{y}=\beta'$.
	Since we have assumed that $(\alpha',\beta')\neq(\alpha,\beta)$, 
	we have $(\alpha,\beta)=\ol\phi(\alpha',\beta')$ by the definition of the interaction.
	Hence, for each $y \in \La''$ and $\eta' \in S^{\Sigma}$ satisfying $\e'|_{\La}=\e^e|_\La,\,\e'_y=\beta'$, if we let $\tilde\e=(\tilde \eta_x) \in S^{\Sigma}$ be the configuration such that $\tilde \eta_x\coloneqq \eta'_x$ for $x\neq o(e),y$,
	and $(\tilde \eta_{o(e)}, \tilde \eta_y)\coloneqq(\alpha,\beta)$, then $\tilde\e$
	is the unique element in $S^\Sigma$ satisfying $\tilde\e|_\La=\e|_\La$, $\tilde \eta_y=\beta$, and
	$\e'=\tilde\e^{o(e)\arr y}$.
	This gives	
	\begin{align*}
		\pi^{\La}(\Fr^*_{\beta'}\,h)(\e^{e})
		&=\frac{1}{\nu(\beta')}\sum_{y\in\La''}
		\frac{1}{\mu(\e^e|_{\La})}\sum_{\substack{\tilde\e\in S^\Sigma\\\tilde\e|_{\La}=\e|_\La,\,\tilde\eta_y=\beta}}
		h(\tilde\e^{o(e)\arr y})\mu(\tilde\e^{o(e)\arr y})\\
		&=\frac{1}{\nu(\beta)}\sum_{y\in\La''}
		\frac{1}{\mu(\e|_{\La})}\sum_{\substack{\tilde\e\in S^\Sigma\\\tilde\e|_{\La}=\e|_\La,\,\tilde\eta_y=\beta}}
		h_{o(e)\arr y}(\tilde\e)\mu(\tilde\e)\\
		&=\frac{1}{\nu(\beta)}\sum_{y\in\La''}
		\pi^\La(1_{\{\eta_y=\beta\}} h_{o(e)\arr y})(\e)
	\end{align*}
	as desired. For the second equality, we used the relations $\mu(\e|_{\La})\nu(\alpha')=\mu(\e^e|_{\La})\nu(\alpha)$
	and $\mu(\tilde\e^{o(e)\arr y})\nu(\alpha)\nu(\beta)=\mu(\tilde\e)\nu(\alpha')\nu(\beta')$,
	which follows from the fact that $\mu$ is the product measure $\mu=\nu^{\otimes\Sigma}$.
\end{proof}

We are now ready to prove the boundary estimate.

\begin{proof}[Proof of Proposition \ref{prop: BE}] 
	Let $f=\pi^\La h$.
	By definition, we have
	\begin{equation}\label{eq:24}
		\dabs{\nabe f}_\mu^2
		=\sum_{\e\in S^{\La\cup\{t(e)\}}}(f(\e^e)-f(\e))^2\mu(\e)
		=\sum_{(\alpha,\beta)\in S\times S}
		\sum_{\substack{\e\in S^{\La\cup\{t(e)\}}\\ (\eta_{o(e)},\eta_{t(e)})=(\alpha,\beta)}}(f(\e^e)-f(\e))^2\mu(\e).
	\end{equation}
	Note that for any $\e\in S^{\La\cup\{t(e)\}}$ such that $(\eta_{o(e)},\eta_{t(e)})=(\alpha,\beta)$, 
	noting that $(\pi^\La\Fr^*_{\beta'})=(\pi^\La\Fr^*_{\beta})=1$,
	we may decompose the value $(f(\e^e)-f(\e))$ as
	\begin{align*}
		f(\e^e)-f(\e)=
		(\pi^\La\Fr^*_{\beta'})f (\e^e)-
		(\pi^\La\Fr^*_\beta)f(\e)&=
		(\pi^\La\Fr^*_{\beta'})f(\e^{e})-
		\pi^\La(\Fr^*_{\beta'}\, h)(\e^{e})\\
		&\qquad+\pi^\La(\Fr^*_{\beta'}\, h)(\e^{e})
		-\pi^\La(\Fr^*_\beta\,h)(\e)\\
		&\qquad+\pi^\La(\Fr^*_\beta\,h)(\e)
		-(\pi^\La\Fr^*_\beta) f(\e)
	\end{align*}
	with $(\alpha',\beta')=\phi(\alpha,\beta)$.
	By Schwartz's inequality, we have
		\begin{align}
		(f(\e^e)-f(\e))^2&\leq
		3\bigl((\pi^\La\Fr^*_{\beta'})f(\e^{e})-\pi^\La(\Fr^*_{\beta'}\, h)(\e^{e})\bigr)^2\label{eq:10}\\
		&\qquad+ 3\bigl(\pi^\La(\Fr^*_{\beta'}\, h)(\e^e)-\pi^\La(\Fr^*_\beta\,h)(\e)\bigr)^2\nonumber\\
		&\qquad+3\bigl(\pi^\La(\Fr^*_\beta\,h)(\e)-(\pi^\La\Fr^*_\beta) f(\e)\bigr)^2.\nonumber
	\end{align}
	 We bound each of the terms on the right hand side of \eqref{eq:10} with respect to
	 the sum over $(\alpha,\beta)\in S\times S$ and $\e\in S^{\La\cup\{t(e)\}}$ such that
	 $(\eta_{o(e)},\eta_{t(e)})=(\alpha,\beta)$.
	Concerning the first and the third terms of the right hand side, by Lemma \ref{lem: prep},
	we see that
	\begin{align*}
		\bigl((\pi^\La\Fr^*_{\beta'})(\pi^\La h)(\e^e)-\pi^\La(\Fr^*_{\beta'}\, h)(\e^e)\bigr)^2
		&\leq\frac{1-\nu(\beta')}{\nu(\beta')\abs{\La''}}\pi^\La( h- \pi^\La h)^2(\e^e)\\
		\bigl((\pi^\La\Fr^*_{\beta})(\pi^\La h)(\e)-\pi^\La(\Fr^*_{\beta}\, h)(\e)\bigr)^2
		&\leq\frac{1-\nu(\beta)}{\nu(\beta)\abs{\La''}}\pi^\La( h- \pi^\La h)^2(\e).
	\end{align*}
	For the first term, note that by the tower property,
	$\pi^\La( h- \pi^\La h)^2 = \pi^\La(\pi^{\La\cup\{t(e)\}}( h- \pi^\La h)^2)$.
	The definition of $\pi^\La$ gives
	\begin{align*}
		\sum_{(\alpha,\beta)\in S\times S}&
		\sum_{\substack{\e\in S^{\La\cup\{t(e)\}}\\ (\eta_{o(e)},\eta_{t(e)})=(\alpha,\beta)}}
		\frac{1-\nu(\beta')}{\nu(\beta')}\pi^\La( h- \pi^\La h)^2(\e^e)\mu(\e)\\
		&\leq  	\sum_{(\alpha,\beta)\in S\times S}
		\sum_{\substack{\e\in S^{\La\cup\{t(e)\}}\\ (\eta_{o(e)},\eta_{t(e)})=(\alpha,\beta)}}
		\frac{1}{\nu(\beta')}
		\frac{\mu(\e)}{\mu(\e^e|_{\La})}\sum_{\substack{\e'\in S^{\La\cup\{t(e)\}}\\\e'|_{\La}=\e^e|_{\La}}}
		\pi^{\La\cup\{t(e)\}}( h- \pi^\La h)^2(\e')\mu(\e')\\
		&= 	\sum_{(\alpha,\beta)\in S\times S}\sum_{\substack{\e\in S^{\La\cup\{t(e)\}}\\ (\eta_{o(e)},\eta_{t(e)})=(\alpha,\beta)}}
		\frac{\nu(\alpha)\nu(\beta)}{\nu(\alpha')\nu(\beta')}\sum_{\substack{\e'\in S^{\La\cup\{t(e)\}}\\\e'|_{\La}=\e^e|_{\La}}}
		\pi^{\La\cup\{t(e)\}}( h- \pi^\La h)^2(\e')\mu(\e')
	\end{align*}
	\begin{align*}
	&\leq  C_{\phi,\nu} \sum_{(\alpha,\beta)\in S\times S}\sum_{\substack{\e\in S^{\La\cup\{t(e)\}}\\ (\eta_{o(e)},\eta_{t(e)})=(\alpha,\beta)}}
		\sum_{\substack{\e'\in S^{\La\cup\{t(e)\}}\\\e'|_{\La}=\e^e|_{\La}}}
		\pi^{\La\cup\{t(e)\}}( h- \pi^\La h)^2(\e')\mu(\e')\\
		&= C_{\phi,\nu} \sum_{\e\in S^{\La\cup\{t(e)\}}}
		\sum_{\substack{\e'\in S^{\La\cup\{t(e)\}}\\\e'|_{\La}=\e^e|_{\La}}}
		\pi^{\La\cup\{t(e)\}}( h- \pi^\La h)^2(\e')\mu(\e')\\
		&\leq C_{\phi,\nu} \abs{S}^2 E_\mu[\pi^{\La\cup\{t(e)\}}( h- \pi^\La h)^2]
		=C_{\phi,\nu} \abs{S}^2 E_\mu[( h- \pi^\La h)^2]\\
		&\leq C_{\phi,\nu} \abs{S}^2 E_\mu[h^2]
		=C_{\phi,\nu} \abs{S}^2 \dabs{h}_\mu^2.
	\end{align*}
	For the third term, we have
	\begin{align*}
		\sum_{(\alpha,\beta)\in S\times S}&
		\sum_{\substack{\e\in S^{\La\cup\{t(e)\}}\\ (\eta_{o(e)},\eta_{t(e)})=(\alpha,\beta)}}
		\frac{1-\nu(\beta)}{\nu(\beta)}\pi^\La( h- \pi^\La h)^2(\e)\mu(\e)\\
		&=
		\sum_{(\alpha,\beta)\in S\times S}
		\sum_{\substack{\e\in S^{\La\cup\{t(e)\}}\\ (\eta_{o(e)},\eta_{t(e)})=(\alpha,\beta)}}
		(1-\nu(\beta))\pi^\La( h- \pi^\La h)^2(\e|_{\La})\mu(\e|_{\La})\\
		&\overset{(*)}{=}
		\sum_{\alpha\in S}
		\sum_{\substack{\e \in S^{\La}\\ \eta_{o(e)}=\alpha}}
		(\abs{S}-1)\pi^\La( h- \pi^\La h)^2(\e)\mu(\e)\\
		&=(\abs{S}-1)E_\mu[\pi^\La( h- \pi^\La h)^2]\leq (\abs{S}-1)E_\mu[h^2]=(\abs{S}-1)\dabs{h}^2_\mu,
	\end{align*}
	where $(*)$ follows from the fact that $\sum_{\beta\in S}(1-\nu(\beta))=\abs{S}-1$.
	Finally, we consider the second term of \eqref{eq:10}.
	By Lemma \ref{lem: diff}, 
	for any $\e\in S^{\La\cup\{t(e)\}}$ such that $(\eta_{o(e)},\eta_{t(e)})=(\alpha,\beta)$, we have
	\begin{align*}
		\bigl(\pi^\La(\Fr^*_{\beta'}\, h)(\e^{e})-\pi^\La(\Fr^*_\beta\,h)(\e)\bigr)^2
		& \le \frac{1}{\nu(\beta)^2\abs{\La''}^2}\Bigl(\pi^\La\Bigl(\sum_{y\in\La''}
		1_{\{\eta_y=\beta\}}(h_{o(e)\arr y}-h)\Bigr)(\e)\Bigr)^2
	\end{align*}
	since if $(\alpha',\beta')=(\alpha,\beta)$, then the left-hand side is $0$. 
	By Schwartz's inequality, 
	\[
	\Bigl| \sum_{y\in\La''}1_{\{\eta_y=\beta\}}(h_{o(e)\arr y}-h)\Bigr| \le \sqrt{ \sum_{y\in\La''}1_{\{\eta_y=\beta\}}} \sqrt{ \sum_{y\in\La''}(h_{o(e)\arr y}-h)^2},
	\]
	so by applying Schwartz's inequality for conditional expectation, we have
	\begin{align*}
	\Bigl(\pi^\La\Bigl(\sum_{y\in\La''}1_{\{\eta_y=\beta\}}(h_{o(e)\arr y}-h)\Bigr)(\e)\Bigr)^2 
	& \leq \pi^\La\Bigl(\sum_{y\in\La''}1_{\{\eta_y=\beta\}}\Bigr)(\e) 
	\pi^\La\Bigl(\sum_{y\in\La''} (h_{o(e)\arr y}-h)^2\Bigr)(\e) \\
	& = |\La''| \nu(\beta) \pi^\La\Bigl( \sum_{y\in\La''}(h_{o(e)\arr y}-h)^2\Bigr) (\e).
	\end{align*}
	Hence
	\begin{align*}
		\sum_{(\alpha,\beta)\in S\times S}&
		\sum_{\substack{\e\in S^{\La\cup\{t(e)\}}\\ (\eta_{o(e)},\eta_{t(e)})=(\alpha,\beta)}}
		\frac{1}{\nu(\beta)^2\abs{\La''}^2}
		\Bigl(\pi^\La\Bigl(\sum_{y\in\La''} 1_{\{\eta_y=\beta\}}(h_{o(e)\arr y}-h)\Bigr)(\e)\Bigr)^2\mu(\e)\\
	       &\leq 
		\frac{1}{\abs{\La''}}
		\sum_{(\alpha,\beta)\in S\times S}\sum_{\substack{\e\in S^{\La\cup\{t(e)\}}\\ (\eta_{o(e)},\eta_{t(e)})=(\alpha,\beta)}}\frac{1}{\nu(\beta)}
		\pi^\La\Bigl(\sum_{y\in\La''}(h_{o(e)\arr y}-h)^2\Bigr)(\e) \mu(\eta) \\
		&=
		\frac{1}{\abs{\La''}}
		\sum_{(\alpha,\beta)\in S\times S}
		\sum_{\substack{\e\in S^{\La}\\ \e_{o(e)}=\alpha}}\sum_{y\in\La''}
		\pi^\La\bigl((h_{o(e)\arr y}-h)^2\bigr)(\e)\mu (\eta) \\
		&= \frac{|S|}{\abs{\La''}}
		\sum_{y\in\La''}\sum_{\alpha\in S}
		\sum_{\substack{\e\in S^{\La}\\ \e_{o(e)}=\alpha}}
		\pi^\La\bigl((h_{o(e)\arr y}-h)^2\bigr)(\e) \mu(\e) \\
		& =  \frac{|S|}{\abs{\La''}}
		\sum_{y\in\La''}
		\dabs{\nabla_{o(e)\arr y}h}^2_\mu.
	\end{align*}
	Applying the above inequalities to \eqref{eq:24}, we see that
	\begin{align*}
		\dabs{\nabe f}_\mu^2&\leq \frac{3}{\abs{\La''}}\Bigl(
		(C_{\phi,\nu} \abs{S}^2+|S|-1) \dabs{h}_\mu^2+ |S|
		\sum_{y\in\La''}
		\dabs{\nabla_{o(e)\arr y}h}^2_\mu \Bigr).
	\end{align*}
	Since $C_{\phi,\nu}$ and $|S|$ depends only on $((S,\phi), \nu)$, 
	we see that if we let $C_{BE}\coloneqq 3(C_{\phi,\nu} \abs{S}^2+|S|)>0$, then $C_{BE}$ satisfies 
	the required property.
	
	The second inequality in the statement is shown in the same way.
\end{proof}

As a last result of this subsection, by assuming that $((S,\phi),\nu)$
has a uniformly bounded spectral gap, we prove the following simple estimate. 

\begin{lemma}\label{lem: sgmp}
If the data $((S,\phi),\nu)$ has a \emph{uniformly bounded spectral gap},
	then for any complete locale $(\La,E_\La)$ and 
	function $h\in C(S^\Lambda)$, we have
	\begin{equation}\label{eq: sg}
		\dabs{h}^2_\La\leq C_{SG}^{-1} \abs{\La}\dabs{\partial_\La h}^2_\sp.
	\end{equation}
\end{lemma}

\begin{proof}
By definition, 
\[
\dabs{h}^2_\La\leq C_{SG, (\La,E_{\La})}^{-1} \sum_{e \in E_{\La} } \dabs{\nabla_e h}^2_\mu  \le C_{SG}^{-1}|\La|^{-1} \sum_{e \in E_{\La} } \dabs{\nabla_e h}^2_\mu.
\]
Then, $\dabs{\nabla_e h}^2_\mu \leq \dabs{\partial_\La h}^2_\sp$ for any $e \in E_{\La}$ and $|E_{\La}| \le |\La|^2$ implies the assertion.
\end{proof}

Then, the Moving Particle Lemma gives the following.

\begin{proposition}\label{prop: SG}
	If the data $((S,\phi),\nu)$ has a \emph{uniformly bounded spectral gap},
	then 
	for any $\Sigma\in \sI$ and $h\in C(S^\Sigma)$, we have
	\[
		\dabs{h}^2_\Sigma\leq C_{SG}^{-1} C_{M\!P}\abs{\Sigma}\diam{\Sigma}^2\dabs{\partial_\Sigma h}^2_\sp.
	\]
\end{proposition}
\begin{proof}
	For the locale $(\Sigma,E_\Sigma)$, let 
	$\wh\Sigma\coloneqq\Sigma$ and $E_{\wh\Sigma}\coloneqq(\Sigma\times\Sigma)\setminus\Delta_\Sigma$.
	Then $(\wh\Sigma, E_{\wh\Sigma})$ by construction is a complete graph.
	Note that $\Ker\partial_{\wh\Sigma}\subset\Ker\partial_\Sigma$, hence we have $\dabs{h}_{\Sigma}\leq
	\dabs{h}_{\wh\Sigma}$ for any $h\in C(S^\Sigma)$.
	By Lemma \ref{lem: sgmp} and the definition of the 
	norm $\dabs{\cdot}_\sp$, we have
	\begin{align*}
		\dabs{h}_{\Sigma}\leq\dabs{h}_{\wh\Sigma}&\leq
		C_{SG}^{-1}\abs{\wh\Sigma}\biggl(\sup_{e\in\wh\Sigma}\,\Dabs{(\partial_{\wh\Sigma}h)_e}^2_\mu\biggr)
		=C_{SG}^{-1}\abs{\Sigma}\biggl(\sup_{x,y\in\Sigma}\dabs{\nabla_{\!x\arr y}h}^2_\mu\biggr).
	\end{align*}
	Here, we have used the fact that $\nabla_{\!x\arr y}h=0$ if $x=y$.
	Moreover, by the Moving Particle Lemma (Lemma \ref{lem: MPL}), we have
	\[
		\dabs{h}_{\Sigma}\leq C_{SG}^{-1}\abs{\Sigma}\Bigl(C_{M\!P}\diam{\Sigma}^2\dabs{\partial_\Sigma h}^2_\sp\Bigr)
		=C_{SG}^{-1} C_{M\!P}\abs{\Sigma}\diam{\Sigma}^2\dabs{\partial_\Sigma h}^2_\sp.
	\]
\end{proof}

\begin{remark}
The assumption that $((S,\phi),\nu)$ has a uniformly bounded spectral gap 
is essentially used only Proposition \ref{prop: SG} (and Lemma \ref{lem: sgmp} for this proposition). 
Namely, if the inequality asserted by Proposition \ref{prop: SG} holds,
then the results in the rest of the article proved assuming the existence of a
uniformly bounded spectral gap are also valid, including our main theorem. 
In the literature of the nongradient method, the condition on the spectral gap is typically given as
an assertion of the form of Proposition \ref{prop: SG} 
for the sequence $\{\La_n\}$ of Definition \ref{def: Euclid}. 
\end{remark}

As a conclusion of this subsection, we introduce the boundary differential and give a bound of it, which is crucial in the next section. 

For any $f\in C^0_\col(S^X_\mu)$ and $\La\in\sI$, we define the \emph{boundary differential} by
\begin{equation}\label{eq: boundary}
	\partial^\dagger_\La f\coloneqq \partial(\pi^\La f) - \partial_{\La}(\pi^\La f).
\end{equation}
Since $(\partial(\pi^\La f))_e=(\partial_{\La}(\pi^\La f))_e$ for any $e\in E_\La$
and $(\partial(\pi^\La f))_e=(\partial_{\La}(\pi^\La f))_e=0$ for $e\in E_{X \setminus\La}$,
we have $(\partial^\dagger_\La f)_e\neq 0$ only if $e\in\partial\La\cup\ol{\partial\La}$.
Note that for $f\in \Ker\partial$, we have $\pi^\La f\in\Ker\partial_\La$.
However, in general, $\partial^\dagger_\La f= \partial(\pi^\La f)\neq0$,
since in general, we have $\Ker\partial_\La\not\subset\Ker\partial$.


\begin{proposition}\label{prop:b2}
	Suppose the interaction $(S,\phi)$ is irreducibly quantified and  has a \emph{uniformly bounded spectral gap}.
	Then there exists a constant $C>0$ depending only on $((S,\phi),\nu)$ such that
	for any $\La,\Sigma\in\sI$ satisfying $\La\subset\Sigma$, $\partial \La \subset E_{\Sigma}$ and 
	$h\in(\Ker\partial_\Sigma)^\perp$, we have
	\[
		\dabs{\partial^\dagger_{\La}h}^2_\sp\leq C\frac{\abs{\Sigma}}{\abs{\Sigma\setminus\La}}\diam{\Sigma}^2
		\dabs{\partial_{\Sigma}h}^2_\sp.
	\]
\end{proposition}

\begin{proof}
	Suppose $h\in(\Ker\partial_\Sigma)^\perp$.  Then $\dabs{h}_\Sigma=\dabs{h}_\mu$.
	By the boundary estimate Proposition \ref{prop: BE} and 
	the spectral gap estimate Proposition \ref{prop: SG}, for any $e\in\partial\La$, we have
	\begin{align*}
		\dabs{\nabe(\pi^{\La}h)}^2_\mu&\leq\frac{C_{BE}}{|\Sigma\setminus\La|}
		\Bigl(\dabs{h}_\Sigma^2+\sum_{y\in\Sigma\setminus\La}\dabs{\nabla_{o(e)\arr y}h}^2_\mu\Bigr)\\
		&\leq\frac{C_{BE}}{|\Sigma\setminus\La|}
		\Bigl(C_{SG}^{-1} C_{M\!P}\abs{\Sigma}\diam{\Sigma}^2\dabs{\partial_\Sigma h}^2_\sp
		+\sum_{y\in\Sigma\setminus\La}\dabs{\nabla_{o(e)\arr y}h}^2_\mu\Bigr).
	\end{align*}
	Combining this with the Moving Particle Lemma (Lemma \ref{lem: MPL}), we see that
	\begin{align*}
		\dabs{\nabe(\pi^{\La}h)}^2_\mu&\leq
		\frac{C_{BE}}{|\Sigma\setminus\La|}
		\Bigl(C_{SG}^{-1} C_{M\!P}\abs{\Sigma}\diam{\Sigma}^2\dabs{\partial_\Sigma h}^2_\sp
		+C_{M\!P}\abs{\Sigma\setminus\La} \diam{\Sigma}^2\dabs{\partial_\Sigma h}^2_\sp\Bigr)\\
		&\leq C\frac{\abs{\Sigma}}{\abs{\Sigma\setminus\La}}\diam{\Sigma}^2
		\dabs{\partial_{\Sigma}h}^2_\sp
	\end{align*}
	for some constant $C>0$.  More precisely, noting that $\abs{\Sigma\setminus\La}\leq\abs{\Sigma}$,
	we may take $C\coloneqq C_{BE}C_{M\!P}(C_{SG}^{-1}+1)>0$.  
	The same equality also holds for $e\in\ol{\partial\La}$.
	Our assertion follows from this inequality and the definition of $\partial^\dagger_\La$.
\end{proof}

%
%
\subsection{Criterion for Locality}\label{subsec: local}
%
%

In this subsection, we give a criterion for a $L^2$-function to be a local function.
This result will be used in the proof of Proposition \ref{prop: local}.
We still work over a general locale $(X,E)$.
For any $Y\subset X$, we define an \emph{approximation} $\La_n\uparrow Y$ to be a
family $\{\La_n\}_{n\in\N}$ of finite subsets of $Y$ satisfying $\La_n\subset\La_{n+1}$
for any $n\in\N$ and $Y=\bigcup_{n\in\N}\La_n$.

\begin{proposition}\label{prop: main}
 	Let $\La\in\sI$, and let $Y\subset X$ be an infinite sub-locale such that $\La\cap Y=\emptyset$.
	For any $f\in L^2(\mu)$, suppose that $f$ is $\cF_{\La\cup Y}$-measurable,
	and we have
	\[
		\nabla_e(\pi^{\La\cup\La'}f) =0
	\]
	for any finite $\La'\subset Y$ and
$e\in E_{\La'}$. Then we have $\pi^\La f=f$.  In particular, $f$ is a local function.
 \end{proposition}

We will give the proof of Proposition \ref{prop: main} at the end of this section.
We first review the $0$-$1$ Law of Hewitt and Savage for exchangeable $\sigma$-algebra. 
For any countable set $I$, we say that $\varrho$ is a \emph{finite permutation of $I$},
if $\varrho\colon I\rightarrow I$ is a bijection such that the set $\{x\in I\mid \varrho(x)\neq x\}$
is a finite set.  We denote by $\frS_I$ the set of finite permutations of $I$.
Let $(\Omega,\cF,P)$ be a probability space.
For a fixed $I$, let $(q_x)_{x\in I}$ be a family of independent and identically 
distributed random variables on $\Omega$ parametrized by $I$ with values in a measure
space $E$.  
We denote by $\sigma((q_x)_{x\in I})$ the sub $\sigma$-algebra of $\cF$
generated by $(q_x)_{x\in I}$.   By definition, for any $A\in \sigma((q_x)_{x\in I})$, there exists
measurable $B\subset E^I$ such that $A=\{ (q_x)\in B\}$.
For any finite permutation $\varrho\in\frS_I$, we let $A^\varrho\coloneqq\{(q_{\varrho(x)})\in B\}$.
We define the \emph{exchangeable $\sigma$-algebra} $\mathfrak{E}$ of $\sigma((q_x)_{x\in I})$
by
\begin{equation}\label{eq: exchangeable}
	\mathfrak{E}\coloneqq\{  A\in \sigma((q_x)_{x\in I})\mid A^\varrho=A\,\forall\varrho\in\frS_I\}.
\end{equation}

We say that a $\sigma$-algebra $\cF'$ is \emph{trivial for a probability measure $P$},
or simply \emph{$P$-trivial}, if  we have $P(A)\in\{0,1\}$ for any $A\in\cF'$.
The following Corollary \ref{cor: HS} 
is a key for our proof of Proposition \ref{prop: main}. 
See for example \cite{Kle14}*{Corollary 12.19} for a proof.

\begin{corollary}[$0$-$1$ Law of Hewitt and Savage]\label{cor: HS}
	Let $I$ be a countable set and let $(q_x)_{x\in I}$ be a family of independent and identically 
	distributed random variables on $(\Omega,\cF,P)$ parametrized by $I$.
	Then the exchangeable $\sigma$-algebra $\mathfrak{E}$ of \eqref{eq: exchangeable} is $P$-trivial. 
\end{corollary}

We now return to the case where our probability space is $(S^X,\cF_X,\mu)$.
Suppose $Y\subset X$ is an infinite locale, and consider an approximation $\La_n\uparrow Y$. Fix a basis $\xi^{(1)}, \ldots, \xi^{(c_{\phi})}$ of $\Consv^{\phi}(S)$.
For any $n\in\N$, we have a measurable map $\bsxi_{\!\La_n}\colon S^X\rightarrow \bbR^{c_{\phi}}$ given by
\[
\eta \mapsto \left(\sum_{x \in \La_n}\xi^{(1)}(\eta_x), \ldots,\sum_{x \in \La_n}\xi^{(c_{\phi})}(\eta_x) \right).
\]
We let
\begin{equation}\label{eq: Gn}
	\cG_n\coloneqq \sigma\bigl(\bsxi_{\!\La_n}, \cF_{Y\setminus\La_n}\bigr)  
	\subset\cF_X
\end{equation}
be the $\sigma$-algebra generated by the measurable functions
$\bsxi_{\!\La_n}$ and $\cF_{Y\setminus\La_n}$.
Then we have a natural inclusion $\cG_n\supset\cG_{n+1}$ 
for any $n\in\N$.

In order to give the proof of Proposition \ref{prop: main},
we first prove the following lemma.

\begin{lemma}\label{lem: converge}
	Let $g\in C(S^{\La\cup\Sigma})$, such that $\La,\Sigma\in\sI$ and $\La\cap\Sigma=\emptyset$.
	We let $Y$ be an infinite locale such that $\La\cap Y=\emptyset$ and $\Sigma\subset Y$.
	Consider an approximation $\La_n\uparrow Y$.  Then we have
	\[
		\lim_{n\rightarrow\infty}E[g|\sigma(\cF_\La,\cG_n)]=E[g|\cF_\La]
	\]
	in $L^2$, where $\cG_n\coloneqq\sigma\bigl(\bsxi_{\!\La_n},\cF_{Y\setminus\La_n}\bigr)$ 
	as in \eqref{eq: Gn}.
\end{lemma}

\begin{proof}
	Since $\La$ and $\Sigma$ are finite, $S^{\La\cup\Sigma}=S^\La\times S^\Sigma$ is also finite.
	Let $g\in C(S^{\La\cup\Sigma})$.
	Then, $g\in C(S^{\La\cup\Sigma})$ is a linear sum of the 
	indicator functions $1_{(\e_\La, \e_\Sigma)}=1_{\e_\La}1_{\e_\Sigma}$
	on $S^{\La\cup\Sigma}=S^\La\times S^\Sigma$,  
	which defines a function in $L^2(\mu)$.
	It is sufficient to prove our statement when $g$ is
	such function.
	Since $1_{\e_\La}$ is $\cF_{\La}$-measurable, 
	\[
		E[1_{\e_\La}1_{\e_\Sigma}|\sigma(\cF_\La, \cG_n)]=1_{\e_\La}E[1_{\e_\Sigma}|\sigma(\cF_\La, \cG_n)].
	\]
	Since $\cF_{\La}$ is independent from $\cF_Y$, which includes $\sigma( 1_{\e_\Sigma})$ and $\cG_n$, 
	\[
	E[1_{\e_\Sigma}|\sigma(\cF_\La, \cG_n)] = E[1_{\e_\Sigma}|\cG_n],
	\]
	whose proof can be found for example in Chapter 9 of \cite{Wil91}.
	Thus,
	$
		E[1_{\e_\La}1_{\e_\Sigma}|\sigma(\cF_\La, \cG_n)]=1_{\e_\La} E[1_{\e_\Sigma}|\cG_n]
	$
	and similarly
	$
		E[1_{\e_\La}1_{\e_\Sigma}|\cF_\La]=1_{\e_\La} E_\mu[1_{\e_\Sigma}].
	$
	By the convergence theorem of backward martingales, 
	we have
	\[
		\lim_{n\rightarrow\infty}E[1_{\e_\Sigma}|\cG_n]= E[1_{\e_\Sigma}|\cG_\infty]
	\]	
	in $L^2(\mu)$ where $\cG_{\infty}=\cap_n \cG_n$.  Hence in order to prove our theorem, it is sufficient to prove that 
	$E[1_{\e_\Sigma}|\cG_\infty]= E[1_{\e_\Sigma}]$ in $L^2(\mu)$.
	
	Let $\pr_x\colon S^X\rightarrow S$ be the projection $\pr_x(\e)=\eta_x$ for any $\e\in S^X$,
	which we view as a random variable on $S^X$, and we consider the system 
	of random variables $(\pr_x)_{x\in Y}$ of $S^X$.
	Since $\cF_Y= \sigma ( \pr_x, x \in Y)$, we have
	\[
		\cG_n \subset\cF_Y.
	\]
	For any finite permutation $\varrho$ of $Y$, if we take $m$ sufficiently large,
	then we have $\varrho(x)=x$ for any $x\in Y\setminus\La_m$.
	This implies that $\varrho$ induces a permutation of the finite set $\La_m$.
	For any $A\in \cG_m$, by the construction, we have $A^\varrho=A$
	since the conserved quantity $\bsxi_{\!\La_m}$ is invariant under permutations of
	the components of $\La_m$.  This implies in particular that 
	$\cG_\infty=\bigcap_{n\in\N}\cG_n\subset\mathfrak{E}$, where $\mathfrak{E}$ is the
	exchangeable $\sigma$-algebra corresponding to the random variables $(\pr_x)_{x\in Y}$.
	By the $0$-$1$ law of Hewitt-Savage given in Corollary \ref{cor: HS},
	we see that $\mathfrak{E}$ hence the $\sigma$-algebra $\cG_\infty$ is $\mu$-trivial.
	Since $E[1_{\e_\Sigma}|\cG_\infty]$ is measurable for $\cG_\infty$, we see that we have
	$E[1_{\e_\Sigma}|\cG_\infty] = E[E[1_{\e_\Sigma}|\cG_\infty]]=E[1_{\e_\Sigma}]$ almost surely.
	This proves that
	\begin{align*}	
			\lim_{n\rightarrow\infty}E[1_{\e_\La}1_{\e_\Sigma}|\sigma(\cF_\La,\cG_n)]&=
			\lim_{n\rightarrow\infty}1_{\e_\La}E[1_{\e_\Sigma}|\cG_n]
			=1_{\e_\La}E[1_{\e_\Sigma}|\cG_\infty]\\
			&=1_{\e_\La}E[1_{\e_\Sigma}]=E[1_{\e_\La}1_{\e_\Sigma}|\cF_\La]
	\end{align*}
	in $L^2(\mu)$ as desired.
\end{proof}

We next review the implication of the irreducibly quantified condition.

\begin{lemma}\label{lem: check1}
	Assume that $(S,\phi)$ is irreducibly quantified.
	Let $\La\in \sI$ be a finite sublocale.
	Suppose $f\colon S^X\rightarrow\R$ is a \emph{local function} which satisfies $\nabla_e f=0$
	for any $e\in E_\La$.  Then, for any $\La' \in \sI$ such that $\La \subset \La'$ and $f \in C(S^{\La'})$, $f$ is measurable with respect to $\sigma(\bsxi_{\!\La}, \cF_{\La' \setminus\La})$.
\end{lemma}

\begin{proof}
	Let $\La' \in \sI$ satisfy $\La \subset \La'$ and $f \in C(S^{\La'})$. 
	Suppose we have $\e,\e'\in S^{\La'}$ such that $\e_x=\e'_x$ for $x \in \La' \setminus \La$ and $\bsxi_\La(\e)=\bsxi_{\!\La}(\e')$. 
	Since $(S,\phi)$ is irreducibly quantified,
	there exists a path $\vec \varphi=
	(\varphi_1,\ldots,\varphi_n)$
	from $\e|_{\La}$ to $\e'|_{\La}$.  We let $\bse=(e_1,\ldots,e_n)$ be the sequence
	of edges $e_i\in E_\La$ defining the path $\vec\varphi$.
	If we let $\e_0=\e$ and $\e_i=\e_{i-1}^{e_i}$ for $i=1,\ldots,n$,
	then $\vec\varphi'=(\varphi'_1,\ldots,\varphi'_n)$ for $\varphi'_i=(\e_i,\e_{i+1})$
	defines a path from $\e$ to $\e'$.  Since $\nabla_e(f)=0$ for $e\in E_\La$,
	we have $f(\e_i)=f(\e_{i+1})$ for any $i=0,\ldots, n-1$.  This shows that $f(\e)=f(\e')$,
	hence we see that $f$ is measurable with respect to $\sigma(\bsxi_{\!\La}, \cF_{\La' \setminus\La})$.
\end{proof}

We now give the proof of Proposition \ref{prop: main}.

\begin{proof}[Proof of Proposition \ref{prop: main}]
	Since $f\in L^2(\mu)$, it is sufficient to prove that
	$
		\pair{f,g}=\pair{\pi^\La f,g}
	$
	for any local function $g\in C_\loc(S^X)$.
	If we note that $\pair{\pi^\La f,g}=\pair{\pi^\La f,\pi^\La g}=\pair{ f,\pi^\La g}$,
	then we see that it is sufficient to prove
	that 
	\[
		\pair{f,g}=\pair{f,\pi^\La g}
	\] 
	for any local function $g$.  
	We fix a local function $g$.
	Since $Y$ is a connected infinite graph, there exists an approximation 
	$\La_n\uparrow Y$ by finite connected sets.  Since $f$ is $\cF_{\La\cup Y}$-measurable,
	we see that $f=\pi^{\La\cup Y}f$ and
	\begin{equation}\label{eq:0}
		\pair{f,g}=\pair{f, \pi^{\La\cup Y} g}.
	\end{equation}
	Since $g$ is a local function, there exists $W\in\sI$ such that $g\in C(S^W)$, and then $(\La\cup Y)\cap W=(\La\cup \La_n)\cap W$ for $n$ sufficiently
	large. Hence for such $n$, since $\mu$ is a product measure and $(Y \setminus \La_n) \cap W =\emptyset$, we have
	\[
	E[g | \cF_{\La \cup Y}]= E[g | \sigma( \cF_{\La \cup \La_n}, \cF_{Y \setminus \La_n})] = E[ g |  \cF_{\La \cup \La_n}]
	\]
	and so
	\begin{equation}\label{eq:1}
		\pair{f,\pi^{\La\cup Y}g}=\pair{f,\pi^{\La\cup\La_n}g}.
	\end{equation}
	Note that for $\La_n$ and any $e\in E_{\La_n}$, our condition implies that we have 
	\[
		\nabla_e(\pi^{\La\cup\La_n}f)=0,
	\]
	hence by Lemma \ref{lem: check1}, this shows that $\pi^{\La\cup\La_n}f$ is 
	measurable for $\sigma(\cF_{\La},\bsxi_{\!\La_n})$.
	Since $\sigma(\cF_\La,\bsxi_{\!\La_n})\subset\sigma(\cF_{\La\cup \La_n})$, 
	we have
	\begin{align*}
		\pair{f, \pi^{\La\cup \La_n}g}
		&=\pair{\pi^{\La\cup \La_n}f,  \pi^{\La\cup \La_n}g}=\pair{E[f|\sigma(\cF_\La, \bsxi_{\!\La_n})],  \pi^{\La\cup \La_n}g}\\
		&=\pair{E[f|\sigma(\cF_\La,\bsxi_{\!\La_n})], E[g|\sigma(\cF_\La,\bsxi_{\!\La_n})]}
		=\pair{f, E[g|\sigma(\cF_\La,\bsxi_{\!\La_n})]}.
	\end{align*}
       Since $(Y \setminus \La_n) \cap W =\emptyset$ and $\mu$ is product again, we have $E[g|\sigma(\cF_\La,\bsxi_{\!\La_n})] = E[g|\sigma(\cF_\La,\cG_n)]$ where
	$\cG_n=\sigma(\bsxi_{\!\La_n}, \cF_{Y\setminus\La_n})$ as in \eqref{eq: Gn}.
	Combining this equality with \eqref{eq:0} and \eqref{eq:1}, we have
	\[
		\pair{f,g}=\pair{f,\pi^{\La\cup Y}g}=\pair{f,\pi^{\La\cup \La_n}g}=\pair{f, E[g|\sigma(\cF_\La,\cG_n)]}=\pair{f, E[\pi^{\La \cup Y} g|\sigma(\cF_\La,\cG_n)]}.
	\]
	Hence in order to prove our assertion, it is sufficient to prove that we have
	\[
		\lim_{n\rightarrow\infty}E[\pi^{\La \cup Y} g|\sigma(\cF_\La,\cG_n)]=E[\pi^{\La \cup Y}g|\cF_{\La}]=\pi^\La g
	\]
	in $L^2(\mu)$. This is simply Lemma \ref{lem: converge} applied to $\pi^{\La\cup Y}g$,
	which is a local function in $C(S^{\La\cup\Sigma})$ for $\Sigma\coloneqq W\cap Y \subset Y$.
\end{proof}

%
%
\section{Construction of a Convergent Sequence}\label{sec: boundary}
%
%

In this section, we assume that the locale $(X,E)=(\Z^d,\E^d)$ is the Euclidean lattice.
We consider an interaction $(S,\phi)$ which is irreducibly quantified, and a product measure 
$\mu=\nu^{\otimes\Z^d}$ on $S^{\Z^d}$ for a probability measure $\nu$ fully supported on $S$.  We consider the trivial weight $r=(r_e)_{e\in\E^d}$
such that $r_e \equiv 1$ for any $e\in\E^d$, and we assume that $G=\Z^d$ acts on $(\Z^d,\E^d)$ by translation.

Let $\omega\in\sC_\mu\coloneqq Z^1_{L^2}(S^{\Z^d}_\mu)^G$ be a shift-invariant closed $L^2$-form.
The purpose of this section is to construct a certain sequence of uniform functions $\{\Psi_n\}$
associated to $\omega$, which will be used to prove Theorem \ref{thm: main}.

\begin{definition}\label{def: Euclid}
	For each $n\in\N$,
	we let
	\[
		\La_n\coloneqq[-n,n]^d=\{(x_1,\ldots, x_d)\in\Z^d \mid \forall j\,\,\abs{x_j}\leq n\}.
	\]
	Then we have $\Z^d=\bigcup_{n\in\N}\La_n$.
	 We call $\{\La_n\}_{n\in\N}$ the \emph{Euclidean approximation of $\Z^d$}.
\end{definition}

Recall that for any $\La\subset \Z^d$, we define the boundary $\partial\La$ of $\La$ by 
\[
	\partial\La = \{e\in E\mid o(e)\in\La, t(e)\not\in\La\}.
\]
Since $\abs{\La_n}=(2n+1)^d$ and $\abs{\partial\La_n}=2d(2n+1)^{d-1}$,
we see that $\La_n$ satisfies
\[
			\lim_{n\rightarrow\infty}\abs{\partial\La_n}/\abs{\La_n}=0.
\]
The Euclidean approximation satisfies the following.

\begin{lemma}\label{lem: Euclid}
	Let $\{\La_n\}_{n\in\N}$ be the Euclidean approximation 
	given in Definition \ref{def: Euclid}.
	Then we have
	\begin{equation}\label{eq: temp}
		\sup_n\frac{\abs{\partial\La_n}^2}{\abs{\La_n}^2}\frac{\abs{\La_{2n}}}{\abs{\La_{2n}\setminus\La_n}}
		\diam{\La_{2n}}^2<\infty.
	\end{equation}
\end{lemma}

\begin{proof}
	Note that $\abs{\La_n}=(2n+1)^d$, $\abs{\partial\La_n}=2d(2n+1)^{d-1}$, $\abs{\La_{2n}}=(4n+1)^d$
	and $\diam{\La_{2n}}=d(4n+1)$.  This shows that
	\[
	\frac{\abs{\partial\La_n}^2}{\abs{\La_n}^2}
		\frac{\abs{\La_{2n}}}{\abs{\La_{2n}\setminus\La_n}}\diam{\La_{2n}}^2
		=\frac{(2d)^2(2n+1)^{2d-2}}{(2n+1)^{2d}}\frac{(4n+1)^{d}}{(4n+1)^{d}-(2n+1)^d}d^2(4n+1)^2.
	\]
	Hence, we have
	\[
		\lim_{n\rightarrow\infty}\frac{\abs{\partial\La_n}^2}{\abs{\La_n}^2}
		\frac{\abs{\La_{2n}}}{\abs{\La_{2n}\setminus\La_n}}\diam{\La_{2n}}^2
		=\frac{4^{d+2}d^4}{(4^{d}-2^d)},
	\]
	which proves the inequality \eqref{eq: temp} as desired.
\end{proof}

In what follows, let $\omega\in\sC_\mu=Z_{L^2}^1(S^{\Z^d}_\mu)^G$.
Since $\omega$ is a closed co-local form, there exists $F\in C^0_\col(S^X_\mu)$ such that
$\partial F=\omega$.  We will use $F$ to construct the uniform functions $\Psi_n$ as follows.

\begin{definition}\label{def: psi}
	Let $\{\La_n\}_{n\in\N}$ be the Euclidean approximation given in Definition \ref{def: Euclid}.
	For any $\omega\in\sC_\mu=Z_{L^2}^1(S^{\Z^d}_\mu)^G$, let $F=(F^\La)\in C^0_\col(S^{\Z^d}_\mu)$ such that
	$\partial F=\omega$, and we let $F_n$ be the projection of $F^{\La_{2n}}$ to 
	$(\Ker\partial_{\La_{2n}})^\perp$.   We let
	\begin{equation}\label{eq: psi}
		\Psi_n\coloneqq\frac{1}{(2n+1)^d}\sum_{\tau\in G}\tau(\pi^{\La_n}F_n).
	\end{equation}
	By construction, $\Psi_n\in C^0_\unif(S^{\Z^d})^G$.
\end{definition}

We will separate $\Psi_n$ into a sum for which, the differential of the first converging to $\omega$, and the differential of the second converging to the boundary term $\omega^\dagger$.
We have
\[
	\partial\Psi_n=\frac{1}{(2n+1)^d}\sum_{\tau\in G}\tau(\partial(\pi^{\La_n}F_n)).
\]
The main difficulty in calculation of the differential stems from the fact that for a general co-local function
$f\in C^0_\col(S^{\Z^d}_\mu)$, we may have $\partial(\pi^{\La} f)\neq \partial_{\La}(\pi^\La f)$.
Recall that for any $f\in C^0_\col(S^{\Z^d}_\mu)$ and $\La\in\sI$, we define the boundary differential by
\[
	\partial^\dagger_\La f = \partial(\pi^\La f) - \partial_{\La}(\pi^\La f).
\]

\begin{definition}\label{def: omega}
	Let the notations be as in Definition \ref{def: psi}.  For any $n\in\N$, let
	\begin{align}\label{eq: omega}
		\omega_n&\coloneqq \frac{1}{(2n+1)^d}\sum_{\tau\in G}\tau(\partial_{\La_n}(\pi^{\La_n}F_n)),&
		\omega^\dagger_n&\coloneqq \frac{1}{(2n+1)^d}\sum_{\tau\in G}
		\tau(\partial^\dagger_{\La_n}(\pi^{\La_n}F_n)).
	\end{align}  
	By \eqref{eq: boundary}, 
	we have $\partial\Psi_n=\omega_n+\omega^\dagger_n$.
	We call $\{\omega^\dagger_n\}_{n\in\N}$ the \emph{boundary forms}.
\end{definition}

\begin{proposition}\label{prop: conv0}
	For the sequence $\{\omega_n\}_{n\in\N}$ of Definition \ref{def: omega}, we have $\omega_n\in C^1_\unif(S^{\Z^d})^G$ and
	\begin{equation}\label{eq: converge}
		\lim_{n\rightarrow\infty}\omega_n=\omega
	\end{equation}
	in $C^1_{L^2}(S^{\Z^d}_\mu)^G$.  
\end{proposition}

\begin{proof}
	We first shows that the sum defines a uniform form in $C^1_\unif(S^{\Z^d}_\mu)^G$.
	Note that $\partial_{\La_n}(\pi^{\La_n}F_n)=\pi^{\La_n}(\partial_{\La_{2n}} F_n)=\pi^{\La_n}\omega^{\La_{2n}}=\omega^{\La_n}$.
	We have $\omega^{\La_n}=(\omega^{\La_n}_e)\in C^1(S^{\La_n})$. 	
	In particular, $\omega^{\La_n}_e\neq 0$ only if $e\in E_{\La_n}$.
	For any $\tau\in G$, we have $\tau(\omega^{\La_n})_e=\tau(\omega^{\La_n}_{\tau^{-1}(e)})$,
	hence
	\[
		\omega_{n,e}\coloneqq\frac{1}{(2n+1)^d}
		\sum_{\tau\in G}\tau(\omega^{\La_n})_e= \frac{1}{(2n+1)^d}\sum_{\substack{\tau\in G\\\tau^{-1}(e)\in E_{\La_n}}}\tau(\omega^{\La_n}_{\tau^{-1}(e)})
	\]
	is a finite sum for any $e\in E$, hence $\omega_{n,e}$ is a local function in $C_\loc(S^{\Z^d})$.
	The fact that $\omega_n$ is $G$-invariant follows from the fact that it is the sum over
	the $G$-translates of $\omega^{\La_n}$.
	This implies in particular that $\omega_n=(\omega_{n,e})\in C^1_\unif(S^{\Z^d})^G$.
	For the convergence, note that
	\begin{equation}\label{eq: terms}
		\omega_e-\omega_{n,e}=\frac{1}{(2n+1)^d}
		\sum_{\substack{\tau\in G\\ \tau^{-1}(e)\in E_{\La_n}}}
		\bigl(\omega_e-\tau\bigl(\omega^{\La_n}_{\tau^{-1}(e)}\bigr)\bigr)
		+\frac{(2n+1)^d-|Ge\cap E_{\La_n}|}{(2n+1)^d}\omega_e
	\end{equation}
	where $Ge=\{\tau e \mid \tau \in G\}$.
	Since $\omega$ is $G$-invariant, we have
	$\tau\bigl(\omega^{\La_n}_{\tau^{-1}(e)}\bigr)=\tau(\omega)_e^{\tau(\La_n)}=\omega_e^{\tau(\La_n)}$.
	From the triangular equality for norms, we have
	\begin{equation}\label{eq: normdiff}
		\dabs{\omega_e-\omega_{n,e}}_\mu\leq 
		\frac{1}{(2n+1)^d}\sum_{\substack{\tau\in G\\ \tau^{-1}(e)\in E_{\La_n}}}
		\bigl\Vert\omega_e-\omega_e^{\tau(\La_n)} \bigr\Vert_\mu
		+\biggl|\frac{(2n+1)^d-|Ge\cap E_{\La_n}|}{(2n+1)^d}\biggr|\dabs{\omega_e}_\mu.
	\end{equation}
	Noting that $|Go(e)\cap \La_n|=(2n+1)^d$, we have
	\[
		\biggl|\frac{(2n+1)^d-|Ge\cap E_{\La_n}|}{(2n+1)^d}\biggr|
		=\frac{|Ge\cap\partial\La_n|}{(2n+1)^d}
		\leq \frac{|\partial\La_n|}{(2n+1)^d}.
	\]
	Hence, noting $|\La_n|=(2n+1)^d$, the second term of the right hand side of
	\eqref{eq: normdiff} converges to \textit{zero} as $n\rightarrow\infty$.
	Finally, we evaluate the first term of the right hand side of \eqref{eq: normdiff}.
	For any $\varepsilon>0$, there exists $\La\in\sI$ such that
	\[
		\dabs{\omega_e-\omega^\La_e}_\mu \le \varepsilon.
	\]
	For this $\La$, let $H_n\coloneqq\{\tau\in G\mid \La\subset\tau\La_n\}$.
	Then
	\begin{align*}
		\sum_{\substack{\tau\in G\\ \tau^{-1}(e)\in E_{\La_n}}}\Dabs{\omega_e-\omega_e^{\tau(\La_n)}}_\mu
		&=\sum_{\substack{\tau\in H_n\\ \tau^{-1}(e)\in E_{\La_n}}}\Dabs{\omega_e-\omega_e^{\tau(\La_n)}}_\mu
		+\sum_{\substack{\tau\in G\setminus H_n\\ \tau^{-1}(e)\in E_{\La_n}}}\Dabs{\omega_e-\omega_e^{\tau(\La_n)}}_\mu\\
		&\leq\sum_{\substack{\tau\in H_n\\ \tau^{-1}(e)\in E_{\La_n}}}\dabs{\omega_e-\omega_e^{\La}}_\mu
		+ |\{\tau\in G\setminus H_n \mid o(e)\in\tau(\La_n) \}|\dabs{\omega_e}_\mu\\
		&\le  |H_n e\cap E_{\La_n}|  \varepsilon
		+ |\{\tau\in G\mid o(e)\in\tau(\La_n),  (\tau(\La_n))^c\cap\La\neq\emptyset \}|\dabs{\omega}_\sp
	\end{align*}
	where for the first inequality, we use $ \Dabs{\omega_e-\omega_e^{\tau(\La_n)}}_\mu \le \Dabs{\omega_e-\omega_e^{\La}}_\mu$ if $\La\subset \tau\La_n$ and $ \Dabs{\omega_e-\omega_e^{\tau(\La_n)}}_\mu \le \Dabs{\omega_e}_\mu$ for any $\tau$. 
	Let $\ell\coloneq\max_{x\in\La} d_X(o(e),x)$
	and $\Sigma^{[\ell]}\coloneqq\{x\in\Sigma \mid d_X(x,\Sigma^c) \le \ell\}$.
	Then we have
	\[
		\frac{1}{ (2n+1)^d}
		\sum_{\substack{\tau\in G\\ \tau^{-1}(e)\in E_{\La_n}}}\Dabs{\omega_e-\omega_e^{\tau(\La_n)}}_\mu
		\le \varepsilon
		+ \frac{\bigl|\La_n^{[\ell]}\bigr|}{ (2n+1)^d}\dabs{\omega}_\sp.
	\]
	By definition, we see that $\bigl|\La_n^{[\ell]}\bigr|\leq (2n+1)^d -(2(n-\ell)+1)^d \le 2\ell d (2n+1)^{d-1}$.
	This shows that the first term of the right hand side of
	\eqref{eq: normdiff} is $\leq \varepsilon$ when $n\rightarrow\infty$.
	This shows that we have $\displaystyle \lim_{n\rightarrow\infty}\dabs{\omega-\omega_n}_\sp=0$ as desired.
\end{proof}

The key point in the proof of Theorem \ref{thm: main} is proving that the 
sequence of boundary forms
$\{\omega_n^\dagger\}_{n\in\N}$ converges weakly to some form $\omega^\dagger$ in $\cC$.
We proved in \S\ref{sec: key-lemmas} certain bounds for norms of differentials of local 
functions which will be used to uniformly bound the norms of $\omega_n^\dagger$.
The existence of the bound
 implies that there exists a subsequence of $\omega_n^\dagger$ which converges weakly to 
some $\omega^\dagger\in\sC_\mu$.
Then in \S\ref{sec: proof}, we show that  $\omega^\dagger$ in fact
is a form in $\cC$ by applying the criterion for the locality shown in  \S\ref{sec: key-lemmas}. We will use this fact to prove Theorem \ref{thm: main}
first for the case of the Euclidean lattice $\Z^d=(\Z^d,\E^d)$.  
Then we will prove the case of a general locale $(\Z^d,E)$
with action of $G=\Z^d$ by translation,
by reducing to the case of the Euclidean lattice.

%
%
\subsection{Convergence of the Boundary Sequence}\label{subsec: convergence}
%
%

Let the notations be as in \S\ref{subsec: MT},
and assume that the locale $(X,E)=(\Z^d,\E^d)$.
We will use the bound of Proposition \ref{prop:b2} to uniformly bound the norms of the boundary forms
$\{\omega^\dagger_n\}$ of Definition \ref{def: omega}.
Consider a closed $G$-invariant $L^2$-form $\omega\in \sC_\mu=C^1_{L^2}(S^{\Z^d})\cap Z^1_\col(S^{\Z^d}_\mu)^G$.
We let $\{\La_n\}_{n\in\N}$ be the Euclidean approximation of Definition \ref{def: Euclid}.
Furthermore, let $F_n$ be as in Definition \ref{def: psi}, and we recall that
\begin{equation*}
	\omega^\dagger_n=\frac{1}{(2n+1)^d}\sum_{\tau\in G}\tau(\partial^\dagger_{\La_n}F_n).
\end{equation*}
Let $G^e_n\coloneqq\{\tau\in G\mid e\in\tau(\partial\La_n)\}$.
Since $\tau(\partial^\dagger_{\La_n}F_n)_e\neq 0$ if and only if $e\in\tau(\partial\La_n)$
or $\bar e\in\tau(\partial\La_n)$, if we let
\begin{align}\label{eq: pm}
	\omega^+_{n,e}&\coloneqq \frac{1}{(2n+1)^d}\sum_{\tau\in G^e_n}\tau(\partial^\dagger_{\La_n}F_n),&
	\omega^-_{n,e}&\coloneqq \frac{1}{(2n+1)^d}\sum_{\tau\in G^{\bar e}_n}\tau(\partial^\dagger_{\La_n}F_n),
\end{align}
then we have $\omega^\dagger_{n,e}=\omega^+_{n,e}+\omega^-_{n,e}$.
We denote $\omega^{\pm}_{n,e}$ to mean $\omega^{+}_{n,e}$ or $\omega^{-}_{n,e}$.
We next prove the existence of a uniform bound of the norms of $\{\omega^{\pm}_{n,e}\}$ for any $e\in E$.

\begin{proposition}\label{prop: conv1}
	For any $e\in E$, let $\omega^\pm_{n,e}$  be the forms defined in \eqref{eq: pm}.
	Then for any $n\in\N$, we have $\omega^\dagger_n\in\prod_{e\in E}L^2(\mu)$, and for any $e\in E$, we have
	\begin{equation}\label{eq: bound}
		\sup_n\dabs{\omega^{\pm}_{n,e}}_\mu<\infty.
	\end{equation}
	This implies in particular that $\{\omega^{\pm}_{n,e}\}$ contains a subsequence which converges weakly to an element $\omega^\pm_{e}$ in $L^2(\mu)$.
\end{proposition}

\begin{proof}
	From the triangular inequality for norms and the fact that the measure is invariant under the action of $G$,
	we have
	\begin{align*}
		(2n+1)^d\dabs{\omega^+_{n,e}}_\mu&\leq
		\sum_{\tau\in G^e_n}
		\Dabs{\tau(\partial^\dagger_{\La_n}F_n)_e}_\mu
		=\sum_{\tau\in G^e_n}
		\Dabs{\tau((\partial^\dagger_{\La_n}F_n)_{\tau^{-1}(e)})}_\mu\\
		&=\sum_{\tau\in G^e_n}
		\Dabs{(\partial^\dagger_{\La_n}F_n)_{\tau^{-1}(e)}}_\mu
		\leq |\partial\La_n|\dabs{\partial^\dagger_{\La_n}F_n}_\sp.
	\end{align*}
	This shows that
	\[
		\dabs{\omega^+_{n,e}}_\mu\leq\frac{\abs{\partial\La_n}}{(2n+1)^d}\Dabs{\partial^\dagger_{\La_n}F_n}_\sp.
	\]
	Combining this with Proposition \ref{prop:b2}, noting that $\partial_{\La_{2n}}F_n=\omega^{\La_{2n}}$
	and $\dabs{\omega^{\La_{2n}}}_\sp\leq\dabs{\omega}_\sp$,
	we have
	\begin{align*}
		\dabs{\omega^+_{n,e}}^2_\mu&\leq C\frac{\abs{\partial\La_n}^2}{(2n+1)^{2d}}
		\frac{\abs{\La_{2n}}}{\abs{\La_{2n}\setminus\La_n}}\diam{\La_{2n}}^2
			\dabs{\omega}^2_\sp.
	\end{align*}
	Hence \eqref{eq: bound} follows from Lemma \ref{lem: Euclid}.	
	The last assertion follows from the fact that any bounded sequence in a Hilbert space 
	has a weakly convergent subsequence.  The statement for $\omega^-_{n,e}$ may be proved in a similar manner.
\end{proof}

In what follows, for each $e\in E$,
we take subsequences of $\{\omega^{\pm}_{n,e}\}_{n\in\N}$ which
converges weakly to $\omega^{\pm}_{e}$ in $L^2(\mu)$.
Since $\omega^\dagger_n$ and $\omega^\pm_n$ by construction is invariant with respect to the action of $G$,
the local functions $\omega^\dagger_{n,e}=\omega^+_{n,e}+\omega^-_{n,e}$ depends only on the class of $e$ in the set of orbits
$E/G$, which is finite.
Hence we may take subsequences $\{\omega^\pm_{n_k}\}_{k\in\N}$ such that 
$\{\omega^\pm_{n_k,e}\}_{k\in\N}$ converges weakly to $\omega^\pm_{e}$ in $L^2(\mu)$
for each $e\in E$.   Since the structure of the Hilbert space on $C^1_{L^2}(S^{\Z^d}_\mu)^G$ is
induced from the inclusion $C^1_{L^2}(S^{\Z^d}_\mu)^G\hookrightarrow\prod_{e\in E/G}L^2(\mu)$,
weakly convergence on each $e\in E$ implies that $\{\omega^\dagger_{n_k}\}$ itself is
weakly convergent.
We let $\omega^\dagger\coloneqq(\omega^\dagger_e)\in C^1_{L^2}(S^{\Z^d}_\mu)^G$
be the weak limit of $\omega^\dagger_{n_k}$.

\begin{proposition}\label{prop: conv2}
	We have $\omega^\dagger\in \sC_\mu=Z^1_{L^2}(S^{\Z^d}_\mu)^G$.
\end{proposition}

\begin{proof}
	By the definition of $\omega^\dagger_n$ and \eqref{eq: boundary}, the form
	$\omega^\dagger$ is the weak limit of a subsequence of the sequence
	\begin{align*}
		\omega^\dagger_n &=\partial\Psi_n-\omega_n,
	\end{align*}
	where $\omega_n$ is defines as in Definition \ref{def: omega}.
	By Proposition \ref{prop: conv0}, we see that $\displaystyle \lim_{n\rightarrow\infty}\omega_n=\omega$
	strongly in $C^1_{L^2}(S^{\Z^d})^G$.  If we let
	\[
		u_n\coloneqq\partial\Psi_n-\omega,
	\]
	then by construction, $u_n\in \sC_\mu=Z^1_{L^2}(S^{\Z^d})^G$. 
	Since a subsequence of $\omega^\dagger_n$ converges weakly to $\omega^\dagger$,
	a subsequence of $u_n$ also converges to $\omega^\dagger$.
	This shows that $\omega^\dagger$ is in the weak closure of 
	$Z^1_{L^2}(S^{\Z^d}_\mu)^G$.  Since by \cite{Yos95}*{Theorem V.1.11}, the weak closure
	of a linear subspace of a Hilbert space coincides with its strong closure,
	this implies that $\omega^\dagger$ is in the strong closure of $Z^1_{L^2}(S^{\Z^d})^G$.
	Our assertion now follows from Lemma \ref{lem: closed} which asserts that 
	$Z^1_{L^2}(S^{\Z^d})^G$ is closed in $\prod_{e\in E/G}L^2(\mu)$.  
\end{proof}

%
%
%
\section{Proof of the Main Theorem}\label{sec: proof}
%
%
%

In this section, we will prove Theorem \ref{thm: main}.
We will first prove that when the underlying locale is the
Euclidean lattice $(\Z^d,\E^d)$,
the form $\omega^\dagger$ in Proposition \ref{prop: conv2}
is in fact a uniform form.  We will then use this fact to prove Theorem \ref{thm: main}
for the case of the Euclidean lattice.  The proof of Theorem \ref{thm: main}
for the case of a general finite range Euclidean lattice $(\Z^d,E)$ 
will be proved by reducing to the case of the Euclidean lattice $(\Z^d,\E^d)$.

%
%
\subsection{Proof of the Main Theorem for the Euclidean Lattice case}\label{subsec : Euclidean}
%
%

Let the notations be as in \S\ref{subsec: convergence}.
In this subsection, we assume that the locale is the Euclidean lattice $(\Z^d,\E^d)$,
with the action of $G=\Z^d$ by translation.
We prove that the form $\omega^\dagger$ 
in Proposition \ref{prop: conv2} obtained as the limit of the boundary forms $\{\omega^\dagger_n\}$ 
is in fact a uniform form.  We will then use this fact to prove Theorem \ref{thm: main}
for the case of the Euclidean lattice.  

For any $e \in \E^d$, $t(e)=o(e)+1_j$ or $t(e)=o(e)-1_j$, where $1_j$ is the element in $\Z^d$
with $1$ in the $j$-th component and $0$ in the other components. Let $o(e)_j$ be the $j$-th 
component of $o(e)$, and
$Y^e\coloneqq\{ (x_1,\ldots,x_d)\in\Z^d\mid x_j \leq o(e)_j\}$ if $t(e)=o(e)+1_j$,
and $Y^e\coloneqq\{ (x_1,\ldots,x_d)\in\Z^d\mid x_j \geq o(e)_j\}$ otherwise.
In other words, if we split the vertices in $\Z^d$ into two parts via the hyperplane 
perpendicular to $e$ so that $o(e)$ and $t(e)$ are in the different parts, then $Y^e$ is defined as the part which contains $o(e)$ and does not contain $t(e)$.

We let $\{\La_n\}_{n\in\N}$ be the Euclidean approximation given
in Definition \ref{def: Euclid}.
For $e \in \E^d$, let
\[
	G^e_n=\{\tau\in G\mid e\in\tau(\partial\La_n)\}
\]
as before, and for $\tilde e\in \E^d$, let $G^{e,\tilde e}_n\coloneqq G^e_n\cap G^{\tilde e}_n$.

\begin{lemma}\label{lem: twoedges}
	Let $\{\La_n\}_{n\in\N}$ be the Euclidean approximation of Definition \ref{def: Euclid}.
	For any $e\in \E^d$ and $\tilde e\in E_{Y^e}=\{e' \in \E^d \mid o(e'),t(e') \in Y^e\}$, we have
	\[
		\lim_{n\rightarrow\infty}\frac{\abs{G^{e,\tilde e}_n}}{\abs{G^e_n}}=0.
	\]
\end{lemma}

\begin{proof}
	Note that $\abs{G^e_n}=\frac{1}{2d}\abs{\partial\La_n}=(2n+1)^{d-1}$.
	If $\tilde e$ is in the same direction as that of $e$, then $G_n^{e,\tilde e}=\emptyset$.
	If $\tilde e$ is in the exactly opposite direction as that of $e$, then $G_n^{e,\tilde e}=\emptyset$
	for $n$ sufficiently large.  Otherwise, $G^{e,\tilde e}_n$ counts the positions of shifts of $\La_n$ touching both $e$ and $\tilde e$ at their origins, 
	hence has $(d-2)$-dimension of freedom.
	This shows that $\abs{G^{e,\tilde e}_n}\leq(2n+1)^{d-2}$, which proves our assertion.
\end{proof}

\begin{remark}\label{rem:finite-range}
We cannot generalize Lemma \ref{lem: twoedges} straightforwardly to models with more general
finite range interactions, namely models on a locale $(\Z^d, E)$ where $E$ is not necessarily $\mathbb{E}^d$. If $\mathbb{E}^d \subset E$, to prove the next proposition, we only need the estimates in Lemma \ref{lem: twoedges} for $e, \tilde{e} \in \mathbb{E}^d$, so we can generalized the main result to this case. However, without the condition $\mathbb{E}^d \subset E$, the generalization is not obvious. 
\end{remark}

Having this estimate, we may now prove that $\omega^\dagger$ is uniform.

\begin{proposition}\label{prop: local}
	We have $\omega^\dagger\in C^1_\unif(S^{\Z^d})^G$.
\end{proposition}

\begin{proof}
	By Proposition \ref{prop: conv2}, the form 
	$\omega^\dagger$ is invariant with respect to the action of $G$.
	Hence
	it is sufficient to prove that $\omega^+_e$ and $\omega^-_e$ are local functions
	for any $e\in E$.
	We first prove that $\omega^+_{e}$ is a local function.
	By construction, for any $n\in\N$,
	the function $\omega^+_{n,e}$ is a function in $C(S^{Y^e\cup\{t(e)\}})$.
	Let  $\La\coloneqq\{o(e),t(e)\}$ and $\La'\subset Y^e\setminus\{o(e)\}$ be a finite set.
	By Proposition \ref{prop: main}, it is sufficient to prove that
	\[
		\nabla_{\!\tilde e}(\pi^{\La\cup\La'}\omega^+_e)=0
	\]
	for any $\tilde e\in E_{\La'}$.  
	By Lemma \ref{lem: bound}, the differential $\nabla_{\!\tilde e}$
	is continuous for local functions. Also, $C(S^{\La \cup \La'})$ is a finite dimensional space and the weak topology and the strong topology are same for a finite dimensional Hilbert spaces. 
	Hence, noting that 
	\[
		\dabs{\nabla_{\!\tilde e}(\pi^{\La\cup\La'}\omega^+_{n,e})}_\mu
		=\dabs{\pi^{\La\cup\La'}(\nabla_{\!\tilde e}(\omega^+_{n,e}))}_\mu
		\leq\dabs{\nabla_{\!\tilde e}\omega^+_{n,e}}_\mu,
	\]
	it is sufficient to prove that
	\[
		\lim_{n\rightarrow\infty}\dabs{\nabla_{\!\tilde e}\omega^+_{n,e}}_\mu=0.
	\]
	Since $\tilde e\in E_{\La'}$, we have $\{o(e),t(e)\}\cap\{o(\tilde e),t(\tilde e)\}=\emptyset$.
	This shows that
	\[
		\nabla_{\!\tilde e}\omega^+_{n,e}=\nabla_{\!\tilde e}\nabe\biggl(\frac{1}{(2n+1)^d}\sum_{\tau\in G^e_n}\tau(F_n)\biggr)
		=\nabe\nabla_{\!\tilde e}\biggl(\frac{1}{(2n+1)^d}\sum_{\tau\in G^e_n}\tau(F_n)\biggr).
	\]
	By Lemma \ref{lem: bound}, for any local function $f$, we have
	\[
		\dabs{\nabe f}_\mu^2\leq 4C_{\phi,\nu}\dabs{f}^2_\mu,
	\]
	where $C_{\phi,\nu} \geq 1$ is the constant given in Definition \ref{def: MB}.
	Hence we have
	\[
		\dabs{\nabla_{\!\tilde e}\omega^+_{n,e}}_\mu
		\leq  \frac{2C_{\phi,\nu}^{1/2}}{(2n+1)^d}\Big\Vert
		\sum_{\tau\in G^e_n}\nabla_{\!\tilde e}(\tau(F_n))\Big\Vert_\mu
		\leq  \frac{2C_{\phi,\nu}^{1/2}}{(2n+1)^d}
		\sum_{\tau\in G^e_n}\dabs{\nabla_{\!\tilde e}(\tau(F_n))}_\mu.
	\]
	Furthermore, if $\tilde e\in E_{\tau (\La_n)}$, then we have $\nabla_{\!\tilde e}(\tau(F_n))=\omega^{\tau(\La_n)}_{\tilde e}$ hence $\dabs{\nabla_{\!\tilde e}(\tau(F_n))}_\mu\leq\dabs{\omega_{\tilde e}}_\mu\leq\dabs{\omega}_\sp$,
	and if $\{o(\tilde e),t(\tilde e)\}\cap\tau(\La_n)=\emptyset$, then we have $\nabla_{\!\tilde e}(\tau(F_n))=0$.
	This shows that
	\begin{equation}\label{eq: term}
		\sum_{\tau \in G^e_n}\dabs{\nabla_{\!\tilde e}(\tau(F_n))}_\mu\leq\abs{G^e_n}\dabs{\omega}_\sp
		+\sum_{\tau\in G^{e,\tilde e}_n}\dabs{\nabla_{\!\tilde e}(\tau(F_n))}_\mu +  \sum_{\tau\in G^{e,\bar{\tilde e}}_n}\dabs{\nabla_{\!\tilde e}(\tau(F_n))}_\mu.
	\end{equation}
	Since $\abs{G^e_n}\leq\abs{\partial\La_n}$,
	the first term of \eqref{eq: term} satisfies
	\[
		\lim_{n\rightarrow\infty}\frac{\abs{G^e_n}}{(2n+1)^d}\dabs{\omega}_\sp
		\leq\lim_{n\rightarrow\infty}\frac{\abs{\partial\La_n}}{\abs{\La_n}}\dabs{\omega}_\sp=0.
	\]
	For the second and third terms of \eqref{eq: term}, the fact that the measure is invariant under 
	the action of the group, we see that
	$\dabs{\nabla_{\tilde e}(\tau(F_n))}_\mu=\dabs{\nabla_{\tau^{-1}(\tilde e)}(F_n)}_\mu$.
	By Proposition \ref{prop:b2}, if $\tau\in G^{e,\tilde e}_n$ or $\tau\in G^{e,\bar{\tilde e}}_n$, we have
	\begin{align*}
		\dabs{\nabla_{\tau^{-1}(\tilde e)}(F_n)}^2_\mu
		\leq\dabs{\partial^\dagger_{\La_n} F_n}^2_\sp
		&\leq C\frac{\abs{\La_{2n}}}{\abs{\La_{2n}\setminus\La_n}}\diam{\La_{2n}}^2\dabs{\omega^{\La_{2n}}}^2_\sp\\
		&\leq C\frac{\abs{\La_{2n}}}{\abs{\La_{2n}\setminus\La_n}}\diam{\La_{2n}}^2\dabs{\omega}^2_\sp.
	\end{align*}
	Hence we have
	\begin{align*}
		\limsup_{n\rightarrow\infty}\dabs{\nabla_{\!\tilde e}\omega^+_{n,e}}^2
		&\leq 4CC_{\phi,\nu}\limsup_{n\rightarrow\infty}
		\frac{(\abs{G^{e,\tilde e}_n} + \abs{G^{e,\bar{\tilde e}}_n})^2}{(2n+1)^{2d}}
		\frac{\abs{\La_{2n}}}{\abs{\La_{2n}\setminus\La_n}}\diam{\La_{2n}}^2\dabs{\omega}^2_\sp\\
		&=4CC_{\phi,\nu}\limsup_{n\rightarrow\infty}
		\frac{(\abs{G^{e,\tilde e}_n} + \abs{G^{e,\bar{\tilde e}}_n})^2}{\abs{G^e_n}^2}\frac{\abs{G^e_n}^2}{(2n+1)^{2d}}
		\frac{\abs{\La_{2n}}}{\abs{\La_{2n}\setminus\La_n}}\diam{\La_{2n}}^2\dabs{\omega}^2_\sp.
	\end{align*}
	By Lemma \ref{lem: Euclid}, noting that 
	$\abs{G^e_n}\leq \abs{\partial\La_n}$ and $(2n+1)^d=\abs{\La_n}$, we see that
	\[
	\sup_n	\frac{\abs{G^e_n}^2}{(2n+1)^{2d}}
		\frac{\abs{\La_{2n}}}{\abs{\La_{2n}\setminus\La_n}}\diam{\La_{2n}}^2\dabs{\omega}^2_\sp<\infty.
	\]
	Hence, Lemma \ref{lem: twoedges} implies
	that
	\[
		\limsup_{n\rightarrow\infty}\dabs{\nabla_{\!\tilde e}\omega^+_e}_\mu=0
	\]
	as desired.  The proof of the locality for $\omega^-_e$ is obtained in a similar fashion,
	by replacing $G^e_n$ with $G^{\bar e}_n$.  This gives our assertion.
\end{proof}

We will use \cref{prop: local} to give a proof of \cref{thm: main} for the case of the Euclidean lattice.
We let $\{\La_n\}_{n\in\N}$ be the Euclidean approximation
of Definition \ref{def: Euclid}.
We assume that $(S,\phi)$ is irreducibly quantified, and that $((S,\phi),\nu)$ has a uniformly bounded 
spectral gap.
In what follows, for $\La \in \sI$, let $C^0(S^\La_\mu)\coloneqq\{f\in C(S^\La)\mid E_\mu[f]=0\}$
and $A^0(S^\La)=C^0(S^\La_\mu)/\Ker\partial_\La$.
If we let 
\[
	A^0_\col(S^{\Z^d})\coloneqq\varprojlim_\La A^0(S^\La),
\] 
then
Proposition \ref{prop: SES} gives the isomorphism $A^0_\col(S^{\Z^d}_\mu)\cong Z^1_\col(S^{\Z^d}_\mu)$, 
hence an isomorphism 
 \[
 	A^0_\col(S^{\Z^d}_\mu)^G\cong Z^1_\col(S^{\Z^d}_\mu)^G.
\]
Let $1_j\in\Z^d$ be the element with $1$ in the $j$-th component and $0$ in the other components.
We may construct an $\R$-linear homomorphism
\begin{equation}\label{eq: delta}
	\delta\colon A^0_\col(S^{\Z^d}_\mu)^G\rightarrow\bigoplus_{j=1}^d H^0_\col(S^{\Z^d}_\mu),   \qquad f\mapsto ((1-\tau_{1_j})f).
\end{equation}
By construction, any element of $C^0_\col(S^{\Z^d}_\mu)^G$ is in $\Ker\delta$.
The map $\delta$ is well-defined on $A^0_\col(S^{\Z^d}_\mu)^G$ 
since $H^0_\col(S^{\Z^d}_\mu)\subset C^0_\col(S^{\Z^d}_\mu)^G$ by \cite{BS21}*{Lemma 4.7}.
We have the following.
\begin{lemma}\label{lem: H}
	The map $\delta$ of \eqref{eq: delta} induces an isomorphism
	\begin{equation}\label{eq: OK}
		\cV\xrightarrow\cong\bigoplus_{j=1}^d\C(S).
	\end{equation}
\end{lemma}
\begin{proof}
	By Lemma \ref{lem: tau}, we have $ (1-\tau_{1_j})\frA^j_{\xi^{(i)}}=\sum_{x\in X}\xi^{(i)}_{x}$, which shows that 
	$\delta$ induces an inverse isomorphism of the map mapping $\sum_{j=1}^d\sum_{i=1}^{c_\phi}a_{ij}\xi^{(i)}$
	in $\bigoplus_{j=1}^d\C(S)$
	to $\sum_{j=1}^d\sum_{i=1}^{c_\phi}a_{ij}\frA^j_{\xi^{(i)}}$.
	This shows that \eqref{eq: OK} is in fact an isomorphism.
\end{proof}

We may now prove the following.

\begin{lemma}\label{lem: DS}
	Let $\sE_\mu= \overline{\partial(C^0_\unif(S^{\Z^d})^G)}$.  Then we have
	\[
		\sE_\mu\cap \partial\cV=\{0\}.
	\]
\end{lemma}

\begin{proof}
	Let $f\in\cV$ such that $\partial f\in\sE_\mu$.
	By definition of the closure, there exists a sequence of functions $\{f_n\}_{n\in\N}$
	in $C^0_\unif(S^{\Z^d})^G$ such that $\displaystyle \lim_{n\rightarrow\infty}\partial f_n=\partial f$ in
	$Z^1_{L^2}(S^{\Z^d}_\mu)^G$.  This implies that for any $\La\in\sI$, we have
	\[
		\lim_{n\rightarrow\infty}\partial_\La f_n^\La
		=\lim_{n\rightarrow\infty}(\partial f_n)^\La=(\partial f)^\La=\partial_\La f^\La
	\]
	in
	$C^1(S^\La)$.
	Lemma \ref{lem: fsg}
	implies that $\displaystyle \lim_{n\rightarrow\infty}\bar f_n^\La=\bar f^\La$
	in $A^0(S^\La)$, where the bar denotes the image of a function in $A^0(S^\La)$.
	By taking the images of $\bar f_n$ and $\bar f$ with respect to the map $\delta$, we see that
	$\delta(\bar f_n)=0$ since $f_n$ is invariant as a function with respect to the action of $G$.
	Since $\displaystyle \delta(\bar f)^\La=\lim_{n\rightarrow\infty}\delta(\bar f_n)^\La=0$ 
	for any $\La\in\sI$, 
	we see that $\delta(\bar f)=0$.
	By Lemma \ref{lem: H}, the homomorphism 
	$\delta$ is injective on $\cV$.  Hence $f\in\cV$ and $\delta(\bar f)=0$ implies that
	$f=0$, which proves that $\partial f=0$ as desired.
\end{proof}
We may now prove Theorem \ref{thm: main}, for the case of the Euclidean lattice $(X,E)=(\Z^d,\E^d)$.

\begin{theorem}\label{thm: Euclidean}
	Theorem \ref{thm: main} is true for the Euclidean lattice $(X,E)=(\Z^d,\E^d)$.
\end{theorem}

\begin{proof}
	Let 
	\[
		\omega=(\omega^\La)\in\sC_\mu= Z^1_{L^2}(S^{\Z^d})^G=C^1_{L^2}(S^{\Z^d})\cap Z^1_\col(S^{\Z^d}_\mu)^G.
	\]
	We let $\Psi_n$ be the shift invariant uniform function of Definition \ref{def: psi}.
	Then by Proposition \ref{prop: conv0} and Proposition \ref{prop: conv1}, we see that
	$\partial\Psi_n$ converges weakly to $\omega_\psi\coloneqq\omega+\omega^\dagger$ in $\sC_\mu$.
	Since weak closure of a subspace of a Hilbert space coincides with its strong closure,
	this implies that 
	\[
		\omega_\psi=\omega+\omega^\dagger\in \sE_\mu=\overline{\partial(C^0_\unif(S^{\Z^d})^G)}.
	\]
	By Proposition \ref{prop: local}, the form $\omega^\dagger$ is a form in $\cC$.
	By Theorem \ref{thm: previous}, there exists $\omega'\in \cE=\partial(C^0_\unif(S^{\Z^d})^G)\subset\sE_\mu$ and
	$\omega^\ddagger\in\partial\cV$ such that
	\[
		\omega^\dagger = \omega'+\omega^\ddagger.
	\]
	Note that $(\omega_\psi-\omega')\in \sE_\mu$. 
	By Lemma \ref{lem: DS}, we see that $\omega=(\omega_\psi-\omega')+(-\omega^\ddagger)$
	is a direct sum,
	hence this gives a decomposition
	\[
		\omega=(\omega_\psi-\omega') +(-\omega^\ddagger)\in
		\sE_\mu\oplus
		\partial\cV
	\]	
	as desired.
\end{proof}

%
%
\subsection{Proof of the Main Theorem}\label{subsec: PMT}
%
%

We will now prove our main theorem, \cref{thm: main}
for a general finite range Euclidean lattice $(\Z^d,E)$.
Let the notations be as in \S\ref{subsec: MT}.
We let $\cX=(\Z^d,E)$ be a general finite range Euclidean lattice,
and for an integer $k>0$, let $\cX_k=(\Z^d,\E^d_k)$ be the Euclidean lattice with finite $k$-range, 
defined by 
\[
	\E^d_k\coloneqq\Bigl\{ (x,y)\in\Z^d\times\Z^d\mid 0<\sum_{j=1}^d |x_j-y_j|\leq k\Bigr\}.
\]
Since $G=\Z^d$ acts on $\cX$ by translation, we see that
$E\subset\E^d_k$ for $k$-sufficiently large.
We denote by $S^\cX$ and $S^{\cX_k}$ the respective configuration space with transition structures,
and by $C^0_\col(S^\cX_\mu)$, $C^0_\col(S^{\cX_k}_\mu)$, $C^1_\col(S^\cX_\mu)$,
$C^1_\col(S^{\cX_k}_\mu)$, etc.\ for
the respected spaces of functions and forms.
We remark that the space of co-local functions $C^0_\col$ and the space
of uniform functions $C^0_\unif$ are independent of the choice of the edges,
hence we have canonically 
\begin{align*}
	C^0_\col(S^\cX)&=C^0_\col(S^{\cX_k}), & C^0_\unif(S^\cX)&=C^0_\unif(S^{\cX_k}).
\end{align*}
Note that the inclusion $E\subset\E^d_k$
induces a natural projection
\begin{align}\label{eq: projection}
	C^1_\col(S^{\cX_k}_\mu)\rightarrow C^1_\col(S^\cX_\mu)
\end{align}
given through \eqref{eq: co-local inclusion} via the projections
\begin{align*}
	\prod_{e\in\E^d_k}C_\col(S^{\cX_k}_\mu)&\rightarrow\prod_{e\in\E^d}C_\col(S^\cX_\mu).
\end{align*}
This projection is compatible with the differentials
\begin{align*}
	\partial_{\cX}\colon C^0_\col(S^\cX_\mu)&\rightarrow C^1_\col(S^\cX_\mu),&
	\partial_{\cX_k}\colon C^0_\col(S^{\cX_k}_\mu)&\rightarrow C^1_\col(S^{\cX_k}_\mu)
\end{align*}
of \cref{prop: compatible}.  We first have the following.

\begin{lemma}\label{lem: co-local isom}
	Suppose $(S,\phi)$ is an interaction which is irreducibly quantified.
	Then there exists a canonical isomorphism $H^0_\col(S^\cX_\mu)\cong H^0_\col(S^{\cX_k}_\mu)$
	and 
	\begin{align*}
		Z^1_\col(S^\cX_\mu)&\cong Z^1_\col(S^{\cX_k}_\mu)
	\end{align*}
	which are compatible with the differentials $\partial_{\cX}$ and $\partial_{\cX_k}$. In particular, this induces a canonical isomorphism 
	\begin{align*}
		Z^1_\col(S^\cX_\mu)^G&\cong Z^1_\col(S^{\cX_k}_\mu)^G.
	\end{align*}
\end{lemma}

\begin{proof}
	Note that if $\La\in\sI$ is connected
	in $\cX$, then it is connected in $\cX_k$.  Hence $\La=(\La,E_\La)$ and $\La'=(\La,E'_\La)$
	for $E'=\E^d_k$
	are finite locales.  By definition, $\Ker\partial_\La$ and $\Ker\partial_{\La'}$ 
	coincide with the subspaces of 
	$C(S^\La)=C(S^{\La'})$ which are respectively
	constant on the connected components of $(S^\La,\Phi_{E_{\La}})$ and $(S^{\La'},\Phi_{E_{\La'}})$.
	Hence a priori, $\Ker\partial_{\La'}\subset\Ker\partial_{\La}$.
	Consider the map
	\[
		\bsxi_\La\colon S^\La\rightarrow\cM\coloneqq\Hom_\R(\C(S),\R)\cong\R^{c_\phi},
	\]
	given by mapping any $\e\in S^\La$ to the map $\xi\mapsto\xi_\La(\e)$ for any $\xi\in\C(S)$.
	Since $(S,\phi)$ is irreducibly quantified, the connected components of $(S^\La,\Phi_{E_{\La}})$ 
	coincide
	bijectively with the image $\bsxi_\La(S^\La)\subset \cM$ (see \cite{BKS20}*{Remark 2.33}).  
	Since $S^\La=S^{\La'}$ as sets,
	and since the map $\bsxi_\La$ is independent of the choice of the set of edges, we see that the connected
	components of $(S^\La,\Phi_{E_{\La}})$ and $(S^{\La'},\Phi_{E_{\La'}})$ coincide.  This shows that
	$\Ker\partial_\La=\Ker\partial_{\La'}$.  
	For any $\La\in\sI$, there exists $\La^*\in\sI$ connected in $\cX=(\Z^d,E)$
	such that $\La\subset \La^*$, passing to the projective limit, we see that
	$H^0_\col(S^\cX_\mu)\cong H^0_\col(S^{\cX_k}_\mu)$ as desired.	
	
	The projection \eqref{eq: projection} is compatible with the differentials,
	hence gives a commutative diagram
	\[
		\xymatrix{
			C^0_\col(S^{\cX_k}_\mu)\ar[r]^{\partial_{\cX_k}}\ar@{=}[d]&C^1_\col(S^{\cX_k}_\mu)\ar[d]\\
			C^0_\col(S^\cX_\mu)\ar[r]^{\partial_{\cX}}&C^1_\col(S^\cX_\mu),
		}
	\]
	which induces a commutative diagram	
	\begin{equation}\label{eq: commute}
		\xymatrix{
			C^0_\col(S^{\cX_k}_\mu)/H^0_\col(S^{\cX_k}_\mu)
			\ar[r]_-\cong^-{\partial_{\cX_k}}\ar@{=}[d]&Z^1_\col(S^{\cX_k}_\mu)\ar[d]\\
			C^0_\col(S^\cX_\mu)/H^0_\col(S^\cX_\mu)\ar[r]_-\cong^-{\partial_{\cX}}&Z^1_\col(S^\cX_\mu).
		}
	\end{equation}
	By \cref{prop: SES}, the horizontal differentials of \eqref{eq: commute}
	are isomorphisms.  This proves that the projection induces an isomorphism
	$Z^1_\col(S^{\cX_k}_\mu)\cong Z^1_\col(S^\cX_\mu)$. Since the projection \eqref{eq: projection}  is commutative with the action of the group $G$, it also induces an isomorphism $Z^1_\col(S^{\cX_k}_\mu)^G \cong Z^1_\col(S^\cX_\mu)^G$ as desired.
\end{proof}

\begin{proposition}\label{lem: L2 isom} 
	Suppose $(S,\phi)$ is an interaction which is irreducibly quantified.
	Let $\cX=(\Z^d,E)$ be any Euclidean lattice with finite range,
	 and let $\cX_k=(\Z^d,\E^d_k)$ be 
	the Euclidean lattice with finite $k$-range such that $E\subset\E^d_k$.
	Then we have a canonical isomorphism
	\begin{equation}\label{eq: L2 isom}
		Z^1_{L^2}(S^\cX_\mu)^G \cong Z^1_{L^2}(S^{\cX_k}_\mu)^G
	\end{equation}
	on the space of shift invariant closed $L^2$-forms, compatible with the differentials $\partial_\cX$
	and $\partial_{\cX_k}$,
	which preserves
	the topology induced from the 
	supremum norms on $Z^1_{L^2}(S^\cX_\mu)^G$ and $Z^1_{L^2}(S^{\cX_k}_\mu)^G$.
\end{proposition}

\begin{proof}
	Since $E\subset\E^d_k$, the natural isomorphism
	$Z^1_\col(S^{\cX_k}_\mu)^G\xrightarrow\cong Z^1_\col(S^\cX_\mu)^G$
	of \cref{lem: co-local isom}
	 induced from the projection
	\eqref{eq: projection} induces an injection 
	$Z^1_{L^2}(S^{\cX_k}_\mu)^G \hookrightarrow Z^1_{L^2}(S^{\cX}_\mu)^G$
	which is continuous for the supremum norm.
	On the other hand, suppose $\omega\in Z^1_{L^2}(S^\cX_\mu)^G$.
	Then by definition of closed $L^2$-forms, we have
	\begin{equation*}
		\omega=(\omega_e)\in Z^1_{\col}(S^\cX_\mu) \cap \prod_{e\in E}L^2(\mu).
	\end{equation*}
	By definition of $Z^1_{\col}(S^\cX_\mu)$, there exists a co-local function $f\in C_\col(S^\cX)$
	such that $\partial_{\cX} f=\omega$.
	Denote by $\omega'$ the image of $\omega$ in $Z^1_{\col}(S^{\cX_k}_\mu)^G$
	through the isomorphism 
	$
		Z^1_\col(S^\cX_\mu)^G \cong Z^1_\col(S^{\cX_k}_\mu)^G
	$
	obtained in \cref{lem: co-local isom}. 
	The compatibility of this isomorphism with the differential shows that we have $\omega'=\partial_{\cX_k} f$,
	hence 
	\[
		\omega'=(\nabla_{e'}f) \in Z^1_\col(S^{\cX_k}_\mu)\subset \prod_{e'\in\E^d_k}C_\col(S^{\cX_k}_\mu).
	\]
	Let $C\coloneqq\max\{d_{\cX}(o(e'),t(e'))\mid e'\in\E^d_k\}$,
	which is finite since $\cX$ and $\cX_k$ are both locally finite and has an action of $G=\Z^d$.
	We have
	$d_{\cX}(x,y)\leq C d_{\cX_k}(x,y)$
	 for any $x,y\in\Z^d$.  Hence for any $e'\in\E^d_k$, there 
	exists a path $\vec p=(e_1,\ldots,e_N)$ in $\cX$ such that $o(e')=o(e_1)$, $t(e')=t(e_N)$ and $N\leq C$.   
	If we let $\La=\{ o(e_1) \}\cup\{t(e_i) \mid i=1,\ldots,N\}$,
	then $(\La,E_\La)$ is a locale satisfying $\operatorname{diam}_\cX{\La}\leq N\leq C$.
	Noting that $\omega'_{e'}=\nabla_{e'}f=\nabla_{o(e')\arr t(e')}f$, the 
	Moving Particle Lemma (\cref{lem: MPL}) gives
	\[
		\dabs{\omega'_{e'}}_\mu
		=\dabs{\nabla_{o(e')\arr t(e')}f}_\mu
		\leq C_{M\!P}^{1/2} N\sup_{e\in E_\La}\dabs{\omega_e}_\mu\leq C''\dabs{\omega}_\sp,
	\]
	where $C''\coloneqq CC_{M\!P}^{1/2}$ is a constant which depends only on $\cX,\cX_k$ and $((S,\phi),\nu)$.
	In particular, we have
	\[
		\dabs{\omega'}_\sp \leq  C''\dabs{\omega}_\sp.
	\]
	This shows that the canonical isomorphism
	of \cref{lem: co-local isom} induces an injection
	$Z^1_{L^2}(S^\cX_\mu)^G \hookrightarrow Z^1_{L^2}(S^{\cX_k}_\mu)^G$
	which is again continuous for the topology induced from the supremum norm.
	This proves our assertion.
\end{proof}

We are now ready to prove \cref{thm: main}.

\begin{proof}[Proof of \cref{thm: main}]
	Let $\cX''\coloneqq(\Z^d,\E^d)$ be the Euclidean lattice with nearest neighbor.
	Let $k$ be sufficiently large so that $E\subset\E^d_k$.
	By applying \cref{lem: L2 isom} 
	first to $\cX=(\Z^d,E)$ and then to $\cX''=(\Z^d,\E^d)$,
	 noting that $\E^d\subset\E^d_k$ and that the 
	isomorphisms are compatible with the differential and the action of $\G$,
	we have a commutative diagram
	\[
		\xymatrix{
			C^0_\unif(S^{\cX''}_\mu)^\G\ar[r]^{\partial_{\cX''}}\ar[d]^\cong&
			Z^1_{L^2}(S^{\cX''}_\mu)^\G\ar[d]^\cong\\
			C^0_\unif(S^{\cX}_\mu)^\G\ar[r]^{\partial_{\cX}}&Z^1_{L^2}(S^{\cX}_\mu)^\G.
		}
	\]
	The isomorphism of \cref{lem: L2 isom} preserves the topology, 
	hence 
	if we let 
	\begin{align*}
		\sC''_\mu&\coloneqq Z^1_{L^2}(S^{\cX''}_\mu)^\G,&
		\sE''_\mu&\coloneqq \overline{\partial_{\cX''}(C^0_\unif(S^{\cX_k})^G)},
	\end{align*}
	then 
	we have $\sC_\mu\cong\sC''_\mu$ and $\sE_\mu\cong\sE''_\mu$
	through \eqref{eq: L2 isom}.
	Furthermore, by definition, we see that \eqref{eq: L2 isom}
	maps $\partial_{\cX}\cV$  bijectively to  $\partial_{\cX''}\cV$.
	Hence our assertion $\sC_\mu=\sE_\mu\oplus\partial_{\cX}\cV$ follows from 
	\cref{thm: Euclidean}, which asserts that $\sC_\mu''=\sE_\mu''\oplus\partial_{\cX''}\cV$ as desired.
\end{proof}


\end{document}